\documentclass[aos,preprint]{imsart}

\RequirePackage[OT1]{fontenc}
\RequirePackage{amsthm,amsmath}
\RequirePackage[numbers]{natbib}
\RequirePackage[colorlinks,citecolor=blue,urlcolor=blue]{hyperref}

\arxiv{arXiv:1712.00771}

\usepackage{xr}
\startlocaldefs
\numberwithin{equation}{section}
\theoremstyle{plain}
\newtheorem{thm}{Theorem}[section]
\newtheorem{lem}[thm]{Lemma}
\newtheorem{prop}[thm]{Proposition}
\newtheorem{cor}[thm]{Corollary}

\theoremstyle{definition}

\newtheorem{exmp}{Example}[section]
\newtheorem{rmk}{Remark}[section]
\endlocaldefs

\usepackage{amssymb}
\usepackage{amsthm}
\usepackage{mathrsfs, color}
\usepackage{graphicx}
\usepackage{graphics}
\usepackage{multirow}
\usepackage{amsmath}
\usepackage{enumitem}
\usepackage[normalem]{ulem}
\usepackage{bm}
\usepackage{natbib}

\newcommand{\vone}[0]{\mathbf 1}

\newcommand{\Cov}[0]{\mathrm{Cov}}
\newcommand{\Var}[0]{\mathrm{Var}}

\newcommand{\diag}[0]{\mathrm{diag}}

\newcommand{\calD}[0]{\mathcal{D}}

\newcommand{\calI}[0]{\mathcal{I}}

\newcommand{\calR}[0]{\mathcal{R}}
\newcommand{\calS}[0]{\mathcal{S}}

\newcommand{\calX}[0]{\mathcal{X}}
\newcommand{\calA}[0]{\mathcal{A}}

\newcommand{\E}[0]{\mathbb{E}}

\newcommand{\R}[0]{\mathbb{R}}

\newcommand{\sign}[0]{\text{sign}}

\newcommand{\Bern}[0]{\mathsf{Ber}}
\newcommand{\Bin}[0]{\mathsf{Bin}}

\newcommand{\Prob}[0]{\mathbb{P}}

\renewcommand{\tilde}{\widetilde}
\renewcommand{\hat}{\widehat}
\renewcommand{\le}{\leqslant}
\renewcommand{\ge}{\geqslant}

\newcommand{\vertiii}[1]{{\left\vert\kern-0.25ex\left\vert\kern-0.25ex\left\vert #1 
    \right\vert\kern-0.25ex\right\vert\kern-0.25ex\right\vert}}

\begin{document}

\begin{frontmatter}
\title{Randomized incomplete $U$-statistics in high dimensions\thanksref{T1}}
\runtitle{Randomized incomplete $U$-statistics in high dimensions}
\thankstext{T1}{First arXiv version: December 3, 2017. This version: \today.}

\begin{aug}
\author{\fnms{Xiaohui} \snm{Chen}\thanksref{t1,m1}\ead[label=e1]{xhchen@illinois.edu}},
\and
\author{\fnms{Kengo} \snm{Kato}\thanksref{m2}\ead[label=e2]{kk976@cornell.edu}}

\thankstext{t1}{Supported in part by NSF DMS-1404891, NSF CAREER Award DMS-1752614, and UIUC Research Board Awards (RB17092,  RB18099).}
\thankstext{T1}{This work is completed in part with the high-performance computing resource provided by the Illinois Campus Cluster Program at UIUC.}
\runauthor{Chen and Kato}

\affiliation{University of Illinois at Urbana-Champaign\thanksmark{m1} and Cornell University\thanksmark{m2}}

\address{Department of Statistics\\
University of Illinois at Urbana-Champaign\\
725 S. Wright Street\\
Champaign, IL 61874 USA\\
\printead{e1}}

\address{Department of Statistical Science\\
Cornell University\\
1194 Comstock Hall \\
Ithaca, NY 14853 USA\\
\printead{e2}}
\end{aug}

\begin{abstract}
This paper studies inference for the mean vector of a high-dimensional $U$-statistic. In the era of Big Data, the dimension $d$ of the $U$-statistic and the sample size $n$ of the observations tend to be both large, and the computation of the $U$-statistic is prohibitively demanding. Data-dependent inferential procedures such as the empirical bootstrap for  $U$-statistics is even more computationally expensive. To overcome such computational bottleneck, incomplete $U$-statistics obtained by sampling fewer terms of the $U$-statistic are attractive alternatives. In this paper, we introduce randomized incomplete $U$-statistics with sparse weights whose computational cost can be made independent of the order of the $U$-statistic. We derive  non-asymptotic Gaussian approximation error bounds for the randomized incomplete $U$-statistics in high dimensions, namely in cases where the dimension $d$ is possibly much larger than the sample size $n$,  for both non-degenerate and degenerate kernels. In addition, we propose generic bootstrap methods for the incomplete $U$-statistics that are computationally much less-demanding than existing bootstrap methods, and establish finite sample validity of the proposed bootstrap methods. Our methods are illustrated on the application to nonparametric testing for the pairwise independence of a high-dimensional random vector under weaker assumptions than those appearing in the literature. 
\end{abstract}



\end{frontmatter}

\section{Introduction}

Let $X_1,\dots,X_n$ be independent and identically distributed (i.i.d.) random variables taking values in a measurable space $(S,\calS) $ with common distribution $P$. Let $r \ge 2$ and $d \ge 1$ be given positive integers, and let $h = (h_{1},\dots,h_{d})^{T}: S^r \to \R^d$ be a fixed and jointly measurable function that is symmetric in its arguments, i.e., $h(x_{1},\dots,x_{r}) = h(x_{i_1},\dots,x_{i_r})$ for every permutation $i_1,\dots,i_r$ of $1,\dots,r$. Suppose that $\E[|h_{j}(X_{1},\dots,X_{r})| ] < \infty$ for all $j=1,\dots,d$, and consider inference on the mean vector $\theta = (\theta_{1},\dots,\theta_{d})^{T} = \E[h(X_{1},\dots,X_{r})]$. 
To this end, a commonly used statistic is the $U$-statistic with kernel $h$, i.e., the sample average of $h(X_{i_{1}},\dots,X_{i_{r}})$ over all distinct $r$-tuples $(i_{1},\dots,i_{r})$ from $\{1,\dots, n \}$
\begin{equation}
\label{eqn:complete_ustat}
U_n := U_n^{(r)}(h) :=  {1 \over |I_{n,r}|} \sum_{(i_1,\dots,i_r) \in I_{n,r}} h(X_{i_1},\dots,X_{i_r}),
\end{equation}
where $I_{n,r} = \{(i_1,\dots,i_r): 1 \le i_1 < \dots < i_r \le n \}$ and $|I_{n,r}| =n!/\{ r!(n-r)! \}$ denotes the cardinality of  $I_{n,r}$.

$U$-statistics are an important and general class of statistics, and applied in a wide variety of statistical problems; we refer to \cite{lee1990} as an excellent monograph on $U$-statistics. 
For univariate $U$-statistics ($d=1$), the asymptotic distributions are derived in the seminal  paper \cite{hoeffding1948} for the non-degenerate case and in \cite{rubinvitale1980} for the degenerate case.
There is also a large literature on bootstrap methods for univariate $U$-statistics \citep{bickelfreedman1981,bretagnolle1983,arconesgine1992,huskovajanssen1993,huskovajanssen1993b,dehlingmikosch1994,wangjing2004}.
A more recent interest lies in the high-dimensional case where $d$ is much larger than $n$. \cite{chen2017a} develops Gaussian and bootstrap approximations for non-degenerate $U$-statistics of order two in high dimensions, which extends the work \cite{cck2013,cck2017_AoP} from sample averages to $U$-statistics; see also \cite{gucaoningliu2015}.

However, a major obstacle of inference using the complete $U$-statistic (\ref{eqn:complete_ustat}) is its computational intractability. 
Namely, the computation of the complete $U$-statistic (\ref{eqn:complete_ustat}) requires $O(n^{r}d)$ operations, and its computational cost can be prohibitively demanding even when $n$ and $d$ are moderately large, especially when the order of the $U$-statistic $r \ge 3$.
For instance, the computation of a complete $U$-statistic with order $3$ and dimension $d = 5000$ when the sample size is $n=1000$  requires $\binom{n}{3} \times d \approx 0.8 \cdot 10^{12}$ (0.8 trillion) operations. 
In addition, the naive application of the empirical bootstrap for the $U$-statistic (\ref{eqn:complete_ustat}) requires even more operations, namely, $O(Bn^{r}d$) operations, where $B$ is the number of bootstrap repetitions. 

This motivates us to study inference using \textit{randomized incomplete $U$-statistics} with sparse weights instead of complete $U$-statistics. Specifically, we consider the Bernoulli sampling and sampling with replacement to construct random weights in Section \ref{sec:randomized_incomplete_ustat}. For a pre-specified \textit{computational budget parameter} $N \le |I_{n,r}|$, these sampling schemes randomly choose (on average) $N$ indices from $I_{n,r}$, and the resulting incomplete $U$-statistics $U_{n,N}'$ are defined as the sample averages of $h(X_{i_1},\dots,X_{i_r})$ taken over the subset of chosen indices $(i_{1},\dots,i_{r})$. Hence the computational cost of the incomplete $U$-statistics is reduced to $O(Nd)$, which can be much smaller than $n^{r}d$ as long as $N \ll n^{r}$ and can be made independent of the order of the $U$-statistic provided that $N$ does not depend on $r$.

The goal of this paper is to develop computationally scalable and statistically correct inferential methods for the incomplete $U$-statistics with high-dimensional kernels and massive data, where $d$ is possibly much larger than $n$ but $n$ can be also large. 
Specifically, we study distributional approximations to the randomized incomplete $U$-statistics in high dimensions. 
Our first main contribution is to derive Gaussian approximation error bounds for the incomplete $U$-statistics on the hyperrectangles in $\R^{d}$ for both non-degenerate and degenerate kernels. In Section \ref{sec:gaussian_approximations}, we show that the derived Gaussian approximation results display an interesting computational and statistical trade-off for non-degenerate kernels (see Remark \ref{rmk:computation_statistics_tradeoff}), and reveal a fundamental difference between complete and randomized incomplete $U$-statistics for degenerate kernels (see Remark \ref{rmk:gauss_appro_degenerate}). The mathematical insight of introducing the random weights is to create the (conditional) independence for the terms in the $U$-statistic sum in order to obtain a Gaussian limit. The Gaussian approximation results are, however, often not directly applicable since the covariance matrices of the approximating Gaussian distributions depend on the underlying  distribution $P$ that is unknown in practice. Our second contribution is to propose fully data-dependent bootstrap methods for incomplete $U$-statistics that are computationally (much) less demanding than existing bootstrap methods for $U$-statistics \cite{arconesgine1992,chen2017a,chenkato2017a}. Specifically, we introduce generic bootstraps for incomplete $U$-statistics in Section \ref{sec:generic_bootstrap}. Our generic bootstrap constructions are flexible enough to cover both non-degenerate and degenerate kernels, and meanwhile they take the computational concern into account for estimating the associated (and unobserved) H\'ajek projection in the non-degenerate case. In particular, we propose two concrete estimation procedures for the H\'ajek projection: one is a deterministic construction based on the divide and conquer algorithm (Section \ref{subsec:divide_and_conquer}), and another is a random construction based on a second randomization independent of everything else (Section \ref{subsec:incomplete_ustat}). For both constructions, the overall computational complexity of the bootstrap methods can be made independent of the $U$-statistic order $r$. 

As a leading example to illustrate the usefulness of the inferential methods developed in the present paper, we consider testing for the pairwise independence of a high-dimensional random vector $X = (X^{(1)},\dots,X^{(p)})^T$, i.e., testing for the hypothesis that 
\begin{equation}
\label{eqn:independence_test}
H_{0}: X^{(1)},\dots,X^{(p)} \; \text{are pairwise independent}.
\end{equation}
Let $X_1,\dots,X_n$ be i.i.d. copies of $X$. Several dependence measures are proposed in the literature, including: Kendall's $\tau$, Spearman's $\rho$, Hoeffding's $D$ \cite{Hoeffding1948b_AoMS}, Bergsma and Dassios' $t^*$ \cite{BergsmaDassios2014_Bernoulli}, and the distance covariance \cite{SzekelyRizzoBakirov2007_AoS,YaoZhangShao2017_JRSSB}, all of which can be estimated by $U$-statistics. So various nonparametric tests for $H_{0}$ can be constructed based on those $U$-statistics. To compute the test statistics, we have to compute $U$-statistics with dimension $d=p(p-1)/2$, which corresponds to the number of upper triangular entries in the $p \times p$ dependence matrix and can be quite large. In addition, the orders of the $U$-statistics are at least $3$ (except for Kendall's $\tau$ which is of order $2$). So the computation of the test statistics is prohibitively demanding, not to mention the empirical bootstrap or subsampling for those $U$-statistics. It should be noted that there are efficient algorithms to reduce the computational costs for computing some of those $U$-statistics \citep[cf.][Section 6.1]{leungdrton2017}, but such computational simplifications are case-by-case and not generically applicable, and more importantly they do not yield computationally tractable methods to approximate or estimate the sampling distributions of the $U$-statistics. The Gaussian and bootstrap approximation theorems developed in the present paper can be  applicable to calibrating critical values for  test statistics based upon incomplete versions of those $U$-statistics.
 Detailed comparisons and discussions of nonparametric pairwise independence test statistics are presented in Section \ref{sec:numerics}.  In addition to pairwise independence testing, values of the dependence measures are also interesting \textit{per se} in some applications. For instance, Spearman's $\rho$ is related to the copula correlation if the marginal distributions are continuous \cite[Chapter 8]{EmbrechtsLindskogMcneil2003}, and our bootstrap methods can be used to construct simultaneous confidence intervals for the copula correlations uniformly over many pairs of variables. 

To verify the finite sample performance of the proposed bootstrap methods for randomized incomplete $U$-statistics, we conduct simulation experiments in Section \ref{sec:numerics} on the leading example for nonparametric testing for the pairwise independence hypothesis in (\ref{eqn:independence_test}). Specifically, we consider to approximate the null distributions of the incomplete versions of the (leading term of) Spearman $\rho$ and Bergsma-Dassios' $t^{*}$ test statistics, and examine the cases where $n=300,500,1000$ and $p=30,50,100$ (and hence $d=p(p-1)/2=435,1225,4950$). Statistically, we observe that the Gaussian approximation of the test statistics is quite accurate and the empirical rejection probability of the null hypothesis with the critical values calibrated by our bootstrap methods is very close to the nominal size for (almost) all setups. Computationally, we find that the (log-)running time for our bootstrap methods scales linearly with the (log-)sample size, and in addition, the slope coefficient matches very well with the computational complexity of the bootstrap methods. Therefore, the simulation results demonstrate a promising agreement between the empirical evidences and our theoretical analysis. 

\subsection{Existing literature}
Incomplete $U$-statistics are first considered in \cite{blom1976}, and the asymptotic distributions of incomplete $U$-statistics (for fixed $d$) are derived in \cite{brownkildea1978} and \cite{Janson1984_PTRF}; see also Section 4.3 in \cite{lee1990} for a review on incomplete $U$-statistics. Closely related to the present paper is \cite{Janson1984_PTRF}, which establishes the asymptotic properties of univariate incomplete $U$-statistics based on sampling with and without replacement and Bernoulli sampling. To the best of our knowledge, the present paper is the first paper that establishes approximation theorems for the distributions of randomized incomplete $U$-statistics in high dimensions. See also Remark \ref{rmk:compare_Janson} for more detailed comparisons with \cite{Janson1984_PTRF}. 
Incomplete $U$-statistics can be viewed as a special case of weighted $U$-statistics, and there is a large literature on limit theorems for weighted $U$-statistics; see \cite{shapirohubert1979,oneilredner1993,major1994,rifiutzet2000,hsingwu2004,hanqian2016} and references therein. 
These references focus on the univariate case and do not cover the high-dimensional case. 
There are few references that study data-dependent inferential procedures for incomplete $U$-statistics that take computational considerations into account. An exception is \cite{BertailTressou2006_Biometrics}, which proposes several inferential methods for univariate (generalized) incomplete $U$-statistics, but do not develop formal asymptotic justifications for these methods. 
It is also interesting to note that incomplete $U$-statistics have gained renewed interests in the recent statistics and machine learning literatures \cite{clemencon2016, mentchhooker2016}, although the focuses of these references are substantially different from ours. 

From a technical point of view, this paper builds on recent development of  Gaussian and  bootstrap approximation theorems for averages of independent high-dimensional random vectors \cite{cck2013,cck2017_AoP} and for high-dimensional $U$-statistics of order two \cite{chen2017a}. 
Importantly, however, developing Gaussian approximations for the randomized incomplete $U$-statistics in high dimensions requires a novel proof-strategy that combines iterative conditioning arguments and  applications of Berry-Esseen type bounds, and extends some of results in \cite{chen2017a} to cover general order incomplete $U$-statistics. In addition, these references do not consider bootstrap methods for incomplete $U$-statistics that take computational considerations into account.

\subsection{Organization}
The rest of the paper is organized as follows. In Section \ref{sec:randomized_incomplete_ustat}, we introduce randomized incomplete $U$-statistics with sparse weights generated from the Bernoulli sampling and sampling with replacement. In Section \ref{sec:gaussian_approximations}, we derive  non-asymptotic Gaussian approximation error bounds for the randomized incomplete $U$-statistics in high dimensions for both non-degenerate and degenerate kernels. In Section \ref{sec:bootstrap}, we first propose generic bootstrap methods for the incomplete $U$-statistics and then incorporate the computational budget constraint by two concrete estimates of the H\'ajek projection: one deterministic estimate by the divide and conquer, and one randomized estimate by incomplete $U$-statistics of a lower order. Simulation examples are provided in Section \ref{sec:numerics} and in the Supplementary Material (SM). 
All the technical proofs are gathered in Appendix \ref{sec:proofs} in the SM. We conclude the paper in Section \ref{sec:discussions} with a brief discussion on some extensions. 

\subsection{Notation}

For a hyperrectangle $R=\prod_{j=1}^{d} [a_{j},b_{j}]$ in $\R^{d}$, a constant $c > 0$, and a vector $y = (y_{1},\dots,y_{d})^{T} \in \R^{d}$, we use the notation $[cR+y] = \prod_{j=1}^{d}[ca_{j}+y_{j},cb_{j}+y_{j}]$. 
For vectors $y = (y_{1},\dots,y_{d})^{T}, z=(z_{1},\dots,z_{d})^{T} \in \R^{d}$, the notation $y \le z$ means that $y_{j} \le z_{j}$ for all $j=1,\dots,d$. 
For $a,b \in \R$, let $a \vee b = \max \{ a,b \}$ and $a \wedge b = \min \{ a,b \}$. For a finite set $J$, $|J|$ denotes the cardinality of $J$. Let $| \cdot |_{\infty}$ denote the max-norm for vectors and matrices, i.e., for a matrix $A = (a_{ij})$, $| A |_{\infty} = \max_{i,j}|a_{ij}|$. ``Constants" refer to finite, positive, and non-random
numbers.

For $0 < \beta < \infty$, let $\psi_{\beta}$ be the function on $[0,\infty)$ defined by $\psi_{\beta} (x) = e^{x^{\beta}}-1$, and for a real-valued random variable $\xi$,  define $\| \xi \|_{\psi_\beta}=\inf \{ C>0: \E[ \psi_{\beta}( | \xi | /C)] \leq 1\}$.
For $\beta \in [1,\infty)$, $\|\cdot\|_{\psi_{\beta}}$ is an Orlicz norm, while for $\beta \in (0,1)$, $\| \cdot \|_{\psi_{\beta}}$ is not a norm but a quasi-norm, i.e., there exists a constant $C_{\beta}$ depending only on $\beta$ such that $\| \xi_{1} + \xi_{2} \|_{\psi_{\beta}} \leq C_{\beta} ( \| \xi_{1} \|_{\psi_{\beta}} + \| \xi_{2} \|_{\psi_{\beta}})$. (Indeed, there is a norm equivalent to $\| \cdot \|_{\psi_{\beta}}$ obtained by linearizing $\psi_{\beta}$ in a neighborhood of the origin; cf. Lemma \ref{lem: quasi-norm} in the SM.) 

For a generic random variable $Y$, let $\Prob_{\mid Y}(\cdot)$ and $\E_{\mid Y}[\cdot]$ denote the conditional probability and expectation given $Y$, respectively.  
For a given probability space $(\calX,\calA,Q)$ and a measurable function  $f$ on $\calX$, we use the notation $Qf = \int f dQ$ whenever the latter integral is well-defined. 
For a jointly measurable symmetric function $f$ on $S^{r}$ and $k=1,\dots,r$, let $P^{r-k}f$ denote the function on $S^{k}$ defined by 
\[P^{r-k}f(x_{1},\dots,x_{k}) = \int \cdots \int f(x_{1},\dots,x_{k},x_{k+1},\dots,x_{r}) dP(x_{k+1}) \cdots dP(x_{r})\]
whenever the integral exists and is finite for every $(x_{1},\dots,x_{k}) \in S^{k}$. 
For given $1 \le k \le \ell \le n$, we use the notation $X_{k}^{\ell} = (X_{k},\dots,X_{\ell})$. 
Throughout the paper, we assume that $n \ge 4 \vee r$ and $d \ge 3$.

\section{Randomized incomplete $U$-statistics}
\label{sec:randomized_incomplete_ustat}

In this paper, to construct sparsely weighted $U$-statistics, we shall use random sparse  weights. 
For $\iota = (i_{1},\dots,i_{r}) \in I_{n,r}$, let us write $X_{\iota} = (X_{i_{1}},\dots,X_{i_{r}})$, and observe that the complete $U$-statistic (\ref{eqn:complete_ustat}) can be written as 
\[
U_{n} = \frac{1}{|I_{n,r}|} \sum_{\iota \in I_{n,r}} h(X_{\iota}). 
\]
Now, let $N := N_{n}$ be an integer such that $0 < N \le |I_{n,r}|$, and let $p_{n}=N/|I_{n,r}|$. 
Instead of taking the average over all possible $\iota$ in $I_{n,r}$, we will take the average over a subset of about $N$ indices chosen randomly from $I_{n,r}$. 
In the present paper, we study Bernoulli sampling and sampling with replacement. 

\subsection{Bernoulli sampling}
\label{subsec:bernoulli_sampling}
Generate i.i.d. $\Bern (p_{n})$ random variables $\{ Z_{\iota} : \iota \in I_{n,r} \}$ with success probability $p_{n}$, i.e., $Z_{\iota}, \iota \in I_{n,r}$ are i.i.d. with  $\Prob(Z_{\iota} = 1) = 1 - \Prob(Z_{\iota} = 0) = p_n$. Consider the following weighted $U$-statistic with random weights
\begin{equation}
\label{eqn:incomplete_ustat}
U_{n,N}' = \frac{1}{\hat{N}}\sum_{\iota \in I_{n,r}} Z_{\iota} h(X_{\iota}),
\end{equation}
where $\hat{N} = \sum_{\iota \in I_{n,r}} Z_{\iota}$ is the number of non-zero weights. We call $U_{n,N}'$ the randomized incomplete $U$-statistic based on the Bernoulli sampling.
The variable $\hat{N}$ follows $\Bin (|I_{n,r}|, p_{n})$, the binomial distribution with parameters $(|I_{n,r}|,p_{n})$. Hence $\E[\hat{N}] = |I_{n,r}| p_{n} = N$ and the computation of the incomplete $U$-statistic (\ref{eqn:incomplete_ustat}) only requires $O(Nd)$ operations on average. 
In addition, by Bernstein's inequality (cf. Lemma 2.2.9 in \cite{vandervaartwellner1996}),
\begin{equation}
\label{eqn:concentration}
\Prob \left(|\hat{N}/N - 1| > \sqrt{2t/N} + 2t/(3N)  \right ) \le 2e^{-t}
\end{equation}
for every $t > 0$, and hence $\hat{N}$ concentrates around its mean $N$. 
Therefore, we can view $N$ as a {\it computational budget parameter} and $p_n$ as a {\it sparsity design parameter} for the incomplete $U$-statistic.

The reader may wonder that generating $|I_{n,r}| \approx n^{r}$ Bernoulli random variables is computationally demanding, but there is no need to do so. In fact,  we can equivalently compute the randomized incomplete $U$-statistic in (\ref{eqn:incomplete_ustat}) as follows.
\begin{enumerate}
\item[1.] Generate $\hat{N} \sim \Bin (|I_{n,r}|,p_{n})$.
\item[2.] Choose indices $\iota_{1},\dots,\iota_{\hat{N}}$ randomly without replacement from $I_{n,r}$. 
\item[3.] Compute $U_{n,N}' = \hat{N}^{-1}\sum_{j=1}^{\hat{N}} h(X_{\iota_{j}})$.
\end{enumerate}
In fact, define $Z_{\iota} = 1$ if $\iota$ is one of $\iota_{1},\dots,\iota_{\hat{N}}$, and $Z_{\iota} = 0$ otherwise; then, it is not difficult to see that $\{ Z_{\iota} : \iota \in I_{n,r} \}$ are i.i.d. $\Bern (p_{n})$ random variables. So, we can think of the Bernoulli sampling as a sampling without replacement with a random sample size.

\begin{rmk}[Comments on the random normalization]
\label{rmk:random_normalization_bernoulli_sampling}
Interestingly, changing the normalization  in (\ref{eqn:incomplete_ustat})  \textit{does} affect  approximating distributions to the resulting incomplete $U$-statistic. Namely, if we change $\hat{N}$ to $N$ in (\ref{eqn:incomplete_ustat}), i.e., $\breve{U}_{n,N}' = N^{-1} \sum_{\iota \in I_{n,r}} Z_{\iota} h(X_{\iota})$, then we have  different approximating distributions unless $\theta=0$. In general, changing $\hat{N}$ to $N$ in (\ref{eqn:incomplete_ustat}) results in the approximating Gaussian distributions with larger covariance matrices, and hence it is recommended to use $U_{n,N}'$ rather than $\breve{U}_{n,N}'$.
See also Remark \ref{rem:deterministic_normalization} ahead. 
\end{rmk}

\subsection{Sampling with replacement}
\label{subsec:sampling_with_replacement}


Conditionally on $X_{1}^{n} =(X_{1},\dots,X_{n})$, let $X_{\iota_{j}}^*, j = 1,\dots,N$ be  i.i.d. draws from the empirical distribution $|I_{n,r}|^{-1} \sum_{\iota \in I_{n,r}} \delta_{X_{\iota}}$ ($\delta_{X_{\iota}}$ denotes the point mass at $X_{\iota}$). Let  
\begin{equation}
\label{eqn:sampling_with_replacement_1}
U_{n,N}' = {1 \over N} \sum_{j=1}^N h(X_{\iota_j}^*)
\end{equation}
be the incomplete $U$-statistic obtained by sampling with replacement. We call $U_{n,N}'$ the randomized incomplete $U$-statistic based on sampling with replacement. Observe that $U_{n,N}'$ in (\ref{eqn:sampling_with_replacement_1}) can be efficiently computed by sampling $r$ distinct terms from $\{X_{1},\dots,X_{n}\}$ independently for $N$ times. The statistic $U_{n,N}'$ can be written as a weighted $U$-statistic. Indeed, for each $\iota \in I_{n,r}$, let $Z_{\iota}$ denote the number of times that $X_{\iota}$ is redrawn in the sample $\{ X_{\iota_{1}}^{*},\dots,X_{\iota_{N}}^{*} \}$. Then the vector $Z=(Z_{\iota})_{\iota \in I_{n,r}}$ (ordered in an arbitrary way) follows a multinomial distribution with parameters $N$ and probabilities $1/|I_{n,r}|,\dots,1/|I_{n,r}|$ independent of $X_{1}^{n}$, and $U_{n,N}'$ can be written as 
\begin{equation}
\label{eqn:sampling_with_replacement_2}
U_{n,N}' = \frac{1}{N} \sum_{\iota \in I_{n,r}} Z_{\iota} h(X_{\iota}).
\end{equation}
Hence we can think of $U_{n,N}'$ as a statistic of $X_{1},\dots,X_{n}$ and $Z_{\iota}, \iota \in I_{n,r}$, but we will use both representations (\ref{eqn:sampling_with_replacement_1}) and (\ref{eqn:sampling_with_replacement_2}) interchangeably in the subsequent analysis.


\begin{rmk}
All the theoretical results presented below apply to incomplete $U$-statistics based on either the Bernoulli sampling or sampling with replacement. Both sampling schemes will be covered in a unified way. 
\end{rmk}

\section{Gaussian approximations}
\label{sec:gaussian_approximations}

In this section, we will derive Gaussian approximation results for the incomplete $U$-statistics (\ref{eqn:incomplete_ustat}) and (\ref{eqn:sampling_with_replacement_1}) on the hyperrectangles in $\R^{d}$.
Let $\calR$ denote the class of (closed) hyperrectangles in $\R^{d}$, i.e., $\calR$ consists sets of the form $\prod_{j=1}^{d} [a_{j},b_{j}]$ where $-\infty \le a_{j} \le b_{j} \le \infty$ for $j=1,\dots,d$ with the convention that $[a_{j},b_{j}] = (-\infty,b_{j}]$ for $a_{j}=-\infty$ and $[a_{j},b_{j}] = [a_{j},\infty)$ for $b_{j}=\infty$. 
For the expository purpose, we mainly focus on the non-degenerate case where $\min_{1 \le j \le d} \Var (\E[ h_{j}(X_{1},\dots,X_{r}) \mid X_{1}])$ is bounded away from zero in the following discussion. However, our Gaussian approximation results also cover the degenerate case (cf. Theorem \ref{thm:gaussian_approx_bern_random_design_degenerate}). The intuition behind and the proof sketch for the Gaussian approximation results are given in Section \ref{subsec:proof_sketch_gauss_approx_non-degenerate} in the SM.

To state the formal Gaussian approximation results, we assume the following conditions. Let $\underline{\sigma} > 0$ and $D_n \ge 1$ be given constants, and define $g := (g_{1},\dots,g_{d})^{T} :=  P^{r-1} h$. Suppose that 
\begin{enumerate}
\item[(C1)] $P^{r}|h_{j}|^{2+k} \le D_n^k$ for all $j = 1,\dots,d$ and $k=1,2$.
\item[(C2)] $\|h_j(X_1^r)\|_{\psi_1} \le D_n$ for all $j = 1,\dots,d$.
\end{enumerate}
In addition, suppose that either one of the following conditions holds:
\begin{enumerate}
\item[(C3-ND)] $P(g_{j}-\theta_{j})^{2} \ge \underline{\sigma}^{2}$ for all $j = 1,\dots,d$.
\item[(C3-D)] $P^{r} (h_{j} - \theta_{j})^{2} \ge \underline{\sigma}^{2}$ for all $j=1,\dots,d$. 
\end{enumerate}

Conditions (C1) and (C2) are adapted from \cite{cck2017_AoP} and \cite{chen2017a}. 
Condition (C2) assumes the kernel $h$ to be sub-exponential, which in particular covers bounded kernels. 
In principle, it is possible to extend our analysis under milder moment conditions on the kernel $h$, but this would result in more involved error bounds. For the sake of clear presentation, we mainly work with Condition (C2) and point out the differences when the kernel satisfies a polynomial moment condition in Remark \ref{rmk:relaxation_moment_condition}. 
By Jensen's inequality, Conditions (C1) and (C2) imply that $P | g_{j} |^{2+k} \le D_{n}^{k}$ for all $j$ and for $k=1,2$, and $\| g_{j}(X_{1}) \|_{\psi_{1}} \le D_{n}$ for all $j$. Here we allow the exponential moment bound $D_{n}$ to depend on $n$ since the distribution $P$ may depend on $n$ in the high-dimensional setting. In addition, Condition (C1) implies that $P^{r} h_{j}^{2} \le 1 + P|h_{j}|^{3} \le 1+D_{n}$ for all $j$. 
Condition (C3-ND) implies that the kernel $h$ is non-degenerate. 
In the degenerate case, we will require Condition (C3-D) to derive Gaussian approximations. 


In all what follows, we assume that 
\[
p_{n}  = N/|I_{n,r}| \le 1/2
\]
without further mentioning.
The value $1/2$ has no special meaning; we can allow $p_{n} \le c$ for any constant $c \in (0,1)$, and in that case, the constants appearing in the following theorems depend in addition on $c$. 
Since we are using randomization for the purpose of computational reduction, we are mainly interested in the case where $N \ll |I_{n,r}|$, and the assumption that $p_{n}$ is bounded away from 1 is immaterial. 

The following theorem derives bounds on the Gaussian approximation to the randomized  incomplete $U$-statistics on the hyperrectangles in the case where the kernel $h$ is non-degenerate. Recall that $\alpha_{n}=n/N, p_n = N / |I_{n,r}|, \theta = P^{r} h = Pg, \Gamma_{g} = P(g-\theta) (g-\theta)^{T}$, and $\Gamma_{h} = P^{r} (h-\theta)(h-\theta)^{T}$.

\begin{thm}[Gaussian approximation under non-degeneracy]
\label{thm:gaussian_approx_bern_random_design}
Suppose that Conditions (C1), (C2), and (C3-ND) hold.
Then there exists a constant $C$ depending only on $\underline{\sigma}$ and $r$ such that 
\begin{equation}
\label{eqn:gaussian_approx_bern_random_design}
\begin{split}
& \sup_{R \in \calR} \left |\Prob \left \{ \sqrt{n} (U_{n,N}'-\theta) \in R \right \} -\Prob (Y \in R)  \right| \\
& \quad = \sup_{R \in \calR} \left|\Prob \{ \sqrt{N} (U_{n,N}'-\theta) \in R \} - \Prob(\alpha_{n}^{-1/2} Y \in R) \right| \le C \left( D_n^2 \log^7(dn) \over n \wedge N \right)^{1/6},
\end{split}
\end{equation}
where $Y \sim N(0,r^{2}\Gamma_{g}+\alpha_{n}\Gamma_{h})$.
\end{thm}

Theorem \ref{thm:gaussian_approx_bern_random_design} shows that the distribution of $\sqrt{n}(U_{n,N}'-\theta)$ can be approximated by the Gaussian distribution $N(0,r^{2}\Gamma_{g}+\alpha_{n} \Gamma_{h})$ on the hyperrectangles  provided that $D_{n}^{2} \log^{7}(dn) \ll n \wedge N$, from which we deduce that the Gaussian approximation on the hyperrectangles holds for $U_{n,N}'$ even when $d \gg n$.  Asymptotically, if e.g. $D_{n}$ is bounded in $n$ and $N \ge n$, then as $n \to \infty$, 
\[
\sup_{R \in \calR} \left |\Prob \left \{ \sqrt{n} (U_{n,N}'-\theta) \in R \right \} -\Prob (Y \in R)  \right| \to 0
\]
whenever $d = d_{n}$ satisfies that $\log d = o(n^{1/7})$, so that the high-dimensional CLT on the hyperrectangles holds for the incomplete $U$-statistics even in ultra-high dimensional cases where $d$ is much larger than $n$. Similar comments apply to all the other results we will derive.

For complete and non-degenerate $U$-statistics (a special case of incomplete $U$-statistics with the complete design and $N = |I_{n,r}|$), it has been argued in \cite{cck2017_AoP} ($r=1$) and \cite{chen2017a} ($r=2$) that the rate of convergence in Theorem \ref{thm:gaussian_approx_bern_random_design} is nearly optimal in the regime where $d$ grows sub-exponentially fast in $n$. On the other hand, the rate of convergence can be improved to $n^{-1/4}$ (up to logarithmic factors) if $d = O(n^{1/7})$, namely if the dimension increases at most polynomially fast with the sample size. 

In the cases where $N \gg n$ (i.e., $\alpha_{n} \ll 1$) and $N \ll n$ (i.e, $\alpha_{n} \gg 1$), the approximating distribution can be simplified to $N(0,r^{2}\Gamma_{g})$ and $N(0,\Gamma_{h})$, respectively. 

\begin{cor}
\label{cor:Gaussian_approximation}
Suppose that Conditions (C1), (C2), and (C3-ND) hold.
Then there exists a constant $C$ depending only on $\underline{\sigma}$ and $r$ such that 
\[
\begin{split}
&\sup_{R \in \calR} \left |\Prob \left \{ \sqrt{n} (U_{n,N}'-\theta) \in R \right \} -\gamma_{A}(R)  \right| \\
&\quad \le C \left \{ \left (\frac{n D_{n} \log^{2} d}{N} \right )^{1/3} + \left( D_n^2 \log^7(dn) \over n \wedge N \right)^{1/6} \right \},
\end{split}
\]
where $\gamma_{A} = N(0,r^{2}\Gamma_{g})$, and 
\[
\begin{split}
&\sup_{R \in \calR} \left |\Prob \left \{ \sqrt{N} (U_{n,N}'-\theta) \in R \right \} -\gamma_{B}(R)  \right| \\
&\quad \le C \left \{ \left (\frac{N D_{n} \log^{2} d}{n} \right )^{1/3} + \left( D_n^2 \log^7{d} \over n \wedge N \right)^{1/6} \right \},
\end{split}
\]
where $\gamma_{B} = N(0,\Gamma_{h})$. 
\end{cor}

\begin{rmk}[Comments on the computational and statistical trade-off for the randomized incomplete $U$-statistics with non-degenerate kernels]
\label{rmk:computation_statistics_tradeoff}
Theorem \ref{thm:gaussian_approx_bern_random_design} and Corollary \ref{cor:Gaussian_approximation} reveal  an interesting phase transition phenomenon between the computational complexity and the statistical efficiency for the randomized incomplete $U$-statistics. Suppose that $n \wedge N \gg D_{n}^{2} \log^{7} (dn)$ and $\underline{\sigma}$ is bounded away from zero. First, if the computational budget parameter $N$ is {\it superlinear} in the sample size $n$ (i.e., $N \gg nD_{n} \log^{2} d$), then both the incomplete $U$-statistic $\sqrt{n}(U_{n,N}'-\theta)$ and its complete version $\sqrt{n}(U_{n}-\theta)$ can be approximated by the same Gaussian distribution $\gamma_{A} = N(0, r^2 \Gamma_{g})$ (cf. \cite{chen2017a} for $r = 2$ case). Second, if $N$ is of the same order as $n$, then the scaling factor of $U_{n,N}'$ remains the same as for $U_{n}$, namely, $\sqrt{n}$. However, the approximating Gaussian distribution for $\sqrt{n}(U_{n,N}'-\theta)$ has covariance matrix $r^{2} \Gamma_{g} + \alpha_{n} \Gamma_{h}$, which is larger than the the corresponding  covariance matrix $r^{2} \Gamma_{g}$ for $\sqrt{n}(U_{n}-\theta)$ in the sense that their difference $\alpha_{n} \Gamma_{h}$ is positive semi-definite. In this case, we sacrifice the statistical efficiency for the sake of keeping the computational cost linear in $n$. Third, if we further reduce the computational budget parameter $N$ to be {\it sublinear} in $n$ (i.e., $N \ll n/(D_{n} \log^{2}d)$), then the scaling factor of $U_{n,N}'$ changes from $\sqrt{n}$ to $\sqrt{N}$, and the distribution of $U_{n,N}'$ is approximated by $N(\theta,N^{-1}\Gamma_{h})$ on the hyperrectangles. Hence, the decay rate of the covariance matrix of the approximating Gaussian distribution  is now $N^{-1}$, which is slower than the $n^{-1}$ rate for the previous two cases. 
\end{rmk}

Next, we consider the case where the kernel $h$ is degenerate, i.e., $P(g_{j}-\theta_{j})^2 = 0$ for all $j = 1,\dots,d$. 
We consider the case where the kernel $h$ is degenerate of order $k-1$ for some $k=2,\dots,r$,  i.e., $P^{r-k+1}h(x_{1},\dots,x_{k-1})=P^{r}h$ for all $(x_{1},\dots,x_{k-1}) \in S^{k-1}$. Even in such cases, a Gaussian approximation holds for $\sqrt{N}(U_{n,N}'-\theta)$ on the hyperrectangles provided that $N \ll n^{k}$ up to logarithmic factors. More precisely, we obtain the following theorem. 

\begin{thm}[Gaussian approximation under degeneracy]
\label{thm:gaussian_approx_bern_random_design_degenerate}
Suppose the kernel $h$ is degenerate of order $k-1$ for some $k=2,\dots,r$. In addition, suppose that Conditions (C1), (C2), and (C3-D) hold.
Then there exists a constant $C$ depending only on $\underline{\sigma}$ and $r$ such that 
\begin{equation}
\label{eqn:gaussian_approx_bern_random_design_degenerate}
\begin{split}
&\sup_{R \in \calR} \left |\Prob \left \{ \sqrt{N} (U_{n,N}'-\theta) \in R \right \}- \gamma_{B}(R) \right| \\
&\le  C \left \{ \left( {N D_n^2 \log^{k+3} d \over n^k} \right)^{1/4} 
+\left({D_n^2 (\log n) \log^5(dn) \over n}\right)^{1/6}+ \left({D_n^2 \log^7 (dn)\over N}\right)^{1/6} \right \},
\end{split}
\end{equation}
where $\gamma_{B} = N(0,\Gamma_{h})$. 
\end{thm}

\begin{rmk}[Comments on the Gaussian approximation under degeneracy]
\label{rmk:gauss_appro_degenerate}
In the degenerate case, for the Gaussian approximation to hold, we must have $N \ll n^k$ (more precisely, $N \ll n^{k}/(D_{n}^{2}\log^{k+3}d)$), which is an indispensable condition even for the $d=1$ case. To see this, consider the Bernoulli sampling case (similar arguments apply to the sampling with replacement case) and  observe that $\sqrt{N} (U_{n,N}'-\theta) = (N/\hat{N}) \cdot \sqrt{N} W_{n} = (N/\hat{N}) (\sqrt{N} A_{n} + \sqrt{N(1-p_{n})} B_{n})$, where $A_{n} = U_{n} - \theta$ and $B_n = U'_{n,N} - U_{n}$. According to Theorem 12.10 in \cite{vandervaart1998}, $n^{k/2} A_{n}$ converges in distribution to a Gaussian chaos of order $k$. Hence, in order to approximate $\sqrt{N}(U_{n,N}' - \theta) \approx \sqrt{N} W_{n}$ by a Gaussian distribution, it is necessary that $\sqrt{N} A_{n}$ is stochastically vanishing, which leads to the condition $N \ll n^{k}$. 

It is worth noting that Theorem \ref{thm:gaussian_approx_bern_random_design_degenerate} reveals a fundamental difference between  complete and randomized incomplete $U$-statistics with the degenerate kernel. Namely, in the degenerate case, the complete $U$-statistic $n^{k/2}(U_{n} - \theta)$ is known to have a non-Gaussian limiting distribution when $d$ is fixed, while thanks to the randomizations, our incomplete $U$-statistics $\sqrt{N}(U_{n,N}'-\theta)$ can be approximated by the Gaussian distribution, and in addition the Gaussian approximation can hold even when $d \gg n$. On one hand, the rate of convergence of the incomplete $U$-statistics is $N^{-1/2}$ and is slower than that of the complete $U$-statistic, namely, $n^{-k/2}$. So in that sense we are sacrificing the rate of convergence by using the incomplete $U$-statistics instead of the complete $U$-statistic, although the rate $N^{-1/2}$ can be arbitrarily close to $n^{-k/2}$ up to logarithmic factors. On the other hand, the approximating Gaussian distribution for the incomplete $U$-statistics is easy to estimate by using a multiplier bootstrap developed in Section \ref{sec:bootstrap}. The multiplier bootstrap developed in Section \ref{sec:bootstrap} is computationally much less demanding than e.g., the empirical bootstraps for complete (degenerate) $U$-statistics \citep[cf.][]{bretagnolle1983,arconesgine1992}, and can consistently estimate the approximating Gaussian distribution $\gamma_{B}$ on the hyperrectangles even when $d \gg n$; see Theorem \ref{thm:bootstrap_validity_A}. To the best of our knowledge, there is no existing work that formally derives Gaussian chaos approximations to degenerate $U$-statistics in high dimensions where $d \gg n$, and in addition such non-Gaussian approximating distributions appear to be more difficult to estimate in high dimensions. Hence, in the degenerate case, the randomizations not only reduce the computational cost but also provide more tractable alternatives to make statistical inference on $\theta$ in high dimensions.

\end{rmk}

\begin{rmk}[Effect of deterministic normalization in the Bernoulli sampling case]
\label{rem:deterministic_normalization}
In the Bernoulli sampling case, consider the deterministic normalization, i.e., $\breve{U}_{n,N}' = N^{-1} \sum_{\iota \in I_{n,r}} Z_{\iota} h(X_{\iota})$, instead of the random one, i.e., $U_{n,N}' = \hat{N}^{-1} \sum_{\iota \in I_{n,r}}Z_{\iota} h(X_{\iota})$. Then, in the non-degenerate case, the distribution of $\sqrt{n}(\breve{U}_{n,N}'-\theta)$ can be approximated by $N(0,r^{2}\Gamma_{g}+\alpha_{n} P^{r}hh^{T})$, and in the degenerate case, $\sqrt{N} (\breve{U}_{n,N}' - \theta)$ can be approximated by $N(0,P^{r}hh^{T})$ (provided that $N \ll n^{k}$ for the degenerate case). To see this, observe that $\breve{U}_{n,N}' - \theta = (U_{n} - \theta) + N^{-1} \sum_{\iota \in I_{n,r}} (Z_{\iota} - p_{n}) h(X_{\iota})$, and the distribution of $N^{-1} \sum_{\iota \in I_{n,r}} (Z_{\iota} - p_{n}) h(X_{\iota})$ can be approximated by $N(0,(1-p_{n})P^{r}hh^{T})$. 
Since $P^{r}hh^{T}$ is larger than $\Gamma_{h}$ unless $\theta=0$ (in the sense that $P^{r}hh^{T} - \Gamma_{h} = \theta \theta^{T}$ is positive semi-definite), the approximating Gaussian distributions have larger covariance matrices for $\breve{U}_{n,N}'$ than those for $U_{n,N}'$, and hence it is in general recommended to use the random normalization rather than the deterministic one. A numerical comparison between these normalizations can be found in Section \ref{sec:normalization_effect} of the SM.
\end{rmk}

\begin{rmk}[Comparisons with \cite{Janson1984_PTRF} for $d = 1$]
\label{rmk:compare_Janson}
The Gaussian approximation results established in Theorems \ref{thm:gaussian_approx_bern_random_design}, \ref{thm:gaussian_approx_bern_random_design_degenerate}, and Corollary \ref{cor:Gaussian_approximation} can be considered as (partial) extensions of Theorem 1 and Corollary 1 in \cite{Janson1984_PTRF} to high dimensions. 
\cite{Janson1984_PTRF} focuses on the univariate case ($d=1$) and derives the asymptotic distributions of randomized incomplete $U$-statistics based on sampling without replacement, sampling with replacement, and Bernoulli sampling (\cite{Janson1984_PTRF} considers the deterministic normalization for the Bernoulli sampling case). For the illustrative purpose, consider sampling with replacement. Suppose that $p_n \to p \in [0,1]$ and the kernel $h$ is degenerate of order $k-1$ for some $k=1,\dots,r$ (the $k = 1$ case corresponds to a non-degenerate kernel).  Then Theorem 1 in \cite{Janson1984_PTRF} shows that  $(n^{k/2}(U_{n}-\theta), N^{1/2}(U_{n,N}'-U_{n})) \stackrel{d}{\to} (V,W)$, where $V$ is a  Gaussian chaos of order $k$ (in particular, $V \sim N(0,r^{2} P(g-\theta)^{2})$ if $k=1$) and $W \sim N(0,  P^{r}(h - \theta)^{2})$ such that $V$ and $W$ are independent. Hence, provided that $n^{k}/N \to \alpha \in [0,\infty]$, $n^{k/2} (U_{n,N}'-\theta) \stackrel{d}{\to} V+\alpha W$ if $\alpha < \infty$ and $\sqrt{N} (U_{n,N}'-\theta) \stackrel{d}{\to} W$ if $\alpha=\infty$. The present paper focuses on the cases where the approximating distributions are Gaussian (i.e., the cases where $k=1$ and $\alpha$ is finite, or $k \ge 2$ and $\alpha=\infty$), but covers high-dimensional kernels and derives explicit and non-asymptotic Gaussian approximation error bounds that are not obtained in \cite{Janson1984_PTRF}. In addition, the proof strategy of our Gaussian approximation results differs substantially from that of \cite{Janson1984_PTRF}. \cite{Janson1984_PTRF} shows the convergence of the joint characteristic function of $(n^{k/2}(U_{n}-\theta), N^{1/2}(U_{n,N}'-U_{n}))$ to obtain his Theorem 1, but 
the characteristic function approach is not very useful to derive explicit error  bounds on distributional approximations in high dimensions. Instead, our proofs iteratively use conditioning arguments combined with Berry-Esseen type bounds.

Finally, we expect that the results of the present paper can be extended to the case where $k \ge 2$ and $\alpha$ is finite; in that case, the approximating distribution to $n^{k/2}(U_{n,N}'-\theta)$ will be non-Gaussian and the technical analysis will be more involved in high dimensions. 
We leave the analysis of this case as a future research topic. 
\end{rmk}

\begin{rmk}[Relaxation of sub-exponential moment Condition (C2)]
\label{rmk:relaxation_moment_condition}
It is possible to relax the sub-exponential moment Condition (C2) to a polynomial moment condition. Suppose that 
\begin{enumerate}
\item[(C2')] $(P^{r}|h|_{\infty}^{q})^{1/q} \le D_{n}$ for some $q \in [4,\infty)$.
\end{enumerate}
Condition (C2') covers a kernel with bounded polynomial moment of a finite degree $q$. 

\begin{thm}[Gaussian approximation under polynomial moment condition]
\label{thm:gaussian_approx_bern_random_design_polymom}
(i) If Conditions (C1), (C2'), and (C3-ND) hold, then there exists a constant $C$ depending only on $\underline{\sigma},r$, and $q$ such that 
\begin{equation}
\label{eqn:gaussian_approx_bern_random_design_polymom}
\begin{split}
& \sup_{R \in \calR} \left |\Prob \left \{ \sqrt{n} (U_{n,N}'-\theta) \in R \right \} -\Prob (Y \in R)  \right| \\
& \qquad \le C \left\{ \left( D_n^2 \log^7(dn) \over n \wedge N \right)^{1/6} +  \left( {D_{n}^{2} n^{2r/q} \log^{3}(dn) \over (n \wedge N)^{1-2/q} } \right)^{1/3} \right\},
\end{split}
\end{equation}
where $Y \sim N(0,r^{2}\Gamma_{g}+\alpha_{n}\Gamma_{h})$. \\
(ii) Suppose the kernel $h$ is degenerate of order $k-1$ for some $k=2,\dots,r$. If Conditions (C1), (C2'), and (C3-D) hold, then there exists a constant $C$ depending only on $\underline{\sigma},r$, and $q$ such that 
\begin{equation}
\label{eqn:gaussian_approx_degenerate_bern_random_design_polymom}
\begin{split}
&\sup_{R \in \calR} \left |\Prob \left \{ \sqrt{N} (U_{n,N}'-\theta) \in R \right \}- \gamma_{B}(R) \right| \\
&\le  C \left\{ \left( {N D_n^2 \log^{k+3} d \over n^k} \right)^{1/4} +\left({D_n^2 \log^5(dn) \over n}\right)^{1/6}  \right. \\
& \qquad \quad \quad \left. + \left({D_n^2 \log^7 {d} \over N}\right)^{1/6} + \left( {D_{n}^{2} n^{2r/q} \log^{3}{d} \over (n \wedge N)^{1-2/q} } \right)^{1/3} \right\},
\end{split}
\end{equation}
where $\gamma_{B} = N(0, \Gamma_{h})$. 
\end{thm}
Comparing Theorem \ref{thm:gaussian_approx_bern_random_design_polymom} with Theorems \ref{thm:gaussian_approx_bern_random_design} and \ref{thm:gaussian_approx_bern_random_design_degenerate}, we see that the same approximating Gaussian distributions under the sub-exponential moment condition (C2) are valid under the polynomial moment condition (C2') as well. The rates of convergence to the Gaussian distributions under (C2') involve an extra Nagaev-type term similarly to the sample average and  complete $U$-statistic cases (cf. \cite{cck2017_AoP,chen2017a}), and so the rates may be slower than those obtained under the sub-exponential moment condition (C2). In particular, the rates in (\ref{eqn:gaussian_approx_bern_random_design_polymom}) and (\ref{eqn:gaussian_approx_degenerate_bern_random_design_polymom}) now depend on the order $r$ through the term $n^{2r/q}$. Still, the leading orders in (\ref{eqn:gaussian_approx_bern_random_design_polymom}) and (\ref{eqn:gaussian_approx_degenerate_bern_random_design_polymom}) coincide with those under the sub-exponential moment condition (C2) as long as $q$ is sufficiently large compared with $r$. 
For example, if $D_n$ is bounded in $n$, $N \ge n$, and $q \ge 4(r+1)$, then the leading order of (\ref{eqn:gaussian_approx_bern_random_design_polymom}) is $(n^{-1} \log^{7}(dn))^{1/6}$, which coincides with that in the sub-exponential case. 


\end{rmk}

\section{Bootstrap approximations}
\label{sec:bootstrap}

The Gaussian approximation results developed in the previous section are often not directly applicable in statistical applications since the covariance matrix of the approximating Gaussian distribution, $r^{2} \Gamma_{g} + \alpha_{n} \Gamma_{h}$ (or $\Gamma_{h}$ in the degenerate case), is unknown to us. In this section, we develop data-dependent procedures to further approximate or estimate the $N(0,r^{2} \Gamma_{g} + \alpha_{n}\Gamma_{h})$ distribution (or the $N(0,\Gamma_{h})$ distribution in the degenerate case) that are computationally (much) less-demanding than existing bootstrap methods for $U$-statistics such as the empirical bootstrap. 

\subsection{Generic bootstraps for incomplete $U$-statistics}
\label{sec:generic_bootstrap}

Let $\calD_{n} = \{ X_{1},\dots,X_{n}\} \cup \{ Z_{\iota} : \iota \in I_{n,r} \}$. For the illustrative purpose, consider to estimate the $N(0,r^{2}\Gamma + \alpha_{n} \Gamma_{h})$ distribution and let $Y \sim N(0,r^{2}\Gamma_{g}+\alpha_{n}\Gamma_{h})$.  The basic idea of our approach is as follows. Since $Y \stackrel{d}{=} Y_{A} + \alpha_{n}^{1/2} Y_{B}$ where $Y_{A} \sim N(0,r^{2}\Gamma_{g})$ and $Y_{B} \sim N(0,\Gamma_{h})$ are independent, to approximate the distribution of $Y$, it is enough to construct data-dependent random vectors $U_{n,A}^{\sharp}$ and $U_{n,B}^{\sharp}$ such that, conditionally on $\calD_{n}$, (i) $U_{n,A}^{\sharp}$ and $U_{n,B}^{\sharp}$ are independent, and (ii) the conditional distributions of $U_{n,A}^{\sharp}$ and $U_{n,B}^{\sharp}$ are computable and ``close'' enough to $N(0,r^{2}\Gamma_{g})$ and $N(0,\Gamma_{h})$, respectively. Then, the conditional distribution of $U_{n}^{\sharp} = U_{n,A}^{\sharp} +\alpha_{n}^{1/2} U_{n,B}^{\sharp}$ should be close to $N(0,r^{2} \Gamma_{g} + \alpha_{n}  \Gamma_{h})$ and hence to the distribution of $\sqrt{n}(U_{n,N}'-\theta)$. Of course, if the target distribution is $N(0,r^{2}\Gamma_{g})$ or $N(0,\Gamma_{h})$, then it is enough to simulate the conditional distribution of $U_{n,A}^{\sharp}$ or $U_{n,B}^{\sharp}$ alone, respectively.

Construction of $U_{n,B}^{\sharp}$ is straightforward; in fact, it is enough to apply the (Gaussian) multiplier bootstrap to 
$\sqrt{Z_{\iota}} h(X_{\iota}), \iota \in I_{n,r}$.

\uline{Construction of $U_{n,B}^{\sharp}$}.
\begin{enumerate}
\item[1.] Generate i.i.d. $N(0,1)$ variables $\{ \xi_{\iota}' : \iota \in I_{n,r} \}$ independent of the data $\calD_{n}$.
\item[2.] Construct 
\[
U_{n,B}^{\sharp}  =
 \frac{1}{\sqrt{\hat{N}}} \sum_{\iota \in I_{n,r}} \xi_{\iota}' \sqrt{Z_{\iota}} \{ h(X_{\iota}) - U_{n,N}' \},
\]
where $\hat{N}$ is replaced by $N$ for the sampling with replacement case.
\end{enumerate}

In the Bernoulli sampling case, $U_{n,B}^{\sharp}$ reduces to $U_{n,B}^{\sharp} =\hat{N}^{-1/2} \sum_{j=1}^{\hat{N}} \xi_{\iota_{j}}' \{ h(X_{\iota_{j}}) - U_{n,N}' \}$, while in the sampling with replacement case, simulating $U_{n,B}^{\sharp}$ can be equivalently implemented by simulating $U_{n,B}^{\sharp} = N^{-1/2}\sum_{j=1}^{N} \eta_{j} \{ h(X_{\iota_{j}}^{*}) - U_{n,N}' \}$ for $\eta_{1},\dots,\eta_{N} \sim N(0,1)$ i.i.d. independent of $X_{\iota_{1}}^{*},\dots,X_{\iota_{N}}^{*}$; in fact the distribution of $U_{n,B}^{\sharp}$ in the latter definition (conditionally on $X_{\iota_{1}}^{*},\dots.X_{\iota_{N}}^{*}$) is Gaussian with mean zero and covariance matrix $N^{-1} \sum_{j=1}^{N} \{ h(X_{\iota_{j}}^{*}) - U_{n,N}' \}\{ h(X_{\iota_{j}}^{*}) - U_{n,N}' \}^{T}$, which is identical to the conditional distribution of $U_{n,B}^{\sharp}$ in the original definition.
In either case, in practice, we only need to generate (on average) $N$ multiplier variables.  The following theorem establishes conditions under which the conditional distribution of $U_{n,B}^{\sharp}$ is able to consistently estimate the $N(0,\Gamma_{h}) \ (=\gamma_{B})$ distribution on the hyperrectangles with  polynomial error rates.

\begin{thm}[Validity of $U_{n,B}^{\sharp}$]
\label{thm:bootstrap_validity_A}
Suppose that (C1), (C2), and (C3-D) hold. If  
\begin{equation}
\label{eqn:growth_condition_B}
\frac{D_{n}^{2} (\log^{2}{n}) \log^{5}(dn)}{n \wedge N} \le C_{1}n^{-\zeta}
\end{equation}
for some constants $0 < C_{1} < \infty$ and  $\zeta \in (0,1)$, then there exists a constant $C$ depending only on $\underline{\sigma}, r$, and $C_{1}$ such that 
\begin{equation}
\label{eqn:thm:bootstrap_validity_B}
\sup_{R \in \calR} \left | \Prob_{\mid \calD_{n}} (U_{n,B}^{\sharp} \in R) - \gamma_{B}(R) \right | \le Cn^{-\zeta/6}
\end{equation}
with probability at least $1-Cn^{-1}$.
\end{thm}

\begin{rmk}[Bootstrap validity under the polynomial moment condition]
Analogous bootstrap validity results for $U_{n,B}^{\sharp}$ in Theorem \ref{thm:bootstrap_validity_A}, as well as those for $U_{n}^{\sharp}$ and $U_{n,A}^{\sharp}$ in Theorem \ref{thm:bootstrap_validity}, \ref{cor:bootstrap_validity} and Proposition \ref{prop:bootstrap_validity_DC}, \ref{prop:bootstrap_validity_incomplete_ustat} ahead, can be obtained under the polynomial moment Condition (C2'). Due to the space concern, detailed results can be found in Section \ref{sec:bootstrap_validity_poly} of the SM. 
\end{rmk}

In the degenerate case, the approximating distribution is $\gamma_{B} = N (0,\Gamma_{h})$.
So, in that case, we can approximate the distribution of $\sqrt{N}(U_{n,N}' - \theta)$ on the hyperrectangles by the conditional distribution of $U_{n,B}^{\sharp}$, which can be simulated by drawing multiplier variables many times. We call the simulation of $U_{n,B}^{\sharp}$ the {\it multiplier bootstrap under degeneracy} (MB-DG). On average, the computational cost of the MB-DG is $O(BNd)$ (where $B$ denotes the number of bootstrap iterations), which can be independent of the order of the $U$-statistic provided that $N$ is so.  In the remainder of this section, we will focus on the non-degenerate case.

In contrast to $U_{n,B}^{\sharp}$, construction of $U_{n,A}^{\sharp}$ is more involved. We might be tempted to apply the multiplier bootstrap to the H\'{a}jek projection, $rn^{-1} \sum_{i_{1}=1}^{n} g(X_{i_{1}})$, but the function $g=P^{r-1}h$ is unknown so the direct application of the multiplier bootstrap to the H\'{a}jek projection  is infeasible. Instead, we shall construct estimates of $g(X_{i_{1}})$ for $i_{1} \in \{ 1,\dots,n \}$ or a subset of $\{ 1,\dots ,n \}$, and then apply the multiplier bootstrap to the estimated H\'{a}jek projection.  Generically, construction of $U_{n,A}^{\sharp}$ is as follows.

\uline{Generic construction of $U_{n,A}^{\sharp}$}.
\begin{enumerate}
\item[1.] Choose a subset $S_{1}$ of $\{ 1,\dots, n \}$ and generate i.i.d. $N(0,1)$ variables $\{ \xi_{i_{1}} : i_{1} \in S_{1} \}$ independent of the data $\calD_{n}$ and $\{ \xi_{\iota}' : \iota \in I_{n,r} \}$. Let $n_{1}=|S_{1}|$. 
\item[2.] For each $i_{1} \in S_{1}$, construct an estimate $\hat{g}^{(i_{1})}$ of $g$ based on $X_{1}^{n}$. 
\item[3.] Construct
\[
U_{n,A}^{\sharp} = \frac{r}{\sqrt{n_{1}}} \sum_{i_{1} \in S_{1}} \xi_{i_{1}} \{ \hat{g}^{(i_{1})}(X_{i_{1}}) - \breve{g}  \},
\]
where $\breve{g} =n_{1}^{-1}\sum_{i_{1} \in S_{1}} \hat{g}^{(i_{1})}(X_{i_{1}})$.
\end{enumerate}

Step 1 chooses a subset $S_{1}$ to reduce the computational cost of the resulting bootstrap. 
Construction of estimates $\hat{g}^{(i_{1})}, i_{1} \in S_{1}$ can be flexible. 
For instance, the estimates $\hat{g}^{(i_{1})}, i_{1} \in S_{1}$ may depend on another randomization independent of everything else.  
In Sections \ref{subsec:divide_and_conquer} and \ref{subsec:incomplete_ustat}, we will consider deterministic and random constructions of $\hat{g}^{(i_{1})}, i_{1} \in S_{1}$, respectively.

It is worth noting that the jackknife multiplier bootstrap (JMB) developed in \cite{chen2017a} (for the $r=2$ case) and \cite{chenkato2017a} (for the general $r$ case) is a special case of $U_{n,A}^{\sharp}$ where $S_{1} = \{1,\dots,n\}$ and $\hat{g}^{(i_{1})}(X_{i_{1}})$ is realized by its jackknife estimate, i.e., by the $U$-statistic with kernel $(x_{2},\dots,x_{r}) \mapsto h(X_{i_{1}},x_{2},\dots,x_{r})$ for the sample without the $i_{1}$-th observation. Nevertheless, the bottleneck is that the computation of the jackknife estimates of $g(X_{i_{1}}), i_{1}=1,\dots,n$ requires $O(n^{r}d)$ operations and hence implementing the JMB can be  computationally demanding.

Now, consider $U_{n}^{\sharp} = U_{n,A}^{\sharp} + \alpha_{n}^{1/2} U_{n,B}^{\sharp}$. We call the simulation of  $U_{n}^{\sharp}$ the {\it multiplier bootstrap under non-degeneracy} (MB-NDG).  The following theorem establishes conditions under which the conditional distribution of $U_{n}^{\sharp}$ is able to consistently estimate the $N(0,r^{2}\Gamma_{g} + \alpha_{n} \Gamma_{h})$ distribution on the hyperrectangles with polynomial error rates. 
Define 
\[
\hat{\Delta}_{A,1}:= \max_{1 \le j \le d} {1 \over n_{1}} \sum_{i_{1} \in S_{1}}  \{ \hat{g}_{j}^{(i_{1})}(X_{i_{1}}) - g_{j}(X_{i_{1}}) \}^{2},
\]
which quantifies the errors of the estimates $\hat{g}^{(i_{1})}, i_{1} \in S_{1}$.  In addition, let $\overline{\sigma}_{g} := \max_{1 \le j \le d} \sqrt{P(g_{j}-\theta_{j})^{2}}$. 

\begin{thm}[Generic bootstrap validity under non-degeneracy]
\label{thm:bootstrap_validity}
Let $U_{n}^{\sharp} = U_{n,A}^{\sharp} + \alpha_{n}^{1/2} U_{n,B}^{\sharp}$. 
Suppose that Conditions (C1), (C2), and (C3-ND) hold. In addition, suppose  that 
\begin{equation}
\label{eqn:growth_condition}
\begin{split}
&\frac{D_{n}^{2}(\log^{2} {n})\log^{5}(dn)}{n_{1} \wedge N} \le C_{1}n^{-\zeta_{1}} \quad  \text{and} \\
&\Prob \left ( \overline{\sigma}_{g}^{2} \hat{\Delta}_{A,1}\log^{4} d > C_{1}n^{-\zeta_{2}} \right ) \le C_{1}n^{-1}\end{split}
\end{equation}
for some constants $0 < C_{1} < \infty$ and $\zeta_{1},\zeta_{2} \in (0,1)$. Then there exists a constant $C$ depending only on $\underline{\sigma}, r$, and $C_{1}$ such that 
\begin{equation}
\label{eqn:bootstrap_validity}
\sup_{R \in \calR} \left | \Prob_{\mid \calD_{n}} (U_{n}^{\sharp} \in R) - \Prob (Y \in R) \right | \le Cn^{-(\zeta_{1} \wedge \zeta_{2})/6}
\end{equation}
with probability at least $1-Cn^{-1}$, where $Y \sim N(0,r^{2}\Gamma_{g}+\alpha_{n}\Gamma_{h})$. If the estimates $g^{(i_{1})}, i_{1} \in S_{1}$ depend on an additional randomization independent of $\calD_{n}, \{ \xi_{i_{1}} : i_{1} \in S_{1} \}$, and $\{ \xi_{\iota}': \iota \in I_{n,r} \}$, then the result (\ref{eqn:bootstrap_validity}), with $\calD_{n}$ replaced by the augmentation of $\calD_{n}$ with variables used in the additional randomization,  holds with probability at least $1-Cn^{-1}$. 
\end{thm}

The second part of Condition (\ref{eqn:growth_condition}) is a high-level condition on the estimation accuracy of $\hat{g}^{(i_{1})}, i_{1} \in S_{1}$. In Sections \ref{subsec:divide_and_conquer} and \ref{subsec:incomplete_ustat}, we will verify the second part of Condition (\ref{eqn:growth_condition}) for deterministic and random constructions of $\hat{g}^{(i_{1})}, i_{1} \in S_{1}$.
The bootstrap distribution is taken with respect to the multiplier variables $\{ \xi_{i_{1}} : i_{1} \in S_{1} \}$ and $\{ \xi_{\iota}': \iota \in I_{n,r} \}$, and so if the estimation step for $g$ depends on an additional randomization, then the variables used in the additional randomization have to be generated outside the bootstrap iterations. 

When the approximating distribution can be simplified to $\gamma_{A} = N(0,r^{2} \Gamma_{g})$, then it suffices  to estimate $N(0,r^{2}\Gamma_{g})$ by the conditional distribution of $U_{n,A}^{\sharp}$.

\begin{cor}[Validity of $U_{n,A}^{\sharp}$]
\label{cor:bootstrap_validity}
Suppose that all the conditions in Theorem \ref{thm:bootstrap_validity} hold. Then there exists a constant $C$ depending only on $\underline{\sigma}, r$, and $C_{1}$ such that 
\begin{equation}
\label{eqn:bootstrap_validity_2}
\sup_{R \in \calR} \left | \Prob_{\mid \calD_{n}} (U_{n,A}^{\sharp} \in R) - \gamma_{A}(R) \right | \le Cn^{-(\zeta_1 \wedge \zeta_2)/6}
\end{equation}
with probability at least $1-Cn^{-1}$.  If the estimates $g^{(i_{1})}, i_{1} \in S_{1}$ depend on an additional randomization independent of $\calD_{n}, \{ \xi_{i_{1}} : i_{1} \in S_{1} \}$, and $\{ \xi_{\iota}': \iota \in I_{n,r} \}$, then the result (\ref{eqn:bootstrap_validity_2}), with $\calD_{n}$ replaced by the augmentation of $\calD_{n}$ with variables used in the additional randomization,  holds with probability at least $1-Cn^{-1}$.
\end{cor}

\begin{rmk}[Comments on the partial bootstrap simplification under non-degeneracy]
When the approximating distribution of $\sqrt{N} (U_{n,N}'-\theta)$ can be simplified to $\gamma_{B} = N(0, \Gamma_{B})$, it is also possible to use the partial bootstrap $U_{n,B}^{\sharp}$ to estimate $N(0, \Gamma_{B})$. In this case, we must take $N$ to be sublinear in $n$ (i.e., $N \ll n/(D_{n} \log^{2}d)$) to ensure the Gaussian approximation validity (cf. Remark \ref{rmk:computation_statistics_tradeoff}). However, we do not recommend this simplification because the decay rate of the covariance matrix of the approximating Gaussian distribution $N(\theta,N^{-1}\Gamma_{B})$ to $U_{n,N}'$ is $N^{-1}$, which is slower than the $n^{-1}$ rate for the linear and superlinear cases. In particular, this implies a power loss in the testing problems if the critical values are calibrated by $U_{n,B}^{\sharp}$. 
\end{rmk}

The rest of this section is devoted to concrete constructions of the estimates $\hat{g}^{(i_{1})}, i_{1} \in S_{1}$.

\subsection{Divide and conquer estimation}
\label{subsec:divide_and_conquer}

We first propose a deterministic construction of $\hat{g}^{(i_{1})}, i_{1} \in S_{1}$ via the divide and conquer (DC) algorithm (cf. \cite{zhangduchiwainwright2015_jmlr}).

\begin{enumerate}
\item[1.]
For each $i_{1} \in S_{1}$, choose $K$ {\it disjoint} subsets $S_{2,k}^{(i_{1})}, k=1,\dots,K$ with common size $L \geq r-1$ from $\{ 1,\dots, n \} \setminus \{ i_{1} \}$. 
\item[2.] For each $i_{1} \in S_{1}$, estimate $g$ by computing $U$-statistics with kernel $(x_{2},\dots,x_{r}) \mapsto h(x,x_{2},\dots,x_{r})$ applied to the subsamples $\{ X_{i} : i \in S_{2,k}^{(i_{1})} \}, k=1,\dots,K$, and taking the average of those $U$-statistics of order $r-1$, i.e., 
\[
\hat{g}^{(i_{1})} (x) = \frac{1}{K} \sum_{k=1}^{K} \frac{1}{|I_{L,r-1}|} \sum_{\substack{i_{2},\dots,i_{r} \in S_{2,k}^{(i_{1})} \\ i_{2} < \cdots < i_{r}}} h(x,X_{i_{2}},\dots,X_{i_{r}}).
\]
\end{enumerate}

The DC algorithm can be viewed as an estimation procedure for $g$ via incomplete $U$-statistics of order $r-1$ with a {\it block diagonal} sampling scheme (up to a permutation on the indices). We call the simulation of  $U_{n}^{\sharp}$ with the DC algorithm  the \textit{MB-NDG-DC}. In Section \ref{subsec:incomplete_ustat}, we will propose a different estimation procedure for $g$ via  randomized incomplete $U$-statistics of order $r-1$ based on an additional Bernoulli sampling. As a practical guidance to implement the DC algorithm, we suggest to choose $S_{1}=\{ 1,\dots,n \}, L = r-1$, and $K = \lfloor (n-1) / L \rfloor$ consecutive blocks, which are the parameter values used in our simulation examples in Section \ref{sec:numerics}. In this case, the DC algorithm turns out to be calculating Hoeffding's averages of the $U$-statistics of order $r-1$, which requires $O(nd)$ operations for each $i_{1}$. In contrast, the JMB constructs $\hat{g}^{(i_{1})}$ by complete $U$-statistics of order $r-1$, which requires $O(n^{r-1}d)$ operations for each $i_{1}$. Since the estimation step for $g$  can be done outside the bootstrap iterations,  the overall computational cost of the MB-NDG-DC is $O((B N + n_{1} K L + B n_{1}) d) = O(n^{2} d + B (N+n) d)$ (where $B$ denotes the number of bootstrap iterations), which is independent of the order of the $U$-statistic. In addition, if we choose to only simulate $U_{n,A}^{\sharp}$, then the computational cost is $O(n^{2}d+Bnd)$, since the $O(BNd)$ computations come from simulating $U_{n,B}^{\sharp}$. We can certainly make the computational cost even smaller by taking $n_{1}$ and $K$ smaller than $n$. For instance, if we choose $n_{1}$ and $K$ in such a way that $n_{1}K = O(n)$ and $L = r -1$, then the overall computational cost is reduced to $O(nd + B(N+n)d) = O(B(N+n)d)$ (or $O(Bnd)$ if we only simulate $U_{n,A}^{\sharp}$). In general, choosing smaller $n_{1}$ and $K$ would  sacrifice the statistical accuracy of the resulting bootstrap, but if $O(n^{2}d)$ computations are difficult to implement,  then choosing smaller $n_{1}$ and $K$ would be a reasonable option. 

Our MB-NDG-DC substantially differs from the the Bag of Little Bootstraps (BLB) proposed in \cite{KTSJ2014_JRSSB}, which is another generically scalable bootstrap method for large datasets based on the DC algorithm.
Due to the space concern, we defer the comparison of our MB-NDG-DC with the BLB in Section \ref{sec:BLB} of the SM.  

The following proposition provides conditions for the validity of the multiplier bootstrap equipped with the DC estimation (MB-NDG-DC).

\begin{prop}[Validity of bootstrap with DC estimation]
\label{prop:bootstrap_validity_DC}
Suppose that Conditions (C1), (C2), and (C3-ND) hold.  In addition, suppose that 
\begin{equation}
\label{eqn:growth_condition_DC}
\frac{D_{n}^{2}(\log^{2}n) \log^{5}(dn)}{n_{1} \wedge N}  \bigvee \left \{ \frac{\overline{\sigma}_{g}^{2}D_{n}^{2} \log^{7}d}{KL} \left ( 1+\frac{\log^{2} d}{K^{1-1/\nu}} \right ) \right \}  \le C_{1}n^{-\zeta}
\end{equation}
for some constants $0 < C_{1} < \infty, \zeta \in (0,1)$, and $\nu \in (1/\zeta,\infty)$. Then, there exists a constant $C$ depending only on $\underline{\sigma},r, \nu$, and $C_{1}$ such that each of the results (\ref{eqn:bootstrap_validity}) and (\ref{eqn:bootstrap_validity_2}) with $(\zeta_1,\zeta_2) = (\zeta, \zeta-1/\nu)$ holds
with probability at least $1-Cn^{-1}$.
\end{prop}

For instance, consider to take  $N = n, S_1 = \{1,\dots,n\}, L = r-1$, and $K = \lfloor (n-1) / L \rfloor$, and suppose that $D_n^2 \log^7(dn) \le n^{1-\zeta}$ for some $\zeta \in (0,1)$.
Then, by Theorem \ref{thm:gaussian_approx_bern_random_design} and Proposition \ref{prop:bootstrap_validity_DC}, for arbitrarily large $\nu \in (1/\zeta,\infty)$, there exists a constant $C$ depending only on $\overline{\sigma}_{g}, \underline{\sigma}, r$, and $\nu$ such that 
\begin{equation}
\label{eqn:practical_guidance_dc_bootstrap}
\sup_{R \in \calR} \left | \Prob(\sqrt{n}(U_{n,N}'-\theta) \in R) - \Prob_{\mid \calD_{n}}(U_{n}^{\sharp} \in R) \right | \le C n^{-(\zeta-1/\nu)/6}
\end{equation}
with probability at least $1-Cn^{-1}$. Hence, the conditional distribution of the MB-NDG-DC approaches uniformly on the hyperrectangles in $\R^d$ to the distribution of the randomized incomplete $U$-statistic at a polynomial rate in the sample size.

\subsection{Random sampling estimation}
\label{subsec:incomplete_ustat}

Next, we propose a random construction of $\hat{g}^{(i_{1})}, i_{1} \in S_{1}$ based on an additional Bernoulli sampling. 
For each $i_{1} =1,\dots,n$, let $I_{n-1,r-1} (i_{1}) = \{ (i_{2},\dots,i_{r}) : 1 \le i_{2} < \cdots < i_{r} \le n, i_{j} \ne i_{1} \ \forall j \neq 1 \}$.
In addition, define $\sigma_{i_{1}} :  \{ 1,\dots,n-1 \} \to \{ 1,\dots,n \} \setminus \{ i_{1} \}$ as follows: if $\{ 1,\dots,n \} \setminus \{ i_{1} \} = \{ j_{1},\dots,j_{n-1} \}$ with $j_{1} <\cdots < j_{n-1}$, then $\sigma_{i_{1}}(\ell) = j_{\ell}$ for $\ell=1,\dots,n-1$. 
For the notational convenience, for  $\iota' = (i_{2},\dots,i_{r}) \in I_{n-1,r-1}$, we write $\sigma_{i_{1}}(\iota') = (\sigma_{i_{1}}(i_{2}),\dots,\sigma_{i_{1}}(i_{r})) \in I_{n-1,r-1}(i_{1})$.

Now, consider the following randomized procedure to construct $\hat{g}^{(i_{1})}, i_{1} \in S_{1}$. 
 
\begin{enumerate}
\item[1.] Let $0 < M = M_{n}  \le |I_{n-1,r-1}|$ be a positive integer, and generate i.i.d. $\Bern (\vartheta_{n})$ random variables $\{ Z_{\iota'}' : \iota' = (i_{2},\dots,i_{r}) \in I_{n-1,r-1} \}$ independent of $\calD_{n}, \{ \xi_{i_{1}} : i_{1} \in S_{1} \}$, and $\{ \xi_{\iota}' : \iota \in I_{n,r} \}$, where $\vartheta_{n} = M/|I_{n-1,r-1}|$.
\item[2.] For each $i_{1} \in S_{1}$, construct $\hat{g}^{(i_{1})} (x) =M^{-1} \sum_{\iota' \in I_{n-1,r-1}} Z'_{\iota'} h(x,X_{\sigma_{i_{1}}(\iota')})$. 
\end{enumerate}
The resulting bootstrap method is called the {\it multiplier bootstrap under non-degeneracy with random sampling} (MB-NDG-RS). Equivalently, the above procedure can be implemented as follows:
\begin{enumerate}
\item[1.] Generate $\hat{M} \sim \Bin (|I_{n-1,r-1}|,\vartheta_{n})$.
\item[2.] Sample $\iota_{1}',\dots,\iota_{\hat{M}}'$ randomly without replacement from $I_{n-1,r-1}$.
\item[3.] Construct 
 $\hat{g}^{(i_{1})}(x) = M^{-1}\sum_{j=1}^{\hat{M}} h(x,X_{\sigma_{i_{1}}(\iota_{j}')})$ for each $i_{1} \in S_{1}$.
\end{enumerate}
So, on average, the computational cost to construct $\hat{g}^{(i_{1})}, i_{1} \in S_{1}$ by the random sampling estimation is $O(n_{1}Md)$, and the overall computational cost of the MB-NDG-RS is $O(n_{1}Md + B(N+n_{1}) d)$ (or $O(n_{1}Md + B n_{1} d)$ if we only simulate $U_{n,A}^{\sharp}$). As a practical guidance to implement the random sampling estimation, we suggest to choose $S_{1} = \{1,\dots,n\}$ and $M$ proportional to $n-1$, which are the parameter values used in our simulation examples in Section \ref{sec:numerics}. Then the overall computational cost of the MB-NDG-RS is $O(n^{2}d+B(N+n)d)$ (or $O(n^{2}d+Bnd)$ if we only simulate $U_{n,A}^{\sharp}$), which is independent of the order of the $U$-statistic. In addition, the computational cost can be made even smaller, e.g., can be reduced to $O(B(N+n)d)$ by choosing $n_{1}$ and $M$ in such a way that $n_{1}M = O(n)$ (or $O(Bnd)$ if we only simulate $U_{n,A}^{\sharp}$), which would be a reasonable option if $O(n^2 d)$ computations are difficult to implement.

\begin{prop}[Validity of bootstrap with Bernoulli sampling estimation]
\label{prop:bootstrap_validity_incomplete_ustat}
Suppose that Conditions (C1), (C2), and (C3-ND) hold.  In addition, suppose that 
\begin{equation}
\label{eqn:growth_condition_incomplete_ustat}
\frac{D_{n}^{2}(\log^{2} n) \log^{5}(dn)}{n_{1} \wedge N} \bigvee \frac{\overline{\sigma}_{g}^{2}D_{n}^{2} \log^{7}(dn) }{n \wedge M}  \le C_{1}n^{-\zeta}
\end{equation}
for some constants $0 < C_{1} < \infty$ and  $\zeta \in (0,1)$. Then, for arbitrarily large $\nu \in (1/\zeta,\infty)$, there exists a constant $C$ depending only on $\underline{\sigma},r,\nu$, and $C_{1}$ such that each of the results (\ref{eqn:bootstrap_validity}) and (\ref{eqn:bootstrap_validity_2}), with $\calD_{n}$ replaced by $\calD_{n}' = \calD_{n} \cup \{ Z_{\iota'}' : \iota' \in I_{n-1,r-1} \}$ and with $(\zeta_1,\zeta_2) = (\zeta,\zeta-1/\nu)$, holds
with probability at least $1-Cn^{-1}$.
\end{prop}

For instance, consider to take $N = n, S_1 = \{1,\dots,n\}$, and $M $ proportional to $n-1$,
and suppose that $D_n^2 \log^7(dn) \le n^{1-\zeta}$ for some $\zeta \in (0,1)$. 
Then, by Theorem \ref{thm:gaussian_approx_bern_random_design} and Proposition \ref{prop:bootstrap_validity_incomplete_ustat},  for arbitrarily large $\nu \in (1/\zeta,\infty)$, there exists a constant $C$ depending only on $\overline{\sigma}_{g},\underline{\sigma}, r$, and $\nu$ such that the result (\ref{eqn:practical_guidance_dc_bootstrap}) holds  with probability at least $1-Cn^{-1}$. 

\begin{rmk}[Alternative options for random sampling estimation]
In construction of $\hat{g}^{(i_{1})}$, instead of normalization by $M$, we may use normalization by $\hat{M}$, namely, $\hat{M}^{-1}\sum_{j=1}^{\hat{M}} h(x,X_{\sigma_{i_{1}}(\iota_{j}')})$ for $\hat{g}^{(i_{1})}(x)$. 
In view of the concentration inequality for $\hat{M}$ (cf. equation (\ref{eqn:concentration})), it is not difficult to see that the same conclusion of Proposition \ref{prop:bootstrap_validity_incomplete_ustat} holds for $\hat{g}^{(i_{1})} (x) = \hat{M}^{-1}\sum_{j=1}^{\hat{M}} h(x,X_{\sigma_{i_{1}}(\iota_{j}')})$.

Next, alternatively to the Bernoulli sampling, we may use sampling with replacement to construct $\hat{g}^{(i_{1})}$, which can be implemented as follows: 1) sample $\iota_{1}',\dots,\iota_{M}'$ randomly with replacement from $I_{n-1,r-1}$ (independently of everything else); and 2) construct $\hat{g}^{(i_{1})} (x) = M^{-1} \sum_{j=1}^{M}h(x,X_{\sigma_{i_{1}}(\iota_{j}')})$ for $i_{1} \in S_{1}$. For each $i_{1} \in S_{1}$, conditionally on $X_{1}^{n}$, $X_{\sigma_{i_{1}}(\iota_{j}')}, j=1,\dots,M$ are i.i.d. draws from the empirical distribution $|I_{n-1,r-1}|^{-1} \sum_{\iota' \in I_{n-1,r-1}(i_{1})} \delta_{X_{\iota'}}$. Mimicking the proof of Proposition \ref{prop:bootstrap_validity_incomplete_ustat}, it is not difficult to see that the conclusion of the proposition holds for the estimation of $g$ via sampling with replacement under the condition (\ref{eqn:growth_condition_incomplete_ustat}) (here $Z_{\iota'}'$ is the number of times that $\iota'$ is redrawn in the sample $\{ \iota_{1}',\dots,\iota_{M}' \}$, for which $\hat{g}^{(i_{1})} (x)$ can be expressed as $\hat{g}^{(i_{1})} (x) = M^{-1} \sum_{\iota' \in I_{n-1,r-1}} Z_{\iota'}' h(x,X_{\sigma_{i_{1}}(\iota')})$). 
\end{rmk}

\section{Numerical examples}
\label{sec:numerics}

In this section, we provide some numerical examples to verify the validity of our Gaussian approximation results and the proposed bootstrap algorithms (i.e., MB-DG, MB-NDG-DC, MB-NDG-RS) for approximating the distributions of incomplete $U$-statistics. In particular, we examine the statistical accuracy and computational running time of the Gaussian approximation and bootstrap algorithms in the leading example of testing for the pairwise independence of a high-dimensional vector. 

\subsection{Test statistics}
\label{subsec:test_stat_pairwise_independence}

In this section, we discuss several nonparametric statistics in the literature for the testing problem of the pairwise independence.

\begin{exmp}[Spearman's $\rho$]
\label{exmp:spearman_rho}
Let $\Pi_r$ be the collection of all possible permutations on $\{1,\dots,r\}$. \cite{hoeffding1948} shows that Spearman's rank correlation coefficient matrix $\rho$  can be written as 
\[
\rho = {n-2 \over n+1} \hat{\rho} + {3 \over n+1} \tau,
\]
where $\hat{\rho} = U_{n}^{(3)} (h^{S})$ is the $p \times p$ matrix-valued $U$-statistic associated with the kernel  
\[
\begin{split}
h^{S}(X_1, X_2, X_3) &= \left ( h_{j,k}^{S}(X_{1},X_{2},X_{3}) \right)_{1 \le j,k \le p} \\
&=  {1 \over 2} \sum_{\pi \in \Pi_3} \sign \left\{ (X_{\pi(1)} - X_{\pi(2)}) (X_{\pi(1)} - X_{\pi(3)})^T \right\},
\end{split}
\]
and $\tau = (\tau_{j,k})_{1 \le j,k \le p} = U_{n}^{(2)} (h^{K})$ is the $p \times p$ Kendall $\tau$ matrix with the kernel 
\[
h^{K}(X_1, X_2) = \sign \left\{ (X_1-X_2) (X_1-X_2)^T \right\}.
\]
Here, for a matrix $A=(a_{j,k})_{1 \le j,k \le p}$, $\sign\{A\}$ is the matrix of the same size as $A$ whose $(j,k)$-th element is $\sign (a_{j,k}) = \vone (a_{j,k} > 0) - \vone (a_{j,k} < 0)$.
It is seen that the leading term in Spearman's $\rho$ is $\hat{\rho}$, and so it is reasonable to reject the null hypothesis (\ref{eqn:independence_test}) if $\max_{1 \le j < k \le p} |\hat{\rho}_{j,k}|$ is large. Precisely speaking, this test is testing for a weaker  hypothesis that 
\[
H_{0}': \E[\sign (X_{1}^{(j)} - X_{2}^{(j)}) \sign (X_{1}^{(k)} - X_{3}^{(k)}) ] = 0 \ \text{for all} \ 1 \le j < k \le p.
\]
\end{exmp}

\begin{exmp}[Bergsma and Dassios' $t^{*}$]
\label{exmp:bergsma-dassios_t}
\cite{BergsmaDassios2014_Bernoulli} propose a $U$-statistic $t^{*} = (t_{j,k}^{*})_{1 \le j,k \le p} = U_{n}^{(4)}(h^{BD})$ of order 4  with the kernel 
\[
h^{BD} (X_{1},\dots,X_{4})= {1 \over 24} \sum_{\pi \in \Pi_4} \phi( X_{\pi(1)}, \dots,  X_{\pi(4)} ) \phi( X_{\pi(1)}, \dots,  X_{\pi(4)} )^T,
\]
where $\phi( X_{1}, \dots,  X_{4}) = \left (\phi_{j}(X_{1},\dots,X_{4}) \right)_{j=1}^{p}$ and 
\[
\begin{split}
&\phi_{j}(X_{1},\dots,X_{4}) \\
&= 
\vone(X_1^{(j)} \vee X_3^{(j)} < X_2^{(j)} \wedge X_4^{(j)}) + \vone(X_1^{(j)} \wedge X_3^{(j)} > X_2^{(j)} \vee X_4^{(j)}) \\
&\quad  -  \vone(X_1^{(j)} \vee X_2^{(j)} < X_3^{(j)} \wedge X_4^{(j)}) - \vone(X_1^{(j)} \wedge X_2^{(j)} > X_3^{(j)} \vee X_4^{(j)}).
\end{split}
\]
Under the assumption that $(X^{(j)},X^{(k)})$ has a bivariate distribution that
is discrete or (absolutely) continuous, or a mixture of both,
\cite{BergsmaDassios2014_Bernoulli} show that $\E[t_{j,k}^*] = 0$ if and only if $X^{(j)}$ and $X^{(k)}$ are independent, and so it is reasonable to reject the null hypothesis (\ref{eqn:independence_test}) if $\max_{1 \le j < k \le p} |t_{j,k}^{*}|$ is large (or $\max_{1 \le j < k \le p} t_{j,k}^{*}$ is large, since in general $\E[t_{j,k}^{*}] \ge 0$). 
\end{exmp}

\begin{exmp}[Hoeffding's $D$]
\label{exmp:hoeffding_D}
\cite{Hoeffding1948b_AoMS} proposes a $U$-statistic $D=(D_{j,k})_{1 \le j,k \le p} = U_{n}^{(5)}(h^{D})$ of order $5$ with the kernel 
\[
h^{D} (X_1,\dots,X_5) = {1 \over 120} \sum_{\pi \in \Pi_5} \phi( X_{\pi(1)}, \dots,  X_{\pi(5)}) \phi( X_{\pi(1)}, \dots,  X_{\pi(5)})^T,
\]
where $\phi( X_{1},\dots,X_{5}) = (\phi_{j}(X_{1},\dots,X_{5}))_{j=1}^{p}$ and $\phi_{j}(X_{1},\dots,X_{5}) 
=[ \vone(X_1^{(j)} \ge X_2^{(j)}) - \vone(X_1^{(j)} \ge X_3^{(j)}) ] [ \vone(X_1^{(j)} \ge X_4^{(j)}) - \vone(X_1^{(j)} \ge X_5^{(j)}) ]/4$.
Under the assumption that the joint distribution of
$(X^{(j)},X^{(k)})$ has continuous joint and marginal densities,
\cite{Hoeffding1948b_AoMS} shows that $\E[ D_{j,k}] = 0$ if and only if $X^{(j)}$ and $X^{(k)}$ are independent, and so it is reasonable to reject the null hypothesis (\ref{eqn:independence_test}) if $\max_{1 \le j < k \le p} |D_{j,k}|$ is large (or $\max_{1 \le j < k \le p} D_{j,k}$ is large, since in general $\E[D_{j,k}] \ge 0$). It is worth noting that Bergsma and Dassios' $t^{*}$ is an improvement on Hoeffding's $D$ since the former can characterize the pairwise independence under weaker assumptions on the distribution of $X$ than the latter. 
\end{exmp}

Here $h^{S}$ is non-degenerate, while $h^{BD}$ and $h^{D}$ are degenerate of order $1$ under $H_{0}$. The above testing problem is motivated from recent papers by \cite{leungdrton2017} and  \cite{hanchenliu2017}, which study testing for the null hypothesis
\[
H_{0}'': X^{(1)},\dots,X^{(p)} \ \text{are mutually independent},
\]
and develop tests based on functions of the $U$-statistics appearing in Examples \ref{exmp:spearman_rho}--\ref{exmp:hoeffding_D}. Note that $H_{0}''$ is a stronger hypothesis than $H_{0}$. Specifically, \cite{leungdrton2017} consider tests statistics such as, e.g., $S_{\hat{\rho}} = \sum_{1 \le j < k \le p}\hat{\rho}_{j,k}^{2} - 3\mu_{\hat{\rho}}$ with $\mu_{\hat{\rho}} = \E[\hat{\rho}_{1,2}^{2}]$ under $H_{0}''$ and show that $nS_{\hat{\rho}}/(9p\zeta_{1}^{\hat{\rho}}) \stackrel{d}{\to} N(0,1)$ under $H_{0}''$ as $(n,p) \to \infty$ where $\zeta_{1}^{\hat{\rho}} = \Var (\E[h_{1,2}^{S}(X_{1},X_{2},X_{3}] \mid X_{1}])$. 
On the other hand, \cite{hanchenliu2017} consider test statistics such as, e.g., $L_{n} = \max_{1 \le j < k \le p} | \hat{\rho}_{j,k} |$ and show that $L_{n}^{2}/\Var (\hat{\rho}_{1,2}) - 4\log p + \log \log p$ converges in distribution to a Gumbel distribution as $n \to \infty$ and $p=p_{n} \to \infty$ under $H_{0}''$ provided that $\log p = o(n^{1/3})$ (precisely speaking, \cite{hanchenliu2017} rule out degenerate kernels). Importantly, compared with the tests developed in \cite{leungdrton2017} and \cite{hanchenliu2017} based on analytical critical values, our bootstrap-based tests can directly detect the pairwise dependence for some pair of coordinates (or $\E[\sign (X_{1}^{(j)} - X_{2}^{(j)}) \sign (X_{1}^{(k)} - X_{3}^{(k)}) ] \neq 0$ for some $1 \le j < k \le p$ for Spearman's $\rho$) rather than the non-mutual-independence and also work for non-continuous random vectors (see, e.g., \cite{geissermantel1962} for interesting examples of pairwise independent but jointly dependent random variables; in particular, their examples include continuous random variables). In contrast, the derivations of the asymptotic null distributions in \cite{leungdrton2017} and \cite{hanchenliu2017} critically depend on the mutual independence between the coordinates of $X$. In addition, they both assume that $X$ is continuously distributed so that there are no ties in $X_{1}^{(j)},\dots,X_{n}^{(j)}$ for each coordinate $j$, thereby ruling out discrete components. It is worth noting that the $U$-statistics appearing Examples \ref{exmp:spearman_rho}--\ref{exmp:hoeffding_D} are rank-based,  and so if $X$ is continuous and $H_{0}''$ is true, then those $U$-statistics are pivotal, i.e., they have known (but difficult-to-compute) distributions, which is also a critical factor in their analysis; however, that is not the case under the weaker hypothesis of pairwise independence and without the continuity assumption on $X$ \cite{NandyWeihsDrton2016_EJS}.


In our simulation studies, we consider two test statistics: Spearman's $\rho$ and Bergsma-Dassios' $t^{*}$. Under $H_0$ in (\ref{eqn:independence_test}), the leading term $\hat{\rho}$ of  Spearman's $\rho$ is non-degenerate while Bergsma-Dassios' $t^{*}$ is degenerate of order 1, both having zero mean. Slightly abusing notation, we will use $\hat{\rho}$ as Spearman's $\rho$ statistic throughout this section. We consider tests of the forms 
\[
\max_{1 \le j < k \le p} | \hat{\rho}_{j,k}' | > c \Rightarrow \text{reject $H_{0}$} \quad \text{and} \quad \max_{1 \le j < k \le p} | {t_{j,k}^{*'}} | > c \Rightarrow  \text{reject $H_{0}$},
\]
where $\hat{\rho}_{j,k}'$ and $t_{j,k}^{*'}$ are incomplete versions of $\hat{\rho}_{j,k}$ and $t_{j,k}^{*}$, respectively, and their critical values are calibrated by the bootstrap methods. In particular, for any nominal size $\alpha \in (0,1)$, the value of $c := c(\alpha)$ can be chosen as the $(1-\alpha)$-th quantile of an appropriate bootstrap conditional distribution given $\calD_{n}$. For Spearman's $\rho$, we use $U_{n}^{\sharp}$ for MB-NDG-DC and MB-NDG-RS. For Bergsma-Dassios' $t^{*}$, we use $U_{n,B}^{\sharp}$ for MB-DG. In addition, we also test the performance of the {\it partial} versions of MB-NDG-DC and MB-NDG-RS (i.e., $U_{n,A}^{\sharp}$; cf. Corollary \ref{cor:bootstrap_validity}) for Spearman's $\rho$ statistic when its distribution can be approximated by $\gamma_{A} = N(0, r^2 \Gamma_{g})$ (cf. Corollary \ref{cor:Gaussian_approximation}).


\subsection{Simulation setup}

We simulate i.i.d. data from the non-central $t$-distribution with $3$ degrees of freedom and non-centrality parameter $2$. This data generating process implies $H_{0}$. We consider $n = 300, 500, 1000$ and $p = 30, 50, 100$ (so the number of the free parameters is $d=p(p-1)/2 = 435, 1225, 4950$).  For each setup $(n, p)$, we fix the bootstrap sample size $B = 200$ and report the empirical rejection probabilities of the bootstrap tests averaged over 2,000 simulations. For Spearman's $\rho$, we apply the MB-NDG-DC and MB-NDG-RS (full version $U_{n}^{\sharp}$) and set the computational budget parameter value $N = 2n$. In addition, we implement the MB-NDG-DC with the parameter values suggested in Section \ref{subsec:divide_and_conquer} (i.e., $S_{1}=\{ 1,\dots,n \}, L = r-1$, and $K = \lfloor (n-1) / L \rfloor$), and the MB-NDG-RS with the parameter values suggested in Section \ref{subsec:incomplete_ustat} (i.e., $S_{1}=\{ 1,\dots,n \}$ and $M = 2(n-1)$). For Bergsma-Dassios' $t^{*}$, we apply the MB-DG $U_{n,B}^{\sharp}$ with $N = n^{4/3}$. Moreover, we also apply the partial versions of MB-NDG-DC and MB-NDG-RS $U_{n,A}^{\sharp}$ with $N = 4 n^{3/2}$. These computational budget parameter values are chosen to minimize the rate in the error bounds of the corresponding Gaussian and bootstrap approximations. 
We only report the simulation results for the randomized incomplete $U$-statistic with the Bernoulli sampling since the simulation results for the sampling with replacement case are qualitatively similar. 

\subsection{Simulation results}

We first examine the statistical accuracy of the bootstrap tests in terms of size for $U_{n}^{\sharp}$ for the incomplete versions of Spearman's $\rho$ and $U_{n,B}^{\sharp}$ for Bergsma-Dassios' $t^*$. For each nominal size $\alpha \in (0,1)$, we denote by $\hat{R}(\alpha)$ the empirical rejection probability of the null hypothesis, where the critical values are calibrated by our bootstrap methods. The uniform errors-in-size on $\alpha \in [0.01, 0.10]$ of our bootstrap tests are summarized in Table \ref{tab:bootstrap_size_error_right_tail}. We observe that the bootstrap approximations become more accurate as $n$ increases, and they work quite well for small values of $\alpha$, which are relevant in the testing application. Due to the space concern, we defer the empirical size graph $\{(\alpha, \hat{R}(\alpha)) : \alpha \in (0,1)\}$ of the bootstrap tests for MB-NDG-DC (Spearman's $\rho$), MB-NDG-RS (Spearman's $\rho$), and MB-DG (Bergsma-Dassios' $t^*$) to Appendix \ref{sec:additional_simulation} in the SM. In addition, we also report the simulation results of the partial bootstrap $U_{n,A}^{\sharp}$ for Spearman's $\rho$ in Appendix \ref{sec:additional_simulation} in the SM. 

\begin{table}[h]
\begin{center}
\begin{tabular}{ c | c c c }
\hline
\multirow{2}{*}{Setup} & Spearman's $\rho$  & Spearman's $\rho$ & Bergsma-Dassios' $t^{*}$  \\
& (MB-NDG-DC) &  (MB-NDG-RS) & (MB-DG) \\
\hline
$p = 30, n = 300$ & 0.0080 & 0.0110 & 0.0280 \\
$p = 30, n = 500$ & 0.0065 & 0.0130 & 0.0225 \\
$p = 30, n = 1000$ & 0.0060 & 0.0055 & 0.0095 \\
$p = 50, n = 300$ & 0.0250 & 0.0135 & 0.0385 \\
$p = 50, n = 500$ & 0.0105 & 0.0035 & 0.0260 \\
$p = 50, n = 1000$ & 0.0145 & 0.0095 & 0.0235 \\
$p = 100, n = 300$ & 0.0180 & 0.0125 & 0.0660 \\
$p = 100, n = 500$ & 0.0135 & 0.0100 & 0.0290 \\
$p = 100, n = 1000$ & 0.0075 & 0.0020 & 0.0170 \\
\hline
\end{tabular}
\end{center}
\caption{Uniform error-in-size $\sup_{\alpha \in [0.01.0.10]} |\hat{R}(\alpha)-\alpha|$ of the bootstrap tests, where $\alpha$ is the nominal size.}
\label{tab:bootstrap_size_error_right_tail}
\end{table}

We also report the empirical performance of the Gaussian approximation for the test statistics. The P-P plots for Spearman's $\rho$ (i.e., $\sqrt{n} U_{n,N}'$ versus $N(0, r^{2} \Gamma_{g} + \alpha_{n} \Gamma_{h})$) and Bergsma-Dassios' $t^{*}$ (i.e., $\sqrt{N} U_{n,N}'$ versus $N(0, \Gamma_{h})$) are shown in Figure \ref{fig:gauss_approx_spearman_rho} and \ref{fig:gauss_approx_BD_t}, respectively. Similarly as the bootstrap approximations, Gaussian approximations become more accurate as $n$ increases. 

\begin{figure}[t!] 
   \centering
       \includegraphics[scale=0.74]{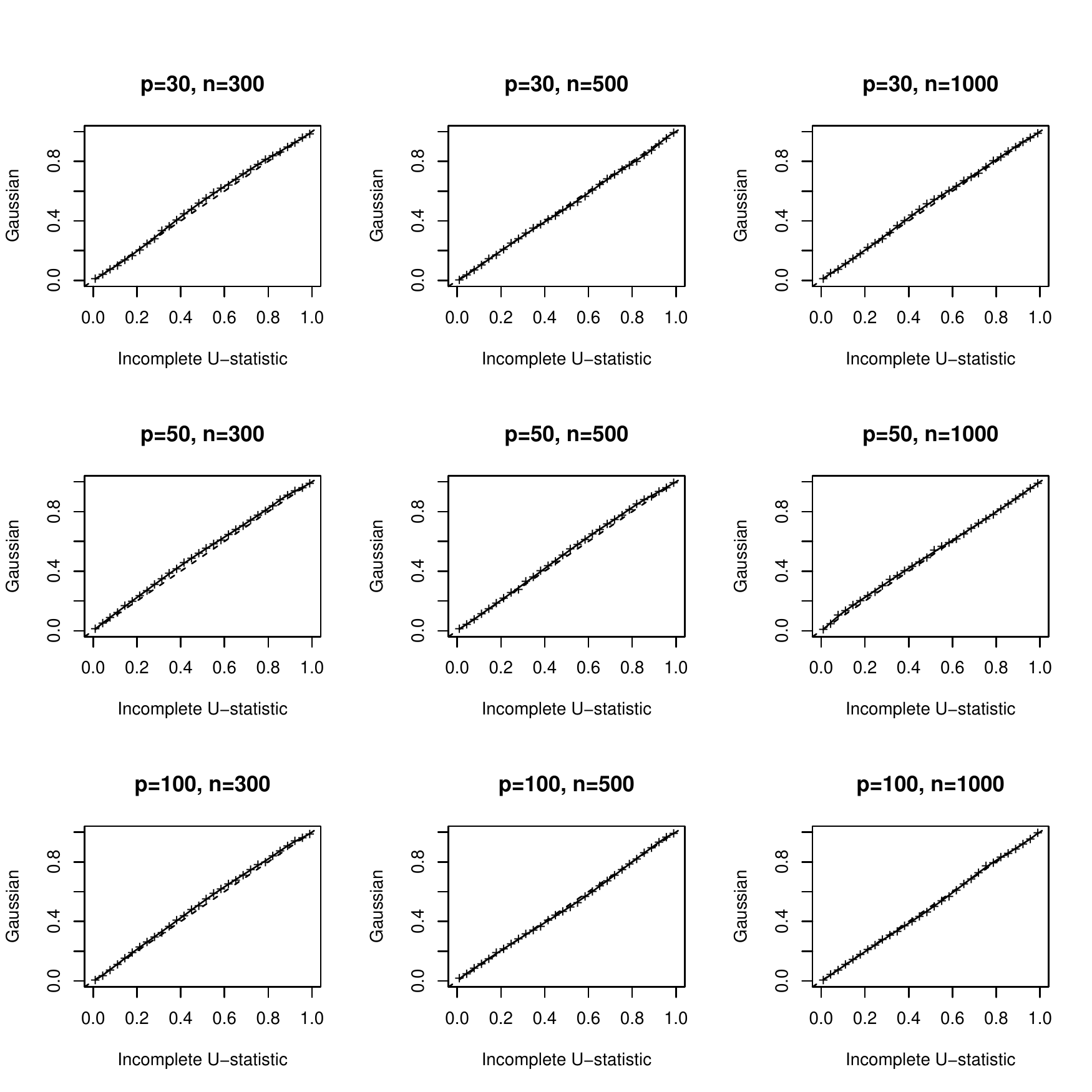}
    \caption{P-P plots for the Gaussian approximation $N(0, r^{2} \Gamma_{g} + \alpha_{n} \Gamma_{h})$ of $\sqrt{n} U_{n,N}'$ for Spearman's $\rho$ test statistic with the Bernoulli sampling.}
   \label{fig:gauss_approx_spearman_rho}
\end{figure}

\begin{figure}[t!] 
   \centering
       \includegraphics[scale=0.74]{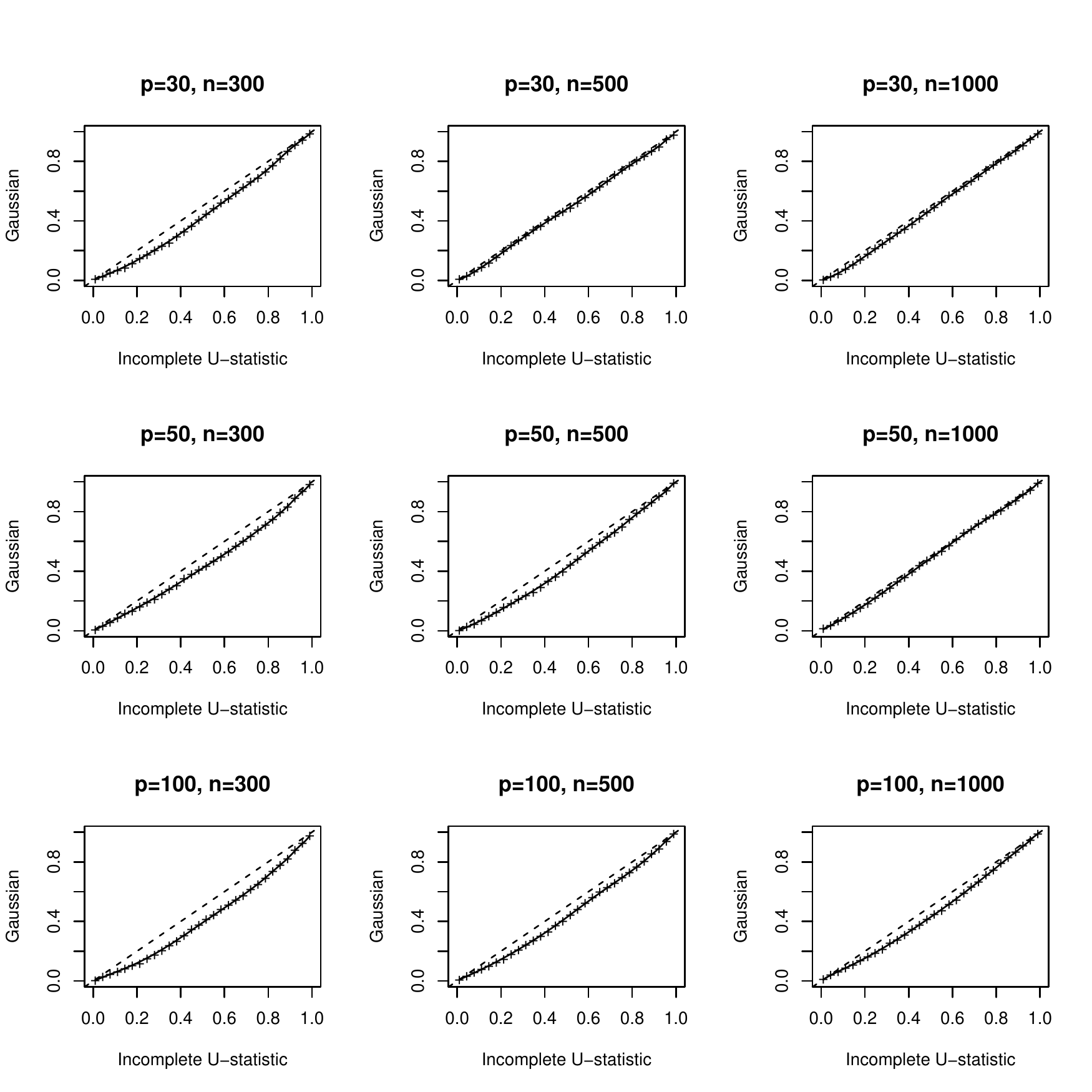}
    \caption{P-P plots for the Gaussian approximation $N(0, \Gamma_{h})$ of $\sqrt{N} U_{n,N}'$ for Bergsma-Dassios' $t^{*}$ test statistic with the Bernoulli sampling.}
   \label{fig:gauss_approx_BD_t}
\end{figure}

Next, we report the computer running time of the bootstrap tests. Figure \ref{fig:running_time_bootstrap} displays the computer running time versus the sample size, both on the log-scale. It is observed that the (log-)running time for the bootstrap methods scales linearly with the (log-)sample size. We further fit a linear model of the (log-)running time against the (log-)sample size (with the intercept term) for each $p$. For Spearman's $\rho$, the slope coefficient for $p=(30,50,100)$ is $(1.820, 1.863, 1.819)$ in the case MB-NDG-DC, and $(1.987, 1.874, 1.918)$ in the case MB-NDG-RS. In both cases, the slope coefficients are close to the theoretic value 2. Recall that the computational complexity for MB-NDG-DC and MB-NDG-RS is the same as $O((n+B)nd)$ for the suggested parameter values. For $n$ larger than $B$, the computational cost is approximately quadratic in $n$ for each $p$. For Bergsma-Dassios' $t^*$, the slope coefficient for $p=(30,50,100)$ is $(1.314 , 1.318, 1.316)$, which matches very well to the exponent $4/3$ of the computational budget parameter value $N = n^{4/3}$. In addition, the running time lines are in parallel with each other. This also makes sense because the computational costs of all the bootstrap methods are linear in $d$ (and thus quadratic in $p$) and the increase of $p$ only affects the intercept on the log-scale. 

\begin{figure}[t!] 
   \centering
       \includegraphics[scale=0.23]{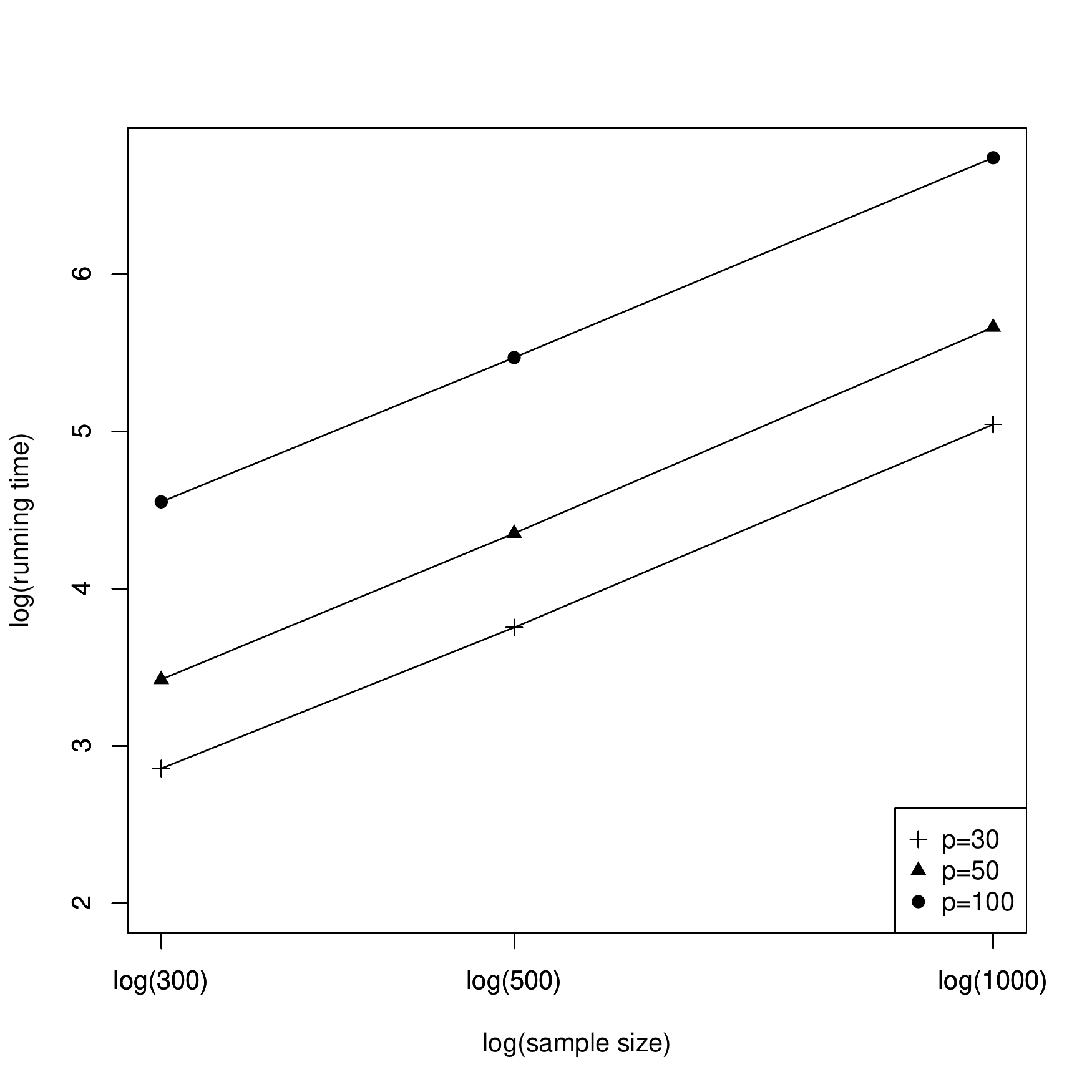}
       \includegraphics[scale=0.23]{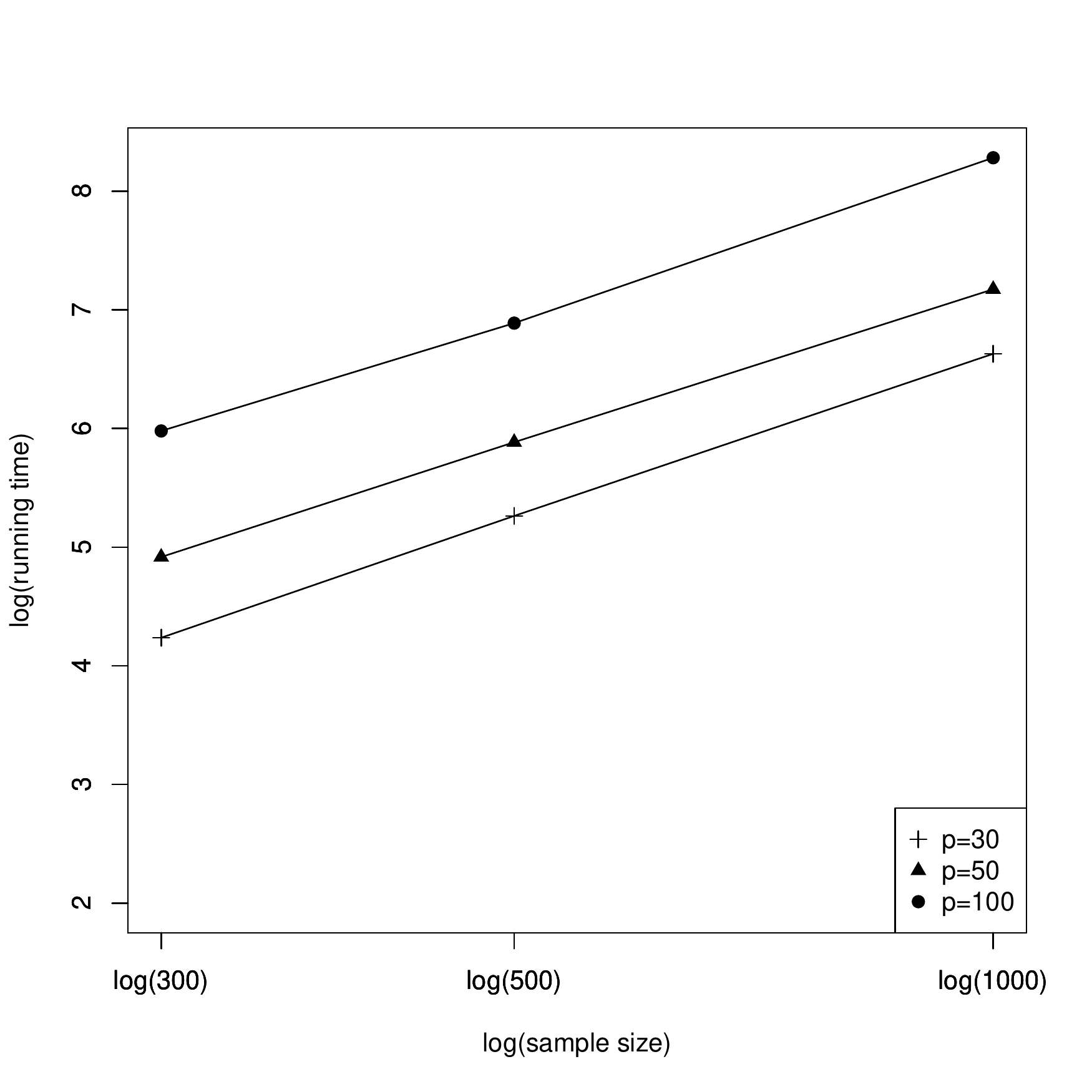}
       \includegraphics[scale=0.23]{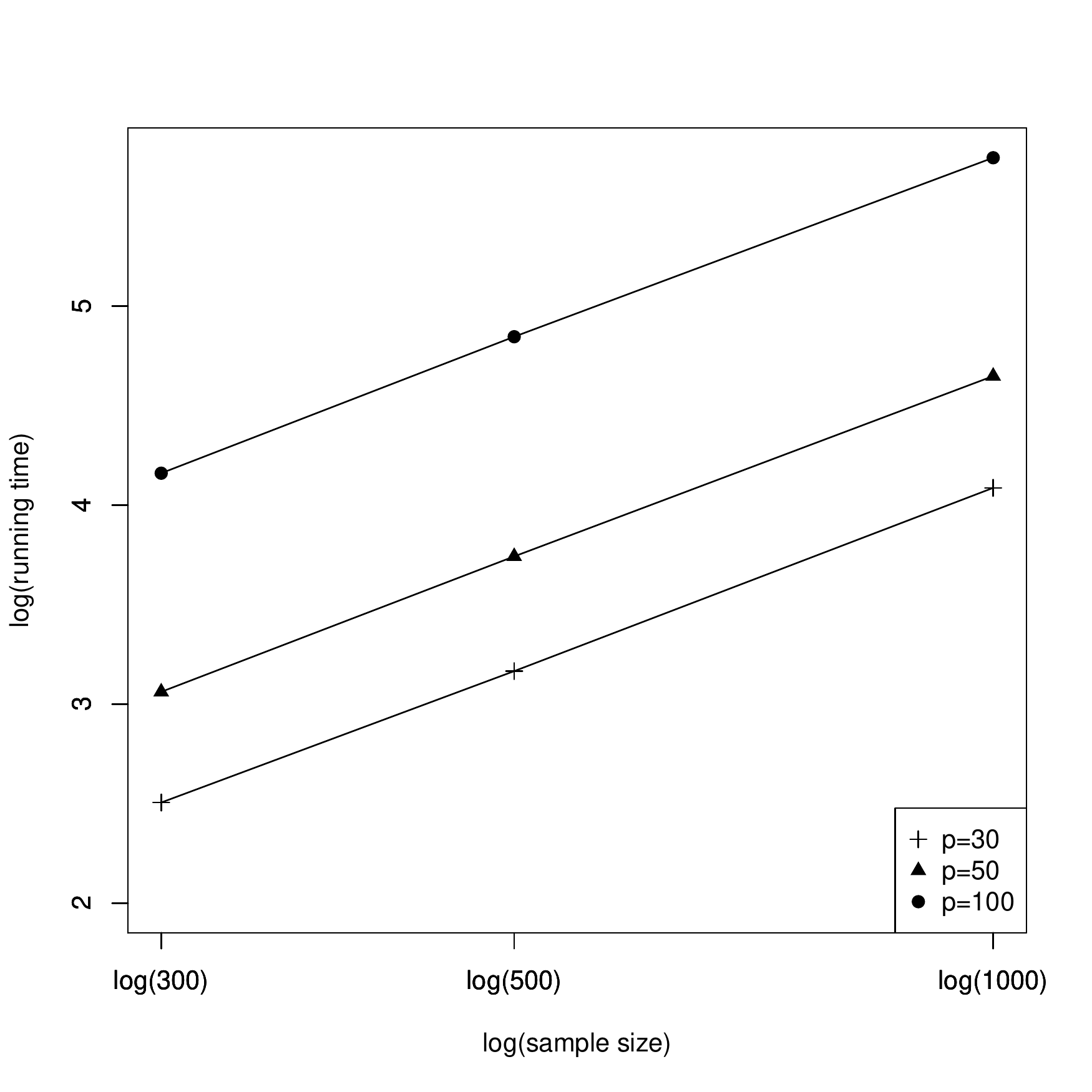}
    \caption{Computer running time of the bootstrap versus the sample size on the log-scale. Left: bootstrap $U_{n}^\sharp$ for Spearman's $\rho$ with the divide and conquer estimation (MB-NDG-DC). Middle: bootstrap $U_{n}^\sharp$ for Spearman's $\rho$ with the random sampling estimation (MB-NDG-RS). Right: bootstrap $U_{n,B}^\sharp$ for Bergsma-Dassios' $t^*$ (MB-DG).}
   \label{fig:running_time_bootstrap}
\end{figure}

\section{Discussions}
\label{sec:discussions}

In this paper, we have derived the Gaussian and bootstrap approximation results for  incomplete $U$-statistics with random and sparse weights in high dimensions. Specifically, we have considered two sampling schemes: Bernoulli sampling and sampling with replacement, both subject to a computational budget parameter to construct the random weights. On one hand, the sparsity in the design makes the computation of the incomplete $U$-statistics tractable. On the other hand, the randomness of the weights opens the possibility for us to obtain  unified Central Limit Theorem (CLT) type behaviors for both non-degenerate and degenerate kernels, thus revealing the fundamental difference between complete and randomized incomplete $U$-statistics. Building upon the Gaussian approximation results, we have developed novel bootstrap methods for incomplete $U$-statistics that take computational considerations into account, and established finite sample error bounds for the proposed bootstrap methods. Additional discussions on two extensions (extensions to normalized $U$-statistics and incomplete $U$-statistics with increasing orders) can be found in Section \ref{sec:additional discussions} of the SM.

\section*{Acknowledgments}
The authors would like to thank the anonymous referees, an Associate Editor, and the Editor for their constructive comments that improve the quality of this paper.

\begin{supplement}
\sname{Supplementary Material}\label{suppA}
\stitle{}
\sdescription{The Supplementary Material contains the proofs and additional discussions, simulation results, and applications of the main paper.}
\end{supplement}

\bibliographystyle{imsart-number}
\bibliography{icp_ustat}

\newpage

\renewcommand\thesection{\Alph{section}}
\setcounter{section}{0}


\begin{frontmatter}
\title{Supplementary Material to ``Randomized incomplete $U$-statistics in high dimensions"\thanksref{T1}}
\runtitle{Randomized incomplete $U$-statistics in high dimensions}
\thankstext{T1}{This version: \today.}

\begin{aug}
\author{\fnms{Xiaohui} \snm{Chen}\thanksref{m1}}
\and
\author{\fnms{Kengo} \snm{Kato}\thanksref{m2}}


\affiliation{University of Illinois at Urbana-Champaign\thanksmark{m1} and Cornell University\thanksmark{m2}}

\address{Department of Statistics\\
University of Illinois at Urbana-Champaign\\
725 S. Wright Street\\
Champaign, Illinois 61874\\
USA \\
\printead{e1}}

\address{Department of Statistical Science\\
Cornell University\\
1194 Comstock Hall \\
Ithaca, New York 14853\\
USA\\
\printead{e2}}
\end{aug}

\begin{abstract}
This Supplementary Material contains the additional discussions, bootstrap validity under the polynomial moment condition, proofs omitted from the main text, and additional simulation results of the paper ``Randomized incomplete $U$-statistics in high dimensions." 
\end{abstract}

%

\end{frontmatter}

\section{Additional discussions}
\label{sec:additional discussions}

\subsection{Comparison of MB-NDG-DC with BLB}
\label{sec:BLB}

Our MB-NDG-DC in Section \ref{subsec:divide_and_conquer} differs from the Bag of Little Bootstraps (BLB) proposed in \cite{KTSJ2014_JRSSB}, which is another generically scalable bootstrap method for large datasets based on the divide and conquer (DC) algorithm. Specifically, tailored to the $U$-statistic $U_{n} := U_{n}^{(r)}(h)$ with kernel $h$, let $Q_{n} := Q_{n}(P)$ be the distribution of $U_n$ and $\lambda(Q_{n}(P)) = \lambda (Q_{n}(P),P)$ be a quality assessment of $U_n$ (cf. Chapter 6.5 in \cite{Lehmann1998}). For instance, $\lambda(Q_{n}(P))$ can be the $95\%$-quantile of the distribution of $\max_{1 \le j \le d} \sqrt{n}(U_{n,j}-\theta_{j})$. A natural estimate of $\lambda(Q_{n}(P))$ is the plug-in estimate $\lambda(Q_{n}(\mathbb{P}_n))$, where $\mathbb{P}_n = n^{-1} \sum_{i=1}^n \delta_{X_i}$ is the empirical distribution of $X_1,\dots,X_n$. Typically, $\lambda(Q_{n}(\mathbb{P}_n))$ is difficult to compute, even for a moderate sample size $n$. The BLB first divides the original sample $\{ X_{1},\dots,X_{n} \}$ into $K$ subsets $\calI_{1},\dots,\calI_{K}$ of size $L$ uniformly at random. Denote by $\mathbb{P}^{(k)}_{n,L} = L^{-1} \sum_{i \in \calI_k} \delta_{X_{i}}$  the empirical distribution of $\{X_{i}\}_{i \in \calI_{k}}$. Then, on each subset $\calI_k, k=1,\dots,K$, the BLB  repeatedly resamples $n$ points i.i.d. from $\mathbb{P}_{n,L}^{(k)}$, computes the $U$-statistic with kernel $h$ for each resample, forms the empirical distribution $\mathbb{Q}_{n,k}^{*}$ of the computed $U$-statistics, 
and approximates $\lambda (Q_{n}(\mathbb{P}_{n,L}^{(k)}))$ by $\lambda(\mathbb{Q}_{n,k}^*)$. Finally the BLB takes the average $K^{-1} \sum_{k=1}^K \lambda(\mathbb{Q}_{n,k}^*)$ as an estimate of $\lambda(Q_{n}(P))$. The computational cost of the BLB is $O(B K L^{r} d) = O(B n L^{r-1} d)$. Note that the asymptotic validity of the BLB requires that $L \to \infty$ (cf. Theorem 1 of \cite{KTSJ2014_JRSSB}), so that $\mathbb{P}^{(k)}_{n,L}$ is close enough to $P$. Therefore, in order for the BLB to approach the population quality assessment value $\lambda(Q_{n}(P))$, its computational complexity has to depend on the order $r$ of the $U$-statistic. On the contrary, our MB-NDG-DC applies the DC algorithm to estimation of the H\'ajek projection and the overall computational cost is $O(n^2 d + B (N+n)d)$, which does not depend on $r$. In particular, the computational cost of the MB-NDG-DC is $O(n^2 d + Bnd)$ if we choose $N$ to be of the same order as $n$.

\subsection{Extension to normalized $U$-statistics}
\label{sec:normalized U-statistics}
In applications to, e.g., testing problems, if the variances of the coordinates of $U_{n,N}'$ are heterogeneous, it would be natural to normalize the incomplete $U$-statistic $U_{n,,N}'$ in such a way that  all the coordinates have approximately unit variance, and use a max-type test statistic of $U_{n,N}'$. Often, the coordinatewise variances are unknown and have to be estimated. 
From Theorems \ref{thm:gaussian_approx_bern_random_design} and \ref{thm:gaussian_approx_bern_random_design_degenerate}, in the non-degenerate case the approximate variance of the $j$-th coordinate of $\sqrt{n}(U_{n,N}'-\theta)$ is $\sigma_{j}^{2} :=\sigma_{A,j}^{2}  + \alpha_{n} \sigma_{B,j}^{2}$, where $\sigma_{A,j}^{2} := r^{2} P(g_{j}-\theta_{j})^{2}$ and $\sigma_{B,j}^{2} := P^{r} (h_{j} - \theta_{j})^{2}$, while in the degenerate case, the approximate variance of the $j$-th coordinate of $\sqrt{N}(U_{n,N}'-\theta)$ is $\sigma_{B,j}^{2}$. So, the problem boils down to estimating $\sigma_{A,j}^{2}$ and $\sigma_{B,j}^{2}$. To this end, we propose the following estimators: recall the setup in Section \ref{sec:generic_bootstrap} and define 
\[
\hat{\sigma}_{A,j}^{2} := \frac{r^{2}}{n_{1}} \sum_{i_{1} \in S_{1}} \{ \hat{g}_{j}^{(i_{1})}(X_{i_{1}}) - \breve{g}_{j} \}^{2} \quad \text{and} \quad \hat{\sigma}_{B,j}^{2} := \frac{1}{\hat{N}} \sum_{\iota \in I_{n,r}} Z_{\iota} \{ h_{j}(X_{\iota}) - U_{n,N,j}' \}^{2},
\]
where $\hat{N}$ is replaced by $N$ in the definition of $\hat{\sigma}_{B,j}^{2}$ for the sampling with replacement case. These estimators are the $(j,j)$-elements of the conditional covariance matrices of $U_{n,A}^{\sharp}$ and $U_{n,B}^{\sharp}$, respectively. Note that the computational cost to construct $\hat{\sigma}_{B,j}^{2},j=1,\dots,d$ is (on average) $O(Nd)$, while that of $\hat{\sigma}_{A,j}^{2},j=1,\dots,d$ is $O(n^{2}d)$ if the DC estimation with the parameter values suggested in Section \ref{subsec:divide_and_conquer} is used for estimation of $g$. Now, let $\Lambda_{A} = \diag \{ \sigma_{A,1}^{2},\dots,\sigma_{A,d}^{2} \}, \Lambda_{B} = \diag \{ \sigma_{B,1}^{2},\dots,\sigma_{B,d}^{2} \}, \hat{\Lambda}_{A} = \diag \{ \hat{\sigma}_{A,1}^{2},\dots,\hat{\sigma}_{A,d}^{2} \}, \Lambda_{B} = \diag \{ \hat{\sigma}_{B,1}^{2},\dots,\hat{\sigma}_{B,d}^{2} \}, \Lambda = \diag \{ \sigma_{1}^{2},\dots,\sigma_{d}^{2} \} = \Lambda_{A}+\alpha_{n}\Lambda_{B}$, and $\hat{\Lambda} = \diag \{ \hat{\sigma}_{1}^{2},\dots,\hat{\sigma}_{d}^{2} \} =  \hat{\Lambda}_{A} + \alpha_{n} \hat{\Lambda}_{B}$. 
We consider to approximate the distributions of $\sqrt{n}\hat{\Lambda}^{-1/2} (U_{n,N}'-\theta)$ in the non-degenerate case and $\sqrt{N}\hat{\Lambda}_{B}^{-1/2} (U_{n,N}'-\theta)$ in the degenerate case. 
Recall the setup in Section \ref{sec:generic_bootstrap}.

\begin{cor}[Gaussian and bootstrap approximations to normalized incomplete $U$-statistics]
\label{cor:normalized_U_statistics}
(i) Suppose that Conditions (C1), (C2), and (C3-ND) hold, and in addition suppose that  Condition (\ref{eqn:growth_condition}) together with $D_{n}^{2}(\log^{7}(dn))/(n \wedge N) \le C_{1}n^{-(\zeta_1 \wedge \zeta_2)}$ hold for some constants $0 < C_{1} < \infty$ and $\zeta_1,\zeta_2 \in (0,1)$. Then there exists a constant $C$ depending only on $\underline{\sigma},r$, and $C_{1}$ such that
\[
\begin{split}
&\sup_{R \in \calR} \left | \Prob (\sqrt{n}\hat{\Lambda}^{-1/2} (U_{n,N}'-\theta) \in R) - \Prob (\Lambda^{-1/2}Y \in R) \right | \le Cn^{-(\zeta_1 \wedge \zeta_2)/6} \quad \text{and} \\
&\Prob \left \{ \sup_{R \in \calR} \left | \Prob_{\mid \calD_{n}} (\hat{\Lambda}^{-1/2} U_{n}^{\sharp} \in R) - \Prob (\Lambda^{-1/2}Y \in R) \right | > Cn^{-(\zeta_1 \wedge \zeta_2)/6} \right \} \le Cn^{-1},
\end{split}
\]
where $Y \sim N(0,r^{2}\Gamma_{g} + \alpha_{n} \Gamma_{h})$. 

(ii) Suppose that Conditions (C1), (C2), and (C3-D) hold, and in addition suppose that Condition (\ref{eqn:growth_condition_B}) holds for some constants $0 < C_{1} < \infty$ and $\zeta \in (0,1)$. Then there exists a constant $C$ depending only on $\underline{\sigma},r$, and $C_{1}$ such that
\[
\sup_{R \in \calR} \left | \Prob_{\mid \calD_{n}} (\hat{\Lambda}_{B}^{-1/2} U_{n,B}^{\sharp} \in R) - \gamma_{B}^{\dagger} (R) \right | \le Cn^{-\zeta/6}
\]
with probability at least $1-Cn^{-1}$, where $\gamma_{B}^{\dagger} = N(0,\Lambda_{B}^{-1/2}\Gamma_{h}\Lambda_{B}^{-1/2})$. If, in addition, the kernel $h$ is degenerate of order $k-1$ for some $k=2,\dots,r$, and if $ND_{n}^{2}(\log^{k+3}d)/n^{k} \le C_{1}n^{-2\zeta/3}$ and $D_{n}^{2}(\log^{7}(dn))/N\le C_{1}n^{-\zeta}$,
then there exists a constant $C'$ depending only on $\underline{\sigma},r$, and $C_{1}$ such that
\[
\sup_{R \in \calR} \left | \Prob (\sqrt{N}\hat{\Lambda}_{B}^{-1/2}(U_{n,N}'-\theta) \in R) - \gamma_{B}^{\dagger}(R) \right | \le C'n^{-\zeta/6}.
\]
\end{cor}




\subsection{Incomplete $U$-statistics with increasing orders} 
Finally, it is interesting to note a connection of incomplete $U$-statistics with machine learning. 
The recent paper by \cite{mentchhooker2016} studies asymptotic theory for one-dimensional incomplete $U$-statistics with increasing orders (i.e., $r = r_{n} \to \infty$). Specifically, they use sampling with replacement and establish asymptotic normality for the non-degenerate case. Their motivation is coming from uncertainty quantification for subbagging and (subsampled) random forests, which, from a mathematical point of view, are defined as \textit{infinite order} $U$-statistics \citep{frees1989} where the order of the $U$-statistics corresponds to the subsample size for a single tree and so $r=r_{n} \to \infty$. Since exact computation of subbagging and random forests is in most cases intractable, a common practice is to choose a smaller number of subsamples randomly. Building on the asymptotic normality result, \cite{mentchhooker2016} develop pointwise confidence intervals for subbagging and random forests; see also \cite{wagerathey2017} for related results. Extending the results of \cite{mentchhooker2016} to high dimensions enables us to develop methods to construct simultaneous confidence bands for subbagging and random forests and hence would be an interesting venue for future research. Such extension is by no means trivial since the constants appearing in the error bounds developed in the present paper depend on the order $r$ in complicated ways.

\section{Bootstrap validity under the polynomial moment condition}
\label{sec:bootstrap_validity_poly}

In this section, we present the bootstrap validity under the polynomial moment condition (C2'). Recall the assumption  
\begin{enumerate}
\item[(C2')] $(P^{r}|h|_{\infty}^{q})^{1/q} \le D_{n}$ for some $q \in [4,\infty)$,
\end{enumerate}
and the definition $\hat{\Delta}_{A,1}:= \max_{1 \le j \le d} n_{1}^{-1} \sum_{i_{1} \in S_{1}}  \{ \hat{g}_{j}^{(i_{1})}(X_{i_{1}}) - g_{j}(X_{i_{1}}) \}^{2}$.

\begin{thm}[Validity of $U_{n,B}^{\sharp}$]
\label{thm:bootstrap_validity_B_poly}
Suppose that (C1), (C2'), and (C3-D) hold. If  
\begin{equation}
\label{eqn:growth_condition_B_poly}
\frac{D_{n}^{2} \log^{5}(dn)}{n \wedge N} \le C_{1}n^{-\zeta} \quad \text{and} \quad {D_{n}^{2} n^{2\max(r+1,4)/q} \log^{3}(dn) \over n \wedge N} \le C_{1} n^{-\zeta/2}
\end{equation}
for some constants $0 < C_{1} < \infty$ and  $\zeta \in (0,1)$, then there exists a constant $C$ depending only on $\underline{\sigma}, r, q$, and $C_{1}$ such that 
\[
\sup_{R \in \calR} \left | \Prob_{\mid \calD_{n}} (U_{n,B}^{\sharp} \in R) - \gamma_{B}(R) \right | \le Cn^{-\zeta/6}
\]
with probability at least $1-Cn^{-1}$.
\end{thm}

\begin{thm}[Generic bootstrap validity under non-degeneracy]
\label{thm:bootstrap_validity_poly}
Let $U_{n}^{\sharp} = U_{n,A}^{\sharp} + \alpha_{n}^{1/2} U_{n,B}^{\sharp}$. 
Suppose that Conditions (C1), (C2'), and (C3-ND) hold.  In addition, suppose that 
\begin{equation}
\label{eqn:growth_condition_poly}
\begin{split}
&\frac{D_{n}^{2} \log^{5}(dn)}{n_{1} \wedge N} \le C_{1}n^{-\zeta_{1}}, \quad  {D_{n}^{2} n^{2\max(r+1,4)/q} \log^{3}(dn) \over n_{1} \wedge N} \le C_{1} n^{-\zeta_{1}/2}, \quad \text{and} \\
&\Prob \left ( \overline{\sigma}_{g}^{2} \hat{\Delta}_{A,1}\log^{4} d > C_{1}n^{-\zeta_{2}} \right ) \le C_{1}n^{-1}
\end{split}
\end{equation}
for some constants $0 < C_{1} < \infty$ and $\zeta_{1},\zeta_{2} \in (0,1)$, where $\overline{\sigma}_{g} := \max_{1 \le j \le d} \sqrt{P(g_{j}-\theta_{j})^{2}}$. Then there exists a constant $C$ depending only on $\underline{\sigma}, r, q$, and $C_{1}$ such that 
\begin{equation}
\label{eqn:bootstrap_validity_poly}
\sup_{R \in \calR} \left | \Prob_{\mid \calD_{n}} (U_{n}^{\sharp} \in R) - \Prob (Y \in R) \right | \le Cn^{-(\zeta_{1} \wedge \zeta_{2})/6}
\end{equation}
with probability at least $1-Cn^{-1}$, where $Y \sim N(0,r^{2}\Gamma_{g}+\alpha_{n}\Gamma_{h})$. If the estimates $g^{(i_{1})}, i_{1} \in S_{1}$ depend on an additional randomization independent of $\calD_{n}, \{ \xi_{i_{1}} : i_{1} \in S_{1} \}$, and $\{ \xi_{\iota}': \iota \in I_{n,r} \}$, then the result (\ref{eqn:bootstrap_validity_poly}), with $\calD_{n}$ replaced by the augmentation of $\calD_{n}$ with variables used in the additional randomization,  holds with probability at least $1-Cn^{-1}$. 
\end{thm}

The following two propositions are concerned with extensions of Propositions \ref{prop:bootstrap_validity_DC} and \ref{prop:bootstrap_validity_incomplete_ustat} under the polynomial moment condition (C2'). 

\begin{prop}[Validity of bootstrap with DC estimation]
\label{prop:bootstrap_validity_DC_poly}
Consider the MB-NDG-DC defined in Section \ref{subsec:divide_and_conquer}. 
Suppose that Conditions (C1), (C2'), and (C3-ND) hold.  In addition, suppose that 
\[
\begin{split}
&\frac{D_{n}^{2}(\log^{2}n) \log^{5}(dn)}{n_{1} \wedge N}  \bigvee \left \{ \frac{\overline{\sigma}_{g}^{2}D_{n}^{2} \log^{5}d}{KL} \left ( 1+\frac{\log^{2} d}{K^{1-2/q}} \right ) \right \}  \le C_{1}n^{-\zeta} \quad \text{and} \\
&{D_{n}^{2} n^{2\max(r+1,4)/q} \log^{3}(dn) \over n_1 \wedge N} \le C_{1} n^{-\zeta/2}
\end{split}
\]
for some constants $0 < C_{1} < \infty$ and $\zeta \in (2/q,1)$. Then there exists a constant $C$ depending only on $\underline{\sigma},r, q$, and $C_{1}$ such that the result (\ref{eqn:bootstrap_validity_poly}) with $(\zeta_1,\zeta_2) = (\zeta, \zeta-2/q)$ holds
with probability at least $1-Cn^{-1}$.
\end{prop}

\begin{prop}[Validity of bootstrap with Bernoulli sampling estimation]
\label{prop:bootstrap_validity_incomplete_ustat_poly}
Consider the MB-NDG-RS defined in Section \ref{subsec:incomplete_ustat}.
Suppose that Conditions (C1), (C2'), and (C3-ND) hold.  In addition, suppose that 
\[
\begin{split}
&\frac{D_{n}^{2}(\log^{2} n) \log^{5}(dn)}{n_{1} \wedge N} \bigvee \frac{\overline{\sigma}_{g}^{2}D_{n}^{2} \log^{5}d}{n} \bigvee \frac{\overline{\sigma}_{g}^{2}D_{n}^{2}n^{2(r-1)/q}(\log d)^{5+2/q}}{M}   \bigvee \frac{\overline{\sigma}_{g}^{2}D_{n}^{2} n^{2r/q} \log^{6}d}{M^{2}} \le C_{1}n^{-\zeta}  \quad \text{and} \\
&{D_{n}^{2} n^{2\max(r+1,4)/q} \log^{3}(dn) \over n_1 \wedge N} \le C_{1} n^{-\zeta/2}
\end{split}
\]
for some constants $0 < C_{1} < \infty$ and  $\zeta \in (2/q,1)$. Then there exists a constant $C$ depending only on $\underline{\sigma},r,q$, and $C_{1}$ such that result (\ref{eqn:bootstrap_validity_poly}), with $\calD_{n}$ replaced by $\calD_{n}' = \calD_{n} \cup \{ Z_{\iota'}' : \iota' \in I_{n-1,r-1} \}$ and with $(\zeta_1,\zeta_2) = (\zeta,\zeta-2/q)$, holds
with probability at least $1-Cn^{-1}$.
\end{prop}

\section{Proofs}
\label{sec:proofs}

\subsection{Preliminary lemmas}

This section collects some useful lemmas that will be used in the subsequent proofs. 
We will freely use the following maximal inequalities for the $\psi_{\beta}$-norms. 

\begin{lem}[Maximal inequalities for the $\psi_{\beta}$-norms]
\label{lem:max_ineq_psi_norms}
Let $\xi_{1},\dots,\xi_{k}$ be real-valued random variables such that $\| \xi_{i} \|_{\psi_{\beta}} < \infty$ for all $i=1,\dots,k$ for some $0 < \beta < \infty$, where $k \ge 2$. Then
\[
\left \| \max_{1 \le i \le k} | \xi_{i} | \right \|_{\psi_{\beta}} \le C_{\beta} (\log k)^{1/\beta} \max_{1 \le i \le k} \| \xi_{i} \|_{\psi_{\beta}},
\]
where $C_{\beta}$ is a constant that depends only on $\beta$. 
\end{lem}

\begin{proof}[Proof of Lemma \ref{lem:max_ineq_psi_norms}]
For $\beta \ge 1$, the lemma follows from Lemma 2.2.2 in \cite{vandervaartwellner1996}. For $\beta \in (0,1)$, $\psi_{\beta}$ is not convex and so we can not directly apply Lemma 2.2.2 in \cite{vandervaartwellner1996}, but apply the lemma for the norm equivalent to $\| \cdot \|_{\psi_{\beta}}$ obtained by linearizing $\psi_{\beta}$ in a neighborhood of the origin; see Lemma \ref{lem: quasi-norm} below. 
\end{proof}

\begin{lem}[Norm equivalent to $\| \cdot \|_{\psi_{\beta}}$]
\label{lem: quasi-norm}
Let $\beta \in (0,1)$, and take $x_{\beta} > 0$ large enough so that the function 
\[
\tilde{\psi}_{\beta}(x) = 
\begin{cases}
\psi_{\beta}(x) & \text{if $x \ge x_{\beta}$} \\
\frac{\psi_{\beta}(x_{\beta})}{x_{\beta}} x & \text{if $0 \le x \le x_{\beta}$}
\end{cases}
\]
is convex. Then there exists a constant $1  <C_{\beta} < \infty$ depending only on $\beta$ such that 
\[
C_{\beta}^{-1} \| \xi \|_{\tilde{\psi}_{\beta}} \le \| \xi \|_{\psi_{\beta}} \le C_{\beta} \| \xi \|_{\tilde{\psi}_{\beta}}
\]
 for every real-valued random variable $\xi$. 
\end{lem}

\begin{proof}[Proof of Lemma \ref{lem: quasi-norm}]
This  seems to be well-known, but we include a proof of the lemma  since we could not find a right reference. 
In this proof, the notation $\lesssim$ signifies that the left hand side is bounded by the right hand side up to a constant that depends  only on $\beta$. We first show that $\| \xi \|_{\tilde{\psi}_{\beta}} \lesssim \| \xi \|_{\psi_{\beta}}$.  To this end, we may assume that $\| \xi \|_{\psi_{\beta}} = 1$, i.e., $\E[ \psi_{\beta}(|\xi|)] = 1$, and show that $\| \xi \|_{\tilde{\psi}_{\beta}} \lesssim 1$. By Taylor's theorem, we have $\psi_{\beta}(x) \gtrsim x$ and $\psi_{\beta}(x/C) \le C^{-\beta} \psi_{\beta}(x)$ for $C>1$, so that 
\[
\E[\tilde{\psi}_{\beta}(|\xi/C|)] \lesssim \E[|\xi/C|] + \E[\psi_{\beta}(|\xi/C|)] \lesssim C^{-\beta}.
\]
This implies that $\| \xi \|_{\tilde{\psi}_{\beta}} \lesssim 1$. Next, suppose that $\| \xi \|_{\tilde{\psi}_{\beta}} = 1$ and we show that $\| \xi \|_{\psi_{\beta}} \lesssim 1$. By convexity of $\tilde{\psi}_{\beta}$, we have $\E[ \tilde{\psi}_{\beta}(|\xi/C|)] \le C^{-1}$ for $C>1$. Combining the fact that $\psi_{\beta}(x/C) \le C^{-\beta}\psi_{\beta}(x_{\beta})$ for $0 \le x \le x_{\beta}$ and $C> 1$, we have 
\[
\E[\psi_{\beta}(|\xi/C|)] \le C^{-\beta} \psi_{\beta}(x_{\beta}) + \E[\tilde{\psi}_{\beta}(|\xi/C|)] \lesssim C^{-\beta},
\]
which implies that $\| \xi \|_{\psi_{\beta}} \lesssim 1$. This completes the proof.
\end{proof}

\begin{lem}[Useful maximal inequalities for  $U$-statistics: sub-exponential moment]
\label{lem:talagrand_ineq_ustat}
Let $X_{1},\dots,X_{n}$ be i.i.d. random variables taking values in a measurable space $(S,\calS)$ with common distribution $P$, and let $h = (h_{1},\dots,h_{d})^{T}: S^{r} \to \R^{d}$ be a symmetric and jointly measurable function such that $\| h_{j}(X_{1}^{r}) \|_{\psi_{\beta}} < \infty$ for all $j=1,\dots,d$ for some $\beta \in (0,1]$. Consider the associated $U$-statistic $U_{n} (h) = |I_{n,r}|^{-1} \sum_{\iota \in I_{n,r}} h(X_{\iota})$ with kernel $h$, and let $\mathsf{Z} = \max_{1 \le j \le d} | U_{n}(h_{j}) - P^{r}h_{j}|$. In addition, let 
\begin{align*}
\breve{\mathsf{Z}} =& \max_{1 \le j \le d} \left| \sum_{i=1}^{m} \{ h_j(X_{(i-1)r+1}^{ir}) - P^{r}h_{j} \} \right|, \ \text{and} \\
\mathsf{M} = & \max_{1 \le i \le m} \max_{1 \le j \le d} |h_j(X_{(i-1)r+1}^{ir}) - P^{r}h_{j}|,
\end{align*}
where $m=\lfloor n/r \rfloor$ is the integer part of $n/r$. Then, for every $\eta \in (0,1]$ and $\delta > 0$, there exists a constant $C$ depending only on $\beta, \eta$, and $\delta$ such that 
\[
\Prob \left (m \mathsf{Z} \ge (1+\eta) \E[\breve{\mathsf{Z}}] + t \right) \le \exp\left( -{t^2 \over 2(1+\delta) m \sigma^{2}} \right) + 3 \exp\left \{ -\left( {t \over C \|\mathsf{M} \|_{\psi_\beta}} \right)^{\beta} \right \}
\]
for every $t > 0$, where $\sigma^{2} = \max_{1 \le j \le d} P^{r}(h_{j} - P^{r}h_{j})^{2}$. 
\end{lem}

\begin{proof}[Proof of Lemma \ref{lem:talagrand_ineq_ustat}]
The proof essentially follows from that of Lemma E.1 in \cite{chen2017a}, and so we only point out required modifications. The difference is that in Lemma E.1 in \cite{chen2017a}, $\breve{\mathsf{Z}}$ is defined as \\
$\max_{1 \le j \le d} | \sum_{i=1}^{m} \{ \overline{h}_j(X_{(i-1)r+1}^{ir}) - P^{r}\overline{h}_{j} \} |$ where $\overline{h}$ is to be defined below. 
Without loss of generality, we may assume that $t \ge C_{1}\| \mathsf{M} \|_{\psi_{\beta}}$ for some sufficiently large constant $C_{1}$ that depends only on $\beta,\eta$ and $\delta$. 
For $\tau = 8\E[ \mathsf{M}]$, let $\overline{h}(x_{1},\dots,x_{r}) = h(x_{1},\dots,x_{r}) \vone (|h (x_{1},\dots,x_{r})|_{\infty} \le \tau)$ and $\underline{h} = h - \overline{h}$. In addition, define $V_{\ell} (x_{1},\dots,x_{n}), T_{\ell}, \ell=1,2$ as in the proof of Lemma E.1 in \cite{chen2017a}. Then, $\mathsf{Z} \le T_{1}+T_{2}$, and since $h = \overline{h} + \underline{h}$ and hence $\overline{h} = h + (-\underline{h})$, we have  $\E[ |W_{1}(X_{1}^{n}) |_{\infty}] \le \E[\breve{\mathsf{Z}}] + \E[| W_{2}(X_{1}^{n}) |_{\infty}]$, so that $\E[\breve{\mathsf{Z}}] \ge \E[|W_{1}(X_{1}^{n})|_{\infty}] - \E[|W_{2}(X_{1}^{n})|_{\infty}]$. Hence, for every $\eta > 0$ and $\varepsilon \in (0,1)$,
\[
\begin{split}
&\Prob \left(\mathsf{Z} \ge (1+\eta) \E[\breve{\mathsf{Z}}] + t \right) \\
&\quad \le \Prob \left ( T_{1} \ge (1+\eta) (\E[|W_{1}(X_{1}^{n})|_{\infty}] - \E[|W_{2}(X_{1}^{n})|_{\infty}]) + (1-\varepsilon)t \right ) + \Prob (T_{2} \ge \varepsilon t). 
\end{split}
\]
Choose $\varepsilon  = \varepsilon (\delta) < 1/2$ small enough so that $(1-2\varepsilon)^{-{2}} (1+\delta/2) \le 1+\delta$. 
From the proof of Lemma E.1 in \cite{chen2017a}, we have $\E[|W_{2}(X_{1}^{n})|_{\infty}] \le C_{2} \| \mathsf{M} \|_{\psi_{\beta}}$ for some constant $C_{2}$ that depends only on $\beta$. 
By choosing $C_{1}$ sufficiently large, we have  $(1+\eta)C_{2} \| \mathsf{M} \|_{\psi_{\beta}} \le \varepsilon t$, so that 
\[
\Prob \left (\mathsf{Z} \ge (1+\eta) \E[\breve{\mathsf{Z}}] + t\right)  \le \Prob \left ( T_{1} \ge (1+\eta) \E[|W_{1}(X_{1}^{n})|_{\infty}]  + (1-2\varepsilon)t \right ) + \Prob (T_{2} \ge \varepsilon t).
\]
The rest of the proof is analogous to the proof of Lemma E.1 in \cite{chen2017a} and hence omitted. 
\end{proof}

\begin{lem}[Useful maximal inequalities for  $U$-statistics: polynomial moment]
\label{lem:talagrand_ineq_ustat_polymom}
Let $X_{1},\dots,X_{n}$ be i.i.d. random variables taking values in a measurable space $(S,\calS)$ with common distribution $P$, and let $h = (h_{1},\dots,h_{d})^{T}: S^{r} \to \R^{d}$ be a symmetric and jointly measurable function such that $P^{r} |h_{j}|^{q} < \infty$ for all $j=1,\dots,d$ for some $q \in [1,\infty)$. Consider the associated $U$-statistic $U_{n} (h) = |I_{n,r}|^{-1} \sum_{\iota \in I_{n,r}} h(X_{\iota})$ with kernel $h$. Let $\mathsf{Z}$, $\breve{\mathsf{Z}}$, and $\mathsf{M}$ have the same definition as in Lemma \ref{lem:talagrand_ineq_ustat}. Then, for every $\eta \in (0,1]$ and $\delta > 0$, there exists a constant $C$ depending only on $q, \eta$, and $\delta$ such that 
\[
\Prob \left (m \mathsf{Z} \ge (1+\eta) \E[\breve{\mathsf{Z}}] + t \right) \le \exp\left( -{t^2 \over 2(1+\delta) m \sigma^{2}} \right) + {C\E[M^{q}] \over t^{q}}
\]
for every $t > 0$, where $\sigma^{2} = \max_{1 \le j \le d} P^{r}(h_{j} - P^{r}h_{j})^{2}$. 
\end{lem}

\begin{proof}[Proof of Lemma \ref{lem:talagrand_ineq_ustat_polymom}]
The proof follows from that of Lemma E.2 in \cite{chen2017a} and a similar modification argument of Lemma \ref{lem:talagrand_ineq_ustat}. Details are omitted. 
\end{proof}

\begin{lem}[Gaussian comparison on hyperrectangles]
\label{lem:Gaussian_comparison}
Let $Y,W$ be centered Gaussian random vectors in $\R^{d}$ with covariance matrices $\Sigma^{Y} = (\Sigma_{j,k}^{Y})_{1 \le j,k \le d}, \Sigma^{W} = (\Sigma_{j,k}^{W})_{1 \le j,k \le d}$, respectively, and let $\Delta = | \Sigma^{Y} - \Sigma^{W} |_{\infty}$. Suppose that $\min_{1 \le j \le d} \Sigma_{j,j}^{Y} \bigvee \min_{1 \le j \le d} \Sigma_{j,j}^{W} \ge \underline{\sigma}^{2}$ for some constant $\underline{\sigma} > 0$. Then
\[
\sup_{R \in \calR} | \Prob (Y \in R) - \Prob (W \in R) | \le C \Delta^{1/3} \log^{2/3} d,
\]
where $C$ is a constant that depends only on $\underline{\sigma}$.
\end{lem}

\begin{proof}[Proof of Lemma \ref{lem:Gaussian_comparison}]
The proof is implicit in the proof of Theorem 4.1 in \cite{cck2017_AoP}. 
\end{proof}

\subsection{Proof sketch of Theorem \ref{thm:gaussian_approx_bern_random_design}}
\label{subsec:proof_sketch_gauss_approx_non-degenerate}

Before we formally prove Theorem \ref{thm:gaussian_approx_bern_random_design}, we first sketch the proof for the Bernoulli sampling case. Similar arguments carry through the sampling with replacement case as well. The rigorous proof of Theorem \ref{thm:gaussian_approx_bern_random_design} is given in Section \ref{subsec:proofs_gaussian_approx}. Let $C$ denote a generic constant that depends only on $\underline{\sigma}$ and $r$. Decompose the difference $U_{n,N}' - \theta$ as 
\[
U_{n,N}' - \theta = \frac{N}{\hat{N}} \cdot \frac{1}{N}\sum_{\iota \in I_{n,r}} Z_{\iota} \{ h(X_{\iota}) - \theta \} = 
\frac{N}{\hat{N}} (A_{n}+\sqrt{1-p_{n}}B_{n}),
\]
where $A_{n}$ and $B_{n}$ are defined by 
\[
A_{n}  = U_{n} - \theta \quad \text{and} \quad B_{n} = \frac{1}{N} \sum_{\iota \in I_{n,r}} \frac{(Z_{\iota} - p_{n})}{\sqrt{1-p_{n}}} \{ h(X_{\iota}) - \theta \}.
\]
For the notational convenience, we write $W_{n} = A_{n}+\sqrt{1-p_{n}}B_{n}$. 
For any hyperrectangle $R \in \calR$, observe that 
\[
\Prob ( \sqrt{n} W_{n} \in R) = \Prob \left (\sqrt{N} B_{n} \in \left[ \frac{1}{\sqrt{\alpha_{n}(1-p_{n})}} R-\sqrt{\frac{N}{1-p_{n}}}A_{n} \right] \right),
\]
where $\alpha_{n} =n/N$. 
Since $A_{n}$ is $\sigma (X_{1}^{n})$-measurable, conditionally on $X_{1}^{n}$, $A_{n}$ can be treated as a constant. 
On the other hand, conditionally on $X_{1}^{n}$, $\sqrt{N}B_{n} = (N(1-p_{n}))^{-1/2}\sum_{\iota \in I_{n,r}} (Z_{\iota} - p_{n}) \{ h(X_{\iota})  - \theta \} = |I_{n,r}|^{-1/2} \sum_{\iota \in I_{n,r}} (p_{n}(1-p_{n}))^{-1/2}(Z_{\iota} - p_{n}) \{ h(X_{\iota}) - \theta  \}$ is the sum of independent random vectors with mean zero whose (conditional) covariance matrix is given by 
\[
\hat{\Gamma}_{h} = |I_{n,r}|^{-1} \sum_{\iota \in I_{n,r}} \{ h(X_{\iota}) - \theta \} \{ h(X_{\iota}) - \theta \}^{T}. 
\]
Subject to the moment conditions (C1) and (C2), Lemma \ref{lem:talagrand_ineq_ustat} in the SM and Lemma 8 in \cite{cck2015_anticoncentration} yield that with probability at least $1-Cn^{-1}$, 
\begin{equation}
\label{eqn:proof_sketch_thm3.1_step1.5}
|\hat{\Gamma}_{h}-\Gamma_{h}|_{\infty} \le C\{ n^{-1/2} D_{n} \log^{1/2}(dn) + n^{-1} D_{n} (\log{n}) \log^{3}(dn) \},
\end{equation}
where $\Gamma_{h} = P^{r} (h-\theta)(h - \theta)^{T}$. Under the non-degeneracy condition (C3-ND), by the Gaussian approximation result (cf. Theorem 2.1 in \cite{cck2017_AoP}) applied conditionally on $X_{1}^{n}$, we can show that with probability at least $1-Cn^{-1}$, 
\begin{equation}
\label{eqn:proof_sketch_thm3.1_step1}
\rho_{\mid X_{1}^{n}}^{\calR}(\sqrt{N}B_{n},\hat{Y}) := \sup_{R \in \calR} \left|\Prob_{\mid X_{1}^{n}}(\sqrt{N}B_{n} \in R) - \Prob_{\mid X_{1}^{n}}(\hat{Y}  \in R) \right| \le C \varpi_{n},
\end{equation}
where $\varpi_{n} = \{D_{n}^{2} \log^{7}(dn) / (n \wedge N)\}^{1/6}$ and $\hat{Y}$ is a random vector in $\R^{d}$ such that $\hat{Y} \mid X_{1}^{n} \sim N(0,\hat{\Gamma}_{h})$. This step is involved and we defer the details of deriving (\ref{eqn:proof_sketch_thm3.1_step1}) to Appendix \ref{sec:proofs} in the SM. Then, in view of (\ref{eqn:proof_sketch_thm3.1_step1.5}), the Gaussian comparison inequality (Lemma \ref{lem:Gaussian_comparison} in the SM) yields that  
\begin{equation}
\label{eqn:proof_sketch_thm3.1_step2}
\sup_{R \in \calR} \left | \Prob_{\mid X_{1}^{n}} (\hat{Y} \in R) - \gamma_{B}(R)\right|  \le C \varpi_{n}
\end{equation}
with probability at least $1-Cn^{-1}$, where $\gamma_{B} =  N(0, \Gamma_{h})$. Since $A_{n}$ is a mean-zero complete $U$-statistic, the high-dimensional CLT for the H\'ajek projection term (cf. Proposition 2.1 in \cite{cck2017_AoP}), the maximal inequality for the higher-order degenerate terms (cf. Corollary 5.6 in \cite{chenkato2017a}), and Nazarov's inequality (cf. Lemma A.1 in \cite{cck2017_AoP}) together yield that 
\begin{equation}
\label{eqn:proof_sketch_thm3.1_step3}
\sup_{R \in \calR} \left | \Prob(\sqrt{n} A_{n} \in R) - \gamma_{A}(R) \right | \le C\varpi_{n},
\end{equation}
where $\gamma_{A}=  N(0,r^{2}\Gamma_{g})$ and  $\Gamma_{g} = P(g-\theta)(g-\theta)^{T}$ with $g = P^{r-1}h$. For any hyperrectangle $R \in \calR$, observe that  
\begin{align*}
\Prob(\sqrt{n} W_n \in R) &=  \E\left[ \Prob_{\mid X_{1}^{n}} \left (\sqrt{N}B_{n} \in \left [\frac{1}{\sqrt{\alpha_{n}(1-p_{n})}}R-\sqrt{\frac{N}{1-p_{n}}}A_{n} \right] \right )  \right] \\
&\le_{(*)}  \E\left [ \gamma_{B} \left (\left[\frac{1}{\sqrt{\alpha_{n}(1-p_{n})}}R-\sqrt{\frac{N}{1-p_{n}}}A_{n}\right]\right  )  \right ] + C\varpi_{n} \\
& =\Prob \left (\sqrt{1-p_{n}}Y_{B} \in [\alpha_{n}^{-1/2}R-\sqrt{N}A_{n} ] \right ) + C \varpi_{n} \\
& =\Prob \left (\sqrt{n}A_n \in [R-\sqrt{\alpha_{n}(1-p_{n})}Y_{B}] \right ) +  C \varpi_{n},
\end{align*}
where $Y_{B} \sim  N(0,\Gamma_{h}) = \gamma_{B}$ independent of $X_{1}^{n}$, and the inequality $(*)$ follows from (\ref{eqn:proof_sketch_thm3.1_step1}) and (\ref{eqn:proof_sketch_thm3.1_step2}). Now we freeze the random variable $Y_{B}$. Since $Y_{B}$ is independent of $X_{1}^{n}$, (\ref{eqn:proof_sketch_thm3.1_step3}) yields that
\[
\Prob_{\mid Y_{B}} \left (\sqrt{n}A_n \in [R-\sqrt{\alpha_{n}(1-p_{n})}Y_{B}] \right ) \le \gamma_{A}\left ([R-\sqrt{\alpha_{n}(1-p_{n})}Y_{B}] \right ) + C \varpi_{n}.
\]
By Fubini, we conclude that 
\[
\begin{split}
\Prob(\sqrt{n} W_n \in R) &\le \E \left [ \gamma_{A}\left ([R-\sqrt{\alpha_{n}(1-p_{n})}Y_{B}] \right ) \right ] + C \varpi_{n} \\
& =\Prob \left (Y_{A} \in [R-\sqrt{\alpha_{n}(1-p_{n})}Y_{B}] \right ) + C\varpi_{n} \\
&=\Prob (Y_{A}+\sqrt{\alpha_{n}(1-p_{n})}Y_{B} \in R) + C\varpi_{n},
\end{split}
\]
where $Y_{A} \sim N(0,r^{2}\Gamma_{g}) = \gamma_{A}$ is independent of $Y_{B}$. Since $\alpha_{n}p_{n} = n/|I_{n,r}| \le C n^{-r+1} \le Cn^{-1}$ and $| \Gamma_{h} |_{\infty} \le CD_{n}$, a second application of the Gaussian comparison inequality (Lemma \ref{lem:Gaussian_comparison}) yields that 
\[
\Prob (Y_{A}+\sqrt{\alpha_{n}(1-p_{n})}Y_{B} \in R) \le \Prob (Y_{A} + \alpha_{n}^{1/2} Y_{B} \in R) + C \left ( \frac{D_{n}\log^{2} d}{n} \right )^{1/3},
\]
and the second term on the right hand side is bounded from above by $C\varpi_{n}$. 
Likewise, we have $\Prob(\sqrt{n} W_n \in R) \ge \Prob (Y_{A}+\alpha_{n}^{1/2}Y_{B} \in R) - C\varpi_{n}$. Hence, for $Y = Y_{A}+\alpha_{n}^{1/2}Y_{B}  \sim N(0,r^{2}\Gamma_{g}+\alpha_{n}\Gamma_{h})$, we have
\begin{equation}
\label{eqn:proof_sketch_thm3.1_step4}
\sup_{R \in \calR} \left | \Prob (\sqrt{n}W_{n} \in R) - \Prob (Y \in R) \right | \le C\varpi_{n}.
\end{equation}
In the last step of the proof, using Bernstein's inequality (cf. Lemma 2.2.9 in \cite{vandervaartwellner1996}), we can verify that the inequality (\ref{eqn:proof_sketch_thm3.1_step4}) holds with $\sqrt{n}W_{n}$ replaced by $\sqrt{n} (U_{n,N}'- \theta)$. This leads to the conclusion of Theorem \ref{thm:gaussian_approx_bern_random_design} in the Bernoulli sampling case. \qed

\subsection{Proofs for Section \ref{sec:gaussian_approximations}}
\label{subsec:proofs_gaussian_approx}

Observe that $P^{r}|h_{j} - \theta_{j}|^{2+k} \le 2^{2+k} D_{n}^{k}$ by Jensen's inequality for all $j$ and $k=1,2$, and $\| h_{j}(X_{1}^{r})  - \theta_{j} \|_{\psi_{1}} \le (1+\log 2) D_{n}$ for all $j$. So, in view of the identity $U_{n,N}' - \theta = \hat{N}^{-1} \sum_{\iota \in I_{n,r}} Z_{\iota}\{ h(X_{\iota}) - \theta \}$ where $\hat{N}$ is replaced by $N$ for the sampling with replacement case, 
it is without loss of generality to assume that
\[
\theta = P^{r}h = 0
\] 
by replacing $h$ with $h-\theta$. 

Throughout this section, the notation $\lesssim$ signifies that the left hand side is bounded by the right hand side up to a constant that depends only on $\underline{\sigma}, r$ under (C2) and only on $\underline{\sigma},r, q$ under (C2'). In addition, let $C$ denote a generic constant that depends only on $\underline{\sigma},r$ under (C2) and only on $\underline{\sigma},r,q$ under (C2'); its value may change from place to place. 

\begin{proof}[Proof of Theorem \ref{thm:gaussian_approx_bern_random_design}]
It is not difficult to see that the equality of the first two terms in (\ref{eqn:gaussian_approx_bern_random_design}) holds since $n = N \alpha_{n}$. So it suffices to prove the second line in (\ref{eqn:gaussian_approx_bern_random_design}). 
In this proof, without loss of generality, we may assume that 
\begin{equation}
\label{eqn:gaussian_approx_reduction}
D_n^2 \log^7(d n) \le c_{1} (n \wedge N)
\end{equation}
for some sufficiently small constant $c_{1} \in (0,1)$ depending only on $\underline{\sigma}$ and $r$, since otherwise the conclusion of the theorem is trivial by taking $C$ in (\ref{eqn:gaussian_approx_bern_random_design}) sufficiently large. In addition, for the notational convenience, let 
\[
\varpi_{n} := \left (\frac{D_{n}^{2} \log^{7}(dn)}{n \wedge N} \right )^{1/6}\le 1.
\]

\uline{Bernoulli sampling case}. First, consider the Bernoulli sampling. The proof is divided into several steps. 

\uline{Step 1}. 
Recall the decomposition $W_{n}= (\hat{N}/N)U_{n,N}' = U_{n} + N^{-1} \sum_{\iota \in I_{n,r}} (Z_{\iota}-p_{n}) h(X_{\iota}) = A_{n} + \sqrt{1-p_{n}}B_{n}$, and observe that $\sqrt{N}B_{n}  = |I_{n,r}|^{-1/2} \sum_{\iota \in I_{n,r}} (p_{n}(1-p_{n}))^{-1/2} (Z_{\iota} - p_{n})h(X_{\iota})$. 
Let $\hat{Y}$ be a random vector in $\R^{d}$ such that $\hat{Y} \mid X_{1}^{n} \sim N(0,\hat{\Gamma}_{h})$ where $\hat{\Gamma}_{h}  =|I_{n,r}|^{-1} \sum_{\iota \in I_{n,r}} h(X_{\iota})h(X_{\iota})^{T}$.
In this step, we shall show that with probability at least $1-Cn^{-1}$,
\begin{align*}
\rho_{\mid X_{1}^{n}}^{\calR}(\sqrt{N}B_{n},\hat{Y}) := & \sup_{R \in \calR} \left|\Prob_{\mid X_{1}^{n}}(\sqrt{N}B_{n} \in R) - \Prob_{\mid X_{1}^{n}}(\hat{Y}  \in R) \right| \le C \varpi_{n}.
\end{align*}
The proof of Step 1 is lengthy and  divided into six sub-steps.

\uline{Step 1.1}. We first derive a generic upper bound on $\rho_{\mid X_{1}^{n}}^{\calR}(\sqrt{N}B_{n}, \hat{Y})$. Let $\hat{Y}_{\iota}, \iota \in I_{n,r}$ be random vectors in $\R^{d}$ independent conditionally on $X_{1}^{n}$ such that $\hat{Y}_{\iota} \mid X_{1}^{n} \sim N(0,h(X_{\iota})h(X_{\iota})^{T})$ for  $\iota \in I_{n,r}$. Observe that conditionally on $X_{1}^{n}$, $\hat{Y} \stackrel{d}{=} |I_{n,r}|^{-1/2} \sum_{\iota \in I_{n,r}} \hat{Y}_{\iota}$. Define
\begin{equation}
\label{eqn:defn_Ln_bern_sampling}
\hat{L}_n = \max_{1 \le j \le d} {1 \over |I_{n,r}|} \sum_{\iota \in I_{n,r}} (p_n(1-p_{n}))^{-3/2} |h_j(X_\iota)|^3 \E[|Z_{\iota}-p_n|^3].
\end{equation}
Further, for $\phi \ge 1$, define
\begin{equation}
\label{eqn:defn_Mn_bern_sampling}
\begin{split}
\hat{M}_{n,X}(\phi) &= {1 \over |I_{n,r}|} \sum_{\iota \in I_{n,r}} \E_{\mid X_{1}^{n}} \left[ \max_{1 \le j \le d} \left| {(Z_{\iota} - p_n) h_j(X_\iota) \over \sqrt{p_n(1-p_{n})}} \right|^3 \vone\left( \max_{1 \le j \le d} \left| {(Z_{\iota} - p_n) h_j(X_\iota) \over \sqrt{p_n(1-p_{n})}} \right| > {\sqrt{|I_{n,r}|} \over 4 \phi \log{d}} \right) \right], \\
\hat{M}_{n,Y}(\phi)&= {1 \over |I_{n,r}|} \sum_{\iota \in I_{n,r}} \E_{\mid X_{1}^{n}} \left[ \max_{1 \le j \le d} |\hat{Y}_{\iota,j}|^3 \vone \left( \max_{1 \le j \le d} |\hat{Y}_{\iota,j}| > {\sqrt{|I_{n,r}|} \over 4 \phi \log{d}} \right)  \right],
\end{split}
\end{equation}
and $\hat{M}_n(\phi) = \hat{M}_{n,X}(\phi) + \hat{M}_{n,Y}(\phi)$.
Let $\overline{L}_n$ and $\overline{M}_n$ be constants whose values will be determined later. 

Then, Theorem 2.1 in \cite{cck2017_AoP} (applied conditionally on $X_{1}^{n}$) yields that there exists a constant $C_{2}$ depending only on $\underline{\sigma}$ such that for 
\[
\phi_n = C_2 \left( {\overline{L}_n^2 \log^4 d \over |I_{n,r}|} \right)^{-1/6},
\]
we have that 
\begin{equation}
\label{eqn:primative_bound_gauss_approx}
\rho_{\mid X_{1}^{n}}^{\calR}(\sqrt{N}B_{n}, \hat{Y}) \le C \left \{  \left( {\overline{L}_n^2 \log^7 d \over |I_{n,r}|} \right)^{1/6} + {\overline{M}_n \over \overline{L}_n} \right \}
\end{equation}
on the event $\mathcal{E}_{n} := \{ \hat{M}_{n}(\phi_{n}) \le \overline{M}_{n} \} \cap \{ \hat{L}_{n} \le \overline{L}_{n} \} \cap \{ \min_{1 \le j \le d}  \hat{\Gamma}_{h,jj} \ge \underline{\sigma}^{2}/2 \}$.
In Steps 1.2--1.4, we will bound $\hat{L}_{n}$ and $\hat{M}_{n}(\phi_{n})$, and in Step 1.5, we will evaluate the probability that $\min_{1 \le j \le d}  \hat{\Gamma}_{h,jj} \ge \underline{\sigma}^{2}/2$. In Step 1.6, we will derive an explicit bound on $\rho_{\mid X_{1}^{n}}^{\calR}(\sqrt{N}B_{n},\hat{Y})$ that holds with probability at least $1-Cn^{-1}$. 

\uline{Step 1.2: Bounding $\hat{L}_n$}. Since $p_n \le 1/2$ and $\E[|Z_{\iota}-p_n|^3] = p_n (1-p_n) \{ p_n^2+(1-p_n)^2 \} \le Cp_{n}$, $\hat{L}_{n}$ is bounded from above by $Cp_{n}^{-1/2}$ times $\max_{1 \le j \le d} |I_{n,r}|^{-1} \sum_{\iota \in I_{n,r}} |h_j(X_{\iota})|^3 =: \mathsf{Z}_{1}$. 
Let $m = \lfloor n/r \rfloor, \breve{\mathsf{Z}}_{1} = \max_{1 \le j \le d} \sum_{i=1}^{m} |h_j(X_{(i-1)r+1}^{ir})|^3$, and $\mathsf{M}_{1} = \max_{1 \le i \le m} \max_{1 \le j \le d} |h_j(X_{(i-1)r+1}^{ir})|$. Then, Lemma E.3 in \cite{chen2017a} yields that 
\[
\Prob \left (m \mathsf{Z}_{1} \ge 2 \E[\breve{\mathsf{Z}}_{1}] + C \| \mathsf{M}_{1}^{3} \|_{\psi_{1/3}} t^{3} \right) \le 3 e^{-t}
\]
for every $t > 0$.
Further,  since the blocks $X_{(i-1)r+1}^{ir}, i=1,\dots,m$ are i.i.d., Lemma 9  in \cite{cck2015_anticoncentration} yields that 
\[
\E[\breve{\mathsf{Z}}_{1}] \lesssim  \max_{1 \le j \le d} \sum_{i=1}^{m} \E \left [ |h_j(X_{(i-1)r+1}^{ir})|^3 \right ] +\E[\mathsf{M}_{1}^{3}] \log d \lesssim m D_{n} + \E[\mathsf{M}_{1}^{3}]\log d.
\]
Since $\E[\mathsf{M}_{1}^3] \lesssim \|\mathsf{M}_{1}^3\|_{\psi_{1/3}} = \|\mathsf{M}_{1}\|_{\psi_1}^3 \lesssim D_n^3 \log^3(d n)$, we have 
\[
\Prob \left (\hat{L}_n \ge C p_n^{-1/2} D_{n} \{ 1 + n^{-1} D_n^2 \log^4(d n) + t^3 n^{-1} D_n^2 \log^3(d n) \} \right) \le 3 e^{-t}.
\]
Since $D_n^2 \log^7(d n) \le c_{1} n$, by choosing $\overline{L}_n =C p_n^{-1/2} D_n$ and $t = \log n$, we conclude that $\Prob(\hat{L}_n \ge \overline{L}_n) \le 3n^{-1}$.

\uline{Step 1.3: Bounding $\hat{M}_{n,X}(\phi_n)$}. We begin with noting that  
\[
\hat{M}_{n,X}(\phi) \le {C \over |I_{n,r}|} \sum_{\iota \in I_{n,r}}  p_{n}^{-1/2}\max_{1 \le j \le d} |h_j(X_{\iota})|^3 \vone\left( \max_{1 \le j \le d} |h_j(X_{\iota})| > {\sqrt{N} \over 4 \phi \log{d}} \right).
\]
Since $\| \max_{\iota \in I_{n,r}} \max_{1 \le j \le d} |h_j(X_{\iota})| \|_{\psi_{1}} \lesssim D_{n}\log (dn)$,
we have that 
\[
\max_{\iota \in I_{n,r}} \max_{1 \le j \le d} |h_j(X_{\iota})| \le CD_{n}\log^{2}(dn)
\]
with probability at least $1-2n^{-1}$. Now, since
\[
{\sqrt{N} \over 4 \phi_n \log{d}} \gtrsim \left ( \frac{D_{n}N}{\log d} \right )^{1/3} \ge c_{1}^{-1/3} D_{n} \log^{2}(dn),
\]
by choosing $c_{1}$ in (\ref{eqn:gaussian_approx_reduction}) sufficiently small. In addition, $\phi_{n}^{-1} = C_{2}^{-1} C^{1/3} (N^{-1} D_{n}^{2} \log^{4}{d})^{1/6} \le C_{2}^{-1} C^{1/3} c_{1}^{1/6} \le 1$ for a sufficiently small $c_{1}$. Hence we have that $\hat{M}_{n,X}(\phi_{n}) = 0$ with probability at least $1-2n^{-1}$.

\uline{Step 1.4: Bounding $\hat{M}_{n,Y}(\phi_n)$}. 
Suppose that 
\[
\max_{\iota \in I_{n,r}} \max_{1 \le j \le d} | h_{j}(X_{\iota})| \le CD_{n}\log^{2}(dn),
\]
which holds with probability at least $1-2n^{-1}$ by Step 1.3. 
Recall that $\| \xi \|_{\psi_{1}} \le (1+e) \| \xi \|_{\psi_{2}}$ for every real-valued random variable. (For completeness, we provide its proof: assume $\| \xi \|_{\psi_{2}} = 1$, and observe that $\E[ e^{|\xi|}] \le e+\E[e^{\xi^{2}}] \le e+2$, so that $\E[\psi_{1}(| \xi |)] \le 1+e$. The desired result follows from the observation that $\psi_{1}(x/C) \le C^{-1} \psi_{1}(x)$ for $C > 1$ and $x \ge 0$.)  
Conditionally on $X_{1}^{n}$, since $\hat{Y}_{\iota,j} \mid X_{1}^{n} \sim N(0,h_j^{2}(X_{\iota}))$ for every $\iota \in I_{n,r}$, we have 
$\| \max_{1 \le j \le d} |\hat{Y}_{\iota,j}| \|_{\psi_{1}} \le (1+e) \| \max_{1 \le j \le d} |\hat{Y}_{\iota,j} | \|_{\psi_{2}} \lesssim \max_{1 \le j \le d} |h_{j}(X_{\iota})| \log^{1/2} d$, so that 
\begin{equation}
\label{eqn:bound_hatY_iota_j_GA}
\Prob_{\mid X_{1}^{n}}\left (\max_{1 \le j \le d} |\hat{Y}_{\iota, j}| \ge t \right) \le 2 \exp\left (  -{t \over C \max_{1 \le j \le d} |h_j(X_{\iota})| \log^{1/2} d} \right)
\end{equation}
for every $t > 0$. 
Hence,  it follows from Lemma C.1 in \cite{cck2017_AoP} that 
\begin{align*}
&\E_{\mid X_{1}^{n}} \left [ \max_{1 \le j \le d} |\hat{Y}_{\iota,j}|^3 \vone \left( \max_{1 \le j \le d} |\hat{Y}_{\iota,j}| > {\sqrt{|I_{n,r}|} \over 4 \phi_{n} \log{d}} \right)  \right] \\
&\lesssim \left( {\sqrt{|I_{n,r}|} \over \phi_n \log d } + \max_{1 \le j \le d} |h_j(X_{\iota})| \log^{1/2} d \right)^3 \exp \left ( -{\sqrt{|I_{n,r}|} \over C \phi_n \max_{1 \le j \le d} |h_j(X_{\iota})| \log^{3/2} d} \right ) \\
&\lesssim (n^{r/2} + D_{n} \log^{5/2} (dn) )^{3} \exp \left ( -{|I_{n,r}|^{1/3} \over C D_{n}^{2/3} \log^{17/6} (dn)} \right ) \\
&\lesssim n^{3r/2} \exp \left ( -{n^{2/3} \over C D_{n}^{2/3} \log^{17/6} (dn)} \right ) \le n^{3r/2} \exp \left ( - \frac{n^{1/3}}{C \log^{1/2}(dn)} \right ) \\
&\le n^{3r/2} \exp (-n^{11/42}/C) \le n^{3r/2} e^{-n^{1/4}/C},
\end{align*}
where we have used the assumption (\ref{eqn:gaussian_approx_reduction}). Therefore, we conclude that 
$\hat{M}_{n,Y}(\phi_{n}) \le C n^{3r/2} e^{-n^{1/4}/C}$ with probability at least $1-2n^{-1}$. 

\uline{Step 1.5: Bounding $|\hat{\Gamma}_{h} - \Gamma_h|_\infty$}.
Let $\mathsf{Z}_{2} = | \hat{\Gamma}_{h} - \Gamma_{h} |_{\infty}$, and 
observe that 
\[
\begin{split}
&\max_{1 \le j,k \le d} \|h_j(X_{\iota}) h_k(X_{\iota})\|_{\psi_{1/2}} \le \max_{1 \le j,k \le d} \| h_{j}^{2}(X_{\iota})/2+ h_{k}^{2}(X_{\iota})/2\|_{\psi_{1/2}} \\
&\quad \lesssim \max_{1 \le j,k \le d} (\| h_{j}^{2}(X_{\iota}) \|_{\psi_{1/2}} + \| h_{k}^{2}(X_{\iota}) \|_{\psi_{1/2}})
\lesssim \max_{1 \le j \le d} \|h_j^{2}(X_{\iota}) \|_{\psi_{1/2}} \\
&\quad = \max_{1 \le j \le d} \|h_j(X_{\iota}) \|_{\psi_{1}}^2 \le D_n^2,
\end{split}
\]
and $\max_{1 \le j,k \le d} P^{r}(h_{j}h_{k})^{2} = \max_{1 \le j \le d}P^{r}h_{j}^{4} \le D_{n}^{2}$. 
Hence,  Lemma \ref{lem:talagrand_ineq_ustat} yields that
\[
\Prob \left(m \mathsf{Z}_{2} \ge 2 \E[\breve{\mathsf{Z}}_{2}] + t \right) \le e^{-t^{2}/(3mD^{2}_{n})}+ 3 \exp\{-(t /( C\|\mathsf{M}_2\|_{\psi_{1/2}}))^{1/2} \},
\]
where $m = \lfloor n/r \rfloor$, and $\breve{\mathsf{Z}}_{2}$ and $\mathsf{M}_{2}$ are defined by 
\begin{align*}
\breve{\mathsf{Z}}_{2} &=  \max_{1 \le j,k \le d}  \left| \sum_{i=1}^{m}\{ h_j(X_{(i-1)r+1}^{ir}) h_k(X_{(i-1)r+1}^{ir})- P^{r}h_{j}h_{k} \} \right| \quad \text{and} \\
\mathsf{M}_{2} &= \max_{1 \le j,k \le d} \max_{1 \le i \le m} \left| h_j(X_{(i-1)r+1}^{ir}) h_k(X_{(i-1)r+1}^{ir}) - P^{r}h_{j}h_{k} \right|.
\end{align*}
Observe that $\|\mathsf{M}_{2}\|_{\psi_{1/2}} \lesssim D_n^2 \log^2(d n)$. In addition, Lemma 8 in \cite{cck2015_anticoncentration} yields that 
\[
\E[\breve{\mathsf{Z}}_{2}] \lesssim   \sqrt{mD_{n}^{2}\log d} + \sqrt{\E[\mathsf{M}_{2}^{2}]} \log  d \lesssim D_{n}\sqrt{n  \log d} + D_n^2 \log^3(d n).
\]
Hence, 
\[
\begin{split}
&\Prob \left(\mathsf{Z}_{2} \ge C \{  n^{-1/2} D_{n} \log^{1/2} d + n^{-1} D_n^2 \log^3(d n) \} + t \right) \\
&\quad \le \exp\left(-{n t^2 \over 3 r D_n^2}\right) + 3 \exp\left(-{(nt)^{1/2} \over C D_n \log(d n)}\right).
\end{split}
\]
Choosing $t=C n^{-1/2} D_{n} (\log{n})^{1/2} \bigvee C n^{-1} D_{n}^{2} (\log{n})^2 \log^2(dn)$ for large enough $C$ leads to 
\begin{equation}
\label{eqn:gaussian_approx_bern_random_design_step1.5}
\Prob \left(\mathsf{Z}_{2} \ge C \{  n^{-1/2} D_{n} \log^{1/2} (dn) + n^{-1} D_n^2 (\log{n}) \log^3(d n) \}  \right) \le Cn^{-1}.
\end{equation}
Choosing $c_{1}$ in (\ref{eqn:gaussian_approx_reduction}) small enough, we conclude that $| \hat{\Gamma}_{h} - \Gamma_{h} |_{\infty} \le \underline{\sigma}^{2}/2$ and hence $\min_{1 \le j \le d} \hat{\Gamma}_{h,jj} \ge \underline{\sigma}^{2}/2$ with probability at least $1-Cn^{-1}$. 

\uline{Step 1.6}. In view of Steps 1.1-1.5, choosing $\overline{L}_{n} = Cp_{n}^{-1/2} D_{n}$ and $\overline{M}_{n} = Cn^{3r/2} e^{-n^{1/4}/C}$, we have  $\Prob (\mathcal{E}_{n}) \ge 1-Cn^{-1}$. 
Hence, 
\[
\rho_{\mid X_{1}^{n}}^{\calR} (\sqrt{N}B_{n},\hat{Y}) \le C \left\{ \left ( {D_n^2 \log^7 d \over N} \right)^{1/6} + {p_n^{1/2} \over D_{n}}n^{3r/2} e^{-n^{1/4}/C} \right\} \lesssim \varpi_{n}
\]
with probability at least $1-Cn^{-1}$.

\uline{Step 2: Gaussian comparison}.  In this step, we shall show that
\[
\sup_{R \in \calR} \left| \Prob_{\mid X_{1}^{n}}(\hat{Y} \in R)- \gamma_{B}(R) \right| \le C \varpi_{n}
\]
with probability at least $1-Cn^{-1}$, where $\gamma_{B}=N(0,\Gamma_{h})$.
First, the Gaussian comparison inequality (Lemma \ref{lem:Gaussian_comparison}) yields that the left hand side is bounded by $C \overline{\Delta}^{1/3} \log^{2/3} d$ on the event $\{|\hat{\Gamma}_{h} - \Gamma_h|_\infty \le \overline{\Delta}\}$. From Step 1.5, $| \hat{\Gamma}_{h} - \Gamma_{h} |_{\infty} \le C \{ n^{-1/2} D_{n} \log^{1/2} (dn) + n^{-1} D_n^2 (\log{n}) \log^3(d n) \}$ with probability at least $1-Cn^{-1}$, so that  
\[
\sup_{R \in \calR} \left | \Prob_{\mid X_{1}^{n}} (\hat{Y} \in R) - \gamma_{B}(R)\right|  \le  C \left \{ \left ( {D_n^2 \log^5(d n) \over n} \right )^{1/6} + \left( \frac{D_{n}^{2} (\log n)  \log^{5}(dn)}{n} \right )^{1/3} \right \}  \lesssim \varpi_{n}
\]
with probability at least $1-Cn^{-1}$. 

\uline{Step 3: Gaussian approximation to $A_n$}.
Recall that $\Gamma_{g} = P gg^{T}$ since $\theta = 0$. In this step, we shall show that
\begin{equation}
\label{eqn:GA_to_A}
\sup_{R \in \calR} \left | \Prob(\sqrt{n} A_{n} \in R) - \gamma_{A}(R) \right | \le C\varpi_{n},
\end{equation}
where $\gamma_{A} =N(0,r^{2}\Gamma_{g})$.
The Hoeffding decomposition yields that
\[
A_n = \sum_{k=1}^r {r \choose k} U_n^{(k)}(\pi_k h) = r U_{n}^{(1)}(\pi_{1}h) +\underbrace{\sum_{k=2}^r {r \choose k} U_n^{(k)}(\pi_k h)}_{=:\mathsf{R}_{n}},
\]
where $(\pi_k h)(x_1,\dots,x_k) = (\delta_{x_1} - P) \cdots (\delta_{x_k} - P) P^{r-k} h$
is the Hoeffding projection at level $k$; see, e.g., \cite{delaPenaGine1999}, p.137.
Since $rU_{n}^{(1)}(\pi_{1}h)  = rn^{-1} \sum_{i=1}^{n} g(X_{i})$ is the average of centered independent random vectors with covariance matrix $r^{2} \Gamma_{g}$, Proposition 2.1 in \cite{cck2017_AoP} yields that 
\[
\sup_{R \in \calR} \left | \Prob (r\sqrt{n}U_{n}^{(1)}(\pi_{1}h) \in R) - \gamma_{A}(R) \right | \leq C \varpi_{n}
\]
under our assumption.
It remains to bound the effect of the remainder term $\mathsf{R}_{n}$. To this end, we make use of Corollary 5.6 in \cite{chenkato2017a}, which yields that 
\[
\E\left [ \max_{1 \le j \le d} |U_{n}^{(k)}(\pi_{k}h_{j})| \right ] \lesssim n^{-k/2} (\log^{k/2} d )\sqrt{P^{r}\left (\max_{1 \le j \le d} h_{j}^{2}\right )} \lesssim n^{-k/2} D_{n} \log^{k/2+1} d
\]
for every $k=2,\dots,r$. Hence, 
\[
\E\left [|\mathsf{R}_{n}|_{\infty}  \right ] \lesssim D_{n}\sum_{k=2}^{r}n^{-k/2}  \log^{k/2+1} d \lesssim n^{-1} D_{n} \log^{2}d.  
\]
Now, for $R = \prod_{j=1}^{d} [a_{j},b_{j}]$, let $a=(a_{1},\dots,a_{d})^{T}$ and $b=(b_{1},\dots,b_{d})^{T}$, and for $t > 0$, we use the convention that $b+t = (b_{1}+t,\dots,b_{d}+t)^{T}$. Observe that
\begin{align*}
&\Prob (\sqrt{n}A_{n} \in R)  = \Prob (\{ -\sqrt{n}A_{n} \le -a \} \cap \{ \sqrt{n} A_{n} \le b \}) \\
&\quad \le  \Prob \left (\{ -\sqrt{n}A_{n} \le -a \} \cap \{ \sqrt{n} A_{n} \le b \} \cap \left \{  |\sqrt{n}\mathsf{R}_{n}|_{\infty}   \le t \right \} \right) +\Prob \left( |\sqrt{n}\mathsf{R}_{n}|_{\infty}  > t \right) \\
&\quad \le \Prob (\{ -r\sqrt{n}U_{n}^{(1)}(\pi_{1}h) \le -a+t \} \cap \{ r\sqrt{n}U_{n}^{(1)}(\pi_{1}h)  \le b+t \}) + Ct^{-1}n^{-1/2}D_{n} \log^{2}d \\
&\quad \le \gamma_{A}( \{ y \in \R^{d}: -y \le -a+t, y \le b+t \})+ C\varpi_{n} +  Ct^{-1}n^{-1/2}D_{n} \log^{2}d \\
&\quad \le  \gamma_{A}(R)+ C t \sqrt{\log d} + C\varpi_{n} + Ct^{-1}n^{-1/2}D_{n} \log^{2}d
\end{align*}
for every $t > 0$, where the last inequality follows from Nazarov's inequality stated in Lemma A.1 in \cite{cck2017_AoP}. Choosing $t=(n^{-1}D_{n}^{2}\log^{3}d)^{1/4}$, we conclude that 
\[
\Prob (\sqrt{n}A_{n} \in R) - \gamma_{A}(R) \le C\left ( \frac{D_{n}^{2}\log^{5}d}{n} \right )^{1/4} + C \varpi_{n} \le C \varpi_{n}
\]
because of the assumption (\ref{eqn:gaussian_approx_reduction}). Likewise, we have  $\Prob (\sqrt{n}A_{n} \in R) \ge \gamma_{A}(R) - C\varpi_{n}$. Therefore, we obtain the conclusion (\ref{eqn:GA_to_A}).

\uline{Step 4: Gaussian approximation to $W_{n}$}.
Pick any hyperrectangle $R \in \calR$. Recall that $\alpha_n = n / N$, and observe that
\[
\Prob(\sqrt{n} W_n \in R) =  \E\left[ \Prob_{\mid X_{1}^{n}} \left (\sqrt{N}B_{n} \in \left [\frac{1}{\sqrt{\alpha_{n}(1-p_{n})}}R-\sqrt{\frac{N}{1-p_{n}}}A_{n} \right] \right )  \right].
\]
Now, we freeze the random variables $X_1^n$. From  Steps 1 and 2, the conditional probability inside the expectation is bounded from above by $\gamma_{B} \left (\left[\frac{1}{\sqrt{\alpha_{n}(1-p_{n})}}R-\sqrt{\frac{N}{1-p_{n}}}A_{n}\right]\right  ) + C \varpi_{n}$
with probability at least $1-Cn^{-1}$. Since the probability is bounded by $1$ and $n^{-1} \lesssim \varpi_{n}$, we have 
\begin{align*}
&\Prob(\sqrt{n} W_n \in R) \le  \E\left [ \gamma_{B} \left (\left[\frac{1}{\sqrt{\alpha_{n}(1-p_{n})}}R-\sqrt{\frac{N}{1-p_{n}}}A_{n}\right]\right  )  \right ] + C\varpi_{n} \\
&\quad =\Prob \left (\sqrt{1-p_{n}}Y_{B} \in [\alpha_{n}^{-1/2}R-\sqrt{N}A_{n} ] \right ) + C \varpi_{n} =\Prob \left (\sqrt{n}A_n \in [R-\sqrt{\alpha_{n}(1-p_{n})}Y_{B}] \right ) +  C \varpi_{n},
\end{align*}
where $Y_{B} \sim  N(0,\Gamma_{h}) = \gamma_{B}$ independent of $X_{1}^{n}$. 
Next, we freeze the random variable $Y_{B}$. Since $Y_{B}$ is independent of $X_{1}^{n}$, Step 3 yields that
\[
\Prob_{\mid Y_{B}} \left (\sqrt{n}A_n \in [R-\sqrt{\alpha_{n}(1-p_{n})}Y_{B}] \right ) \\
\le \gamma_{A}\left ([R-\sqrt{\alpha_{n}(1-p_{n})}Y_{B}] \right ) + C \varpi_{n}.
\]
By Fubini, we conclude that 
\[
\begin{split}
&\Prob(\sqrt{n} W_n \in R) \le \E \left [ \gamma_{A}\left ([R-\sqrt{\alpha_{n}(1-p_{n})}Y_{B}] \right ) \right ] + C \varpi_{n} \\
&\quad =\Prob \left (Y_{A} \in [R-\sqrt{\alpha_{n}(1-p_{n})}Y_{B}] \right ) + C\varpi_{n} =\Prob (Y_{A}+\sqrt{\alpha_{n}(1-p_{n})}Y_{B} \in R) + C\varpi_{n},
\end{split}
\]
where $Y_{A} \sim N(0,r^{2}\Gamma_{g}) = \gamma_{A}$ is independent of $Y_{B}$.
Since $\alpha_{n}p_{n} = n/|I_{n,r}| \lesssim n^{-r+1} \le n^{-1}$ and $| \Gamma_{h} |_{\infty} \lesssim D_{n}$, using the Gaussian comparison inequality (Lemma \ref{lem:Gaussian_comparison}), we have 
\[
\Prob (Y_{A}+\sqrt{\alpha_{n}(1-p_{n})}Y_{B} \in R) \le \Prob (Y_{A} + \alpha_{n}^{1/2} Y_{B} \in R) + C \left ( \frac{D_{n}\log^{2} d}{n} \right )^{1/3},
\]
and the second term on the right hand side is bounded from above by $C\varpi_{n}$. 
Likewise, we have $\Prob(\sqrt{n} W_n \in R) \ge \Prob (Y_{A}+\alpha_{n}^{1/2}Y_{B} \in R) - C\varpi_{n}$. Hence, for $Y = Y_{A}+\alpha_{n}^{1/2}Y_{B}  \sim N(0,r^{2}\Gamma_{g}+\alpha_{n}\Gamma_{h})$, we have
\begin{equation}
\label{eqn:GA1}
\sup_{R \in \calR} \left | \Prob (\sqrt{n}W_{n} \in R) - \Prob (Y \in R) \right | \le C\varpi_{n}.
\end{equation}

\uline{Step 5: Gaussian approximation to $U_{n,N}'$}. We shall verify that the inequality (\ref{eqn:GA1}) holds with $\sqrt{n}W_{n}$ replaced by $\sqrt{n} U_{n,N}'$. 
Since $Y$ is centered Gaussian and $\max_{1 \le j \le d} \Var (Y_{j}) \lesssim D_{n}(1+\alpha_{n})$, we have 
$\E[ | Y |_{\infty}]  \lesssim \sqrt{D_{n}(1+\alpha_{n}) \log d}$. By the Borell-Sudakov-Tsirel'son inequality (cf. Theorem 2.5.8 in \cite{ginenickl2016}), we have 
\[
\Prob \left ( |\alpha_{n}^{-1/2} Y|_{\infty}  >  C\sqrt{D_{n}(1+\alpha_{n}^{-1})\log (dn)} \right ) \le 2n^{-1}.
\]
Combining this estimate with (\ref{eqn:GA1}), we have 
\[
\Prob \left ( |\sqrt{N}W_{n}|_{\infty} > C\sqrt{D_{n}(1+\alpha_{n}^{-1}) \log (dn)} \right ) \le C\varpi_{n}.
\]
Next, since $\hat{N} = \sum_{\iota \in I_{n,r}} Z_{\iota}$ and $Z_{\iota}, \iota \in I_{n,r}$ are i.i.d. $\Bern (p_{n})$ with $p_{n}=N/|I_{n,r}|$,  by Bernstein's inequality (cf. Lemma 2.2.9 in \cite{vandervaartwellner1996}), we have 
\[
\Prob \left ( |\hat{N} - N| >  \sqrt{2 N t}  + 2 t / 3 \right ) \le 2e^{-t}
\]
for every $t > 0$. Choosing $t=\log n$ and choosing $c_{1}$ sufficiently small in (\ref{eqn:gaussian_approx_reduction}) such that $\sqrt{(\log  n)/N} \le 1/4$, we have 
\[
\Prob \left ( |\hat{N} / N - 1| > 2 \sqrt{(\log n)/N}  \right ) \le 2n^{-1}.
\]
Since $|z^{-1}-1| \le 2|z-1|$ for $|z-1| \le 1/2$, we have that
\[
|N/\hat{N} -1| \le 2 | \hat{N}/N - 1| \le 4 \sqrt{(\log n)/N} 
\]
with probability at least $1-2n^{-1}$.

Now, observe that $\sqrt{N}U_{n,N}' = \sqrt{N}W_{n} + (N/\hat{N}-1)\sqrt{N}W_{n}$,
and with probability at least $1-C\varpi_{n}$,
\[
|(N/\hat{N}-1)\sqrt{N}W_{n}|_{\infty} \le C \sqrt{\frac{D_{n}(\log n)\log (dn)}{n \wedge N}}.
\]
Arguing as in Step 3 and noting that $\min_{1 \le j \le d} \Var (\alpha_{n}^{-1/2}Y_{j}) \ge  \min_{1 \le j \le d}P^{r}h_{j}^{2} \gtrsim 1$, we conclude that for every $R \in \calR$,
\[
\begin{split}
\Prob ( \sqrt{N}U_{n,N}' \in R) &\le \Prob (\alpha_{n}^{-1/2} Y \in R) + C\varpi_{n} + C\sqrt{\frac{D_{n}(\log n)\log^{2} (dn)}{n \wedge N}} \\
&\le \Prob (\alpha_{n}^{-1/2} Y \in R)+C \varpi_{n}.
\end{split}
\]
Likewise, we have $\Prob (\sqrt{N}U_{n,N}' \in R) \ge \Prob (\alpha_{n}^{-1/2} Y \in R) - C\varpi_{n}$. This leads to the conclusion of the theorem in the Bernoulli sampling case.

\uline{Sampling with replacement case}. 
Next, consider sampling with replacement. The proof is similar to the Bernoulli sampling case, so we only point out the differences. Recall that we assume $\theta = 0$. Observe that $U_{n,N}' =  U_{n} + N^{-1}\sum_{j=1}^{N} \{ h(X_{\iota_{j}}^{*}) - U_{n} \} =: A_{n}  + B_{n}$. Since $X_{\iota_1}^*,\dots,X_{\iota_N}^*$ are i.i.d. draws  from the empirical distribution $|I_{n,r}|^{-1} \sum_{\iota \in I_{n,r}} \delta_{X_{\iota}}$ conditionally on $X_{1}^{n}$, $\sqrt{N} B_{n}$ is $\sqrt{N}$ times the average of i.i.d. random vectors with mean zero and covariance matrix $\hat{\Gamma}_{h} - U_{n}U_{n}^{T}$ conditionally on $X_{1}^{n}$, where $\hat{\Gamma}_{h} = |I_{n,r}|^{-1} \sum_{\iota \in I_{n,r}} h(X_{\iota}) h(X_{\iota})^T$. Let $\hat{Y}$ be a random vector in $\R^d$ such that $\hat{Y} \mid X_1^n \sim N(0, \hat{\Gamma}_{h} - U_{n} U_{n}^T)$. We first verify that 
\[
\rho_{\mid X_{1}^{n}}^{\calR}(\sqrt{N}B_{n},\hat{Y}) \le C \varpi_{n}
\]
with probability at least $1-Cn^{-1}$. Define 
\begin{equation}
\label{eqn:defn_Ln_sampling_with_replacement}
\hat{L}_n := \max_{1 \le j \le d} {1 \over |I_{n,r}|} \sum_{\iota \in I_{n,r}} |h_{j}(X_{\iota}) - U_{n,j}|^3.
\end{equation}
By Jensen's inequality, $\hat{L}_n \le 8 {\mathsf{Z}}_{1}$, where ${\mathsf{Z}}_{1}$ is defined in Step 1.2 for the Bernoulli sampling case. By Step 1.2, we have $\Prob(\hat{L}_n \ge C D_n) \le 3 n^{-1}$ under the assumption (\ref{eqn:gaussian_approx_reduction}). So we can take $\overline{L}_n = C D_n$ and $\phi_n = C_2 (N^{-1} \overline{L}_n^2 \log^4 d)^{-1/6} \ge 1$ by choosing $c_{1}$ in (\ref{eqn:gaussian_approx_reduction}) sufficiently small. For $\phi \ge 1$, define 
\begin{equation}
\label{eqn:defn_Mn_sampling_with_replacement}
\begin{split}
\hat{M}_{n,X}(\phi) &= {1 \over |I_{n,r}|} \sum_{\iota \in I_{n,r}} \left[ \max_{1 \le j \le d} \left| h_j(X_\iota) -U_{n,j} \right|^3 \vone\left( \max_{1 \le j \le d} \left| h_j(X_\iota) -U_{n,j} \right| > {\sqrt{N} \over 4 \phi \log{d}} \right) \right], \\
\hat{M}_{n,Y}(\phi)&=\E_{\mid X_{1}^{n}} \left[ \max_{1 \le j \le d} |\hat{Y}_{j}|^3 \vone \left( \max_{1 \le j \le d} |\hat{Y}_{j}| > {\sqrt{N} \over 4 \phi \log{d}} \right)  \right],
\end{split}
\end{equation}
and $\hat{M}_n(\phi) = \hat{M}_{n,X}(\phi) + \hat{M}_{n,Y}(\phi)$. Observe that 
\begin{align*}
\left \| \max_{\iota \in I_{n,r}} \max_{1 \le j \le d} |h_j(X_{\iota}) - U_{n,j}| \right \|_{\psi_1} & \lesssim \max_{\iota \in I_{n,r}} \max_{1 \le j \le d}  \| h_j(X_{\iota}) - U_{n,j} \|_{\psi_1} \log (dn) \\
& \lesssim \max_{\iota \in I_{n,r}} \max_{1 \le j \le d}  \| h_j(X_{\iota}) \|_{\psi_1}\log (dn)  \le D_n \log(dn),
\end{align*}
and hence 
\[
\max_{\iota \in I_{n,r}} \max_{1 \le j \le d} |h_j(X_{\iota}) - U_{n,j}| \le C D_n \log^2(dn)
\]
 with probability at least $1 - 2n^{-1}$. Using similar calculations to those in Step 1.3, we have that $\hat{M}_{n,X}(\phi_n) = 0$ with probability at least $1-2n^{-1}$. Step 1.4 needs a modification. Since $\hat{Y}_j \mid X_1^n \sim N(0, |I_{n,r}|^{-1} \sum_{\iota \in I_{n,r}} (h_j(X_{\iota}) - U_{n,j})^2)$, we have $\| \max_{1 \le j \le d} |\hat{Y}_{j}| \|_{\psi_{1}} \lesssim \| \max_{1 \le j \le d} |\hat{Y}_{j}| \|_{\psi_{2}} \lesssim \sqrt{V_{n} \log  d}$ conditionally on $X_{1}^{n}$ where $V_{n} = \max_{1 \le j \le d} |I_{n,r}|^{-1} \sum_{\iota \in I_{n,r}} h_j^2(X_{\iota})$, from which we have
\begin{equation}
\label{eqn:bound_hatY_j_GA_replacement}
\Prob_{\mid X_1^n} \left( \max_{1 \le j \le d} |\hat{Y}_j| \ge t \right) \le 2 \exp\left( -{t \over C \sqrt{V_{n} \log d}} \right).
\end{equation}
Let $m = \lfloor n/r \rfloor$ and  $\breve{V}_{n} = \max_{1 \le j \le d} \sum_{i=1}^{m} h_j^2(X_{(i-1)r+1}^{ir})$. Then, Lemma E.3 in \cite{chen2017a} yields that 
\[
\Prob \left (m V_{n} \ge 2 \E[\breve{V}_{n}] + C \| \mathsf{M}_{1}^{2} \|_{\psi_{1/2}} t^{2} \right) \le 3 e^{-t}
\]
for every $t > 0$, where $\mathsf{M}_{1} = \max_{1 \le i \le m} \max_{1 \le j \le d} |h_j(X_{(i-1)r+1}^{ir})|$. 
Further, Lemma 9 in   \cite{cck2015_anticoncentration} yields that 
\[
\E[\breve{V}_{n}] \lesssim  \max_{1 \le j \le d} \sum_{i=1}^{m} \E \left [ h_j^2(X_{(i-1)r+1}^{ir}) \right ] +\E[\mathsf{M}_{1}^{2}] \log d \lesssim m D_{n} + \E[\mathsf{M}_{1}^{2}]\log d.
\]
Since $\E[\mathsf{M}_{1}^2] \lesssim \|\mathsf{M}_{1}^2\|_{\psi_{1/2}} = \|\mathsf{M}_{1}\|_{\psi_1}^2 \lesssim D_n^2 \log^2(d n)$, we have 
\[
\Prob \left (V_{n} \ge C D_{n} \{ 1 + n^{-1} D_n \log^3(d n) + t^2 n^{-1} D_n \log^2(d n) \} \right) \le 3 e^{-t}.
\]
Since $D_n \ge 1$ and $D_n^2 \log^7(d n) \le c_{1} n$, by choosing $t = \log n$, we conclude that $\Prob(V_{n} \ge C D_n) \le 3n^{-1}$. Now, suppose that $V_{n} \le CD_{n}$, which holds with probability at least $1-3n^{-1}$. Then, since
\[
\Prob_{\mid X_1^n} \left( \max_{1 \le j \le d} |\hat{Y}_j| \ge t \right) \le 2 \exp\left( -{t \over C \sqrt{D_n \log d}} \right),
\]
 it follows from Lemma C.1 in \cite{cck2017_AoP} that 
\begin{align*}
& \E_{\mid X_1^n} \left[ \max_{1\le j \le d} |\hat{Y}_j|^3 \vone\left( \max_{1\le j \le d} |\hat{Y}_j| > {\sqrt{N} \over 4 \phi_n \log d} \right) \right] \\
& \quad \lesssim \left[ {\sqrt{N} \over 4 \phi_n \log d} + \sqrt{D_n \log d} \right]^3 \exp\left( -{\sqrt{N} \over C \phi_n D_n^{1/2} \log^{3/2} d} \right) \\
& \quad \lesssim N^{3/2} \exp \left( -{N^{1/3} \over C D_n^{1/6} \log^{5/6} d} \right) \le N^{3/2} \exp(-N^{1/6} / C),
\end{align*}
where we have used the assumption (\ref{eqn:gaussian_approx_reduction}). Therefore, we conclude that 
$\hat{M}_{n,Y}(\phi_{n}) \le C N^{3/2} \exp(-N^{1/6} / C)$ with probability at least $1-3n^{-1}$. 

Step 1.5 also needs a modification. Note that $|\hat{\Gamma}_{h} - U_n U_n^T - \Gamma_{h}|_\infty \le |\hat{\Gamma}_{h} - \Gamma_{h}|_\infty + |U_n|_\infty^2$. In the Bernoulli sampling case, we have shown in Step 1.5 that $|\hat{\Gamma}_{h} - \Gamma_{h}|_\infty \le \underline{\sigma}^2 / 4$ with probability at least $1-Cn^{-1}$ (changing the constant from $1/2$ to $1/4$ does not affect the proof). So we only need to show that $|U_{n}|_\infty^2 \le \underline{\sigma}^2 / 4$ with probability at least $1-Cn^{-1}$. By Lemma \ref{lem:talagrand_ineq_ustat}, 
\[
\Prob \left(m |U_{n}|_\infty \ge 2 \E[\breve{\mathsf{Z}}_{3}] + t \right) \le e^{-t^{2}/(3mD_{n})}+ 3 \exp \{ -t / (C\|\mathsf{M}_1\|_{\psi_{1}}) \},
\]
where $\breve{\mathsf{Z}}_{3} =  \max_{1 \le j \le d}  |\sum_{i=1}^{m} h_j(X_{(i-1)r+1}^{ir})|$. Observe that $\|\mathsf{M}_{1}\|_{\psi_{1}} \lesssim D_n \log(d n)$. In addition, Lemma 8 in \cite{cck2015_anticoncentration} yields that 
\[
\E[\breve{\mathsf{Z}}_{3}] \lesssim   \sqrt{mD_{n}\log d} + \sqrt{\E[\mathsf{M}_{1}^{2}]} \log  d \lesssim \sqrt{n  D_{n} \log d} + D_n \log^2(d n).
\]
Hence, 
\[
\begin{split}
&\Prob \left( |U_{n}|_\infty \ge C \{  n^{-1/2} D_{n}^{1/2} \log^{1/2} d + n^{-1} D_n \log^2(d n) \} + t \right) \\
&\quad \le \exp\left(-{n t^2 \over 3 r D_n}\right) + 3 \exp\left(-{nt \over C D_n \log(d n)}\right).
\end{split}
\]
Choosing $t=C n^{-1/2} D_{n}^{1/2} (\log{n})^{1/2} \bigvee C n^{-1} D_{n} (\log{n}) \log(dn)$ for large enough $C$ leads to 
\[
\Prob \left( |U_{n}|_\infty \ge C \{  n^{-1/2} D_{n}^{1/2} \log^{1/2} (dn) + n^{-1} D_n \log^2(d n) \}  \right) \le Cn^{-1}.
\]
Choosing $c_{1}$ in (\ref{eqn:gaussian_approx_reduction}) small enough, we conclude that $|U_{n}|_\infty^2 \le \underline{\sigma}^{2}/4$ and hence $\min_{1 \le j \le d} \{ \hat{\Gamma}_{h,jj} - U_{n,j}^2 \} \ge \underline{\sigma}^{2}/2$ with probability at least $1-Cn^{-1}$. Therefore, the overall bound in Step 1 for the sampling with replacement case is given by 
\[
\rho_{\mid X_{1}^{n}}^{\calR}(\sqrt{N}B_{n},\hat{Y}) \le C \left \{ \left ( \frac{D_{n}^{2}\log^{7}d}{N} \right )^{1/6} +\frac{N^{3/2} e^{-N^{1/6}/C}}{D_{n}} \right \} \lesssim \varpi_{n}
\]
with probability at least $1-Cn^{-1}$.

Step 2 in the Bernoulli sampling case goes through under the assumption (\ref{eqn:gaussian_approx_reduction}). Step 3 remains exactly the same as the Bernoulli sampling case. Step 4 follows similarly as the Bernoulli sampling case with $p_n = 0$. Step 5 is not needed in the sampling with replacement case. This completes the proof.
\end{proof}

\begin{proof}[Proof of Corollary \ref{cor:Gaussian_approximation}]
In view of Theorem \ref{thm:gaussian_approx_bern_random_design}, the corollary follows from the Gaussian comparison inequality (Lemma \ref{lem:Gaussian_comparison}) and the fact that $| \Gamma_{g} |_{\infty} \le | \Gamma_{h} |_{\infty} \le CD_{n}$. 
\end{proof}

\begin{proof}[Proof of Theorem \ref{thm:gaussian_approx_bern_random_design_degenerate}]
We shall follow the notation used in the proof of Theorem \ref{thm:gaussian_approx_bern_random_design}.
 In this proof, without loss generality, we may assume that 
\begin{equation}
\label{eqn:gaussian_approx_reduction_degenerate}
ND_{n}^{2} \log^{k+3} d \le c_{2} n^{k}, \quad D_{n}^{2} (\log n) \log^{5}(dn) \le c_{2} n, \quad \text{and} \quad  D_{n}^{2} \log^{7}(dn) \le c_{2}N
\end{equation}
for some sufficiently small constant $c_{2}$ depending only on $\underline{\sigma}$ and $r$, since otherwise the conclusion of the theorem is trivial by taking $C$ in (\ref{eqn:gaussian_approx_bern_random_design_degenerate}) sufficiently large. 

\uline{Bernoulli sampling case}. We first verify that 
\begin{equation}
\label{eqn:step1_degenerate}
\rho_{\mid X_{1}^{n}}^{\calR} (\sqrt{N}B_{n},\hat{Y}) \le C \left \{ \left ( \frac{D_{n}^{2}\log^{7}d}{N} \right )^{1/6} +\frac{p_{n}^{1/2}}{D_{n}} n^{3r/2} e^{-n^{-1/10}/C} \right \}
\end{equation}
with probability at least $1-Cn^{-1}$.

It is not difficult to verify from Step 1.2 in the proof of Theorem \ref{thm:gaussian_approx_bern_random_design} that $\Prob (\hat{L}_{n} \ge C p_{n}^{-1/2}D_{n}) \le 3n^{-1}$ under the assumption that $D_{n}^{2} (\log n) \log^{5}(dn) \le c_{2} n$, and so take $\overline{L}_{n}=Cp_{n}^{-1/2} D_{n}$. Step 1.3 goes through as it is. Step 1.4 needs a modification. From Step 1.4, we have that on the event $\max_{\iota \in I_{n,r}}\max_{1 \le j \le d} | h_{j} (X_{\iota}) | \le CD_{n}\log^{2}(dn)$, 
\[
\E_{\mid X_{1}^{n}} \left [ \max_{1 \le j \le d} |\hat{Y}_{\iota,j}|^3 \vone \left( \max_{1 \le j \le d} |\hat{Y}_{\iota,j}| > {\sqrt{|I_{n,r}|} \over 4 \phi_{n} \log{d}} \right)  \right] \\
\le C n^{3r/2} \exp \left ( -{n^{2/3} \over C D_{n}^{2/3} \log^{17/6} (dn)} \right ),
\]
and the assumption that $D_{n}^{2} (\log n) \log^{5}(dn) \le c_{2} n$ yields that the right hand side is bounded from above by
\[
n^{3r/2} \exp \left ( -{(n\log n)^{1/3} \over C   \log^{7/6} (dn)} \right ) \le n^{3r/2} e^{-n^{1/10}/C}. 
\]
Since $\max_{\iota \in I_{n,r}}\max_{1 \le j \le d} | h_{j} (X_{\iota}) | \le CD_{n}\log^{2}(dn)$  with probability at least $1-2n^{-1}$, we have that $\hat{M}_{n,Y} (\phi_{n}) \le C n^{3r/2} e^{-n^{1/10}/C}$ with probability at least $1-2n^{-1}$. 

Step 1.5 holds under the present assumption. Hence, the inequality (\ref{eqn:step1_degenerate}) holds with probability at least $1-Cn^{-1}$. 
In addition, Step 2 in the proof of Theorem \ref{thm:gaussian_approx_bern_random_design} goes through under the present assumption (\ref{eqn:gaussian_approx_reduction_degenerate}), so that 
\[
\begin{split}
\sup_{R \in \calR} \left | \Prob_{\mid X_{1}^{n}}  (\hat{Y} \in R) - \gamma_{B} (R) \right | &\le C \left \{ \left ( {D_n^2 \log^5(d n) \over n} \right )^{1/6} + \left( \frac{D_{n}^{2} (\log n)  \log^{5}(dn)}{n} \right )^{1/3} \right \} \\
&\le C  \left ({D_n^2 (\log n)\log^{5}(dn) \over n} \right )^{1/6}
\end{split}
\]
with probability at least $1-Cn^{-1}$. 
Therefore, we have that
\[
\sup_{R \in \calR} \left | \Prob_{\mid X_{1}^{n}} (\sqrt{N}B_{n} \in R) - \gamma_{B} (R) \right | \le C \left \{ \left({D_n^2 (\log n) \log^5(nd) \over n}\right)^{1/6}+ \left({D_n^2 \log^7 (dn) \over N}\right)^{1/6} \right\} =: C \breve{\varpi}_{n}
\]
with probability at least $1-Cn^{-1}$, and  in view of the fact that $n^{-1} \lesssim \breve{\varpi}_{n}$, we conclude that 
\[
\sup_{R \in \calR} \left | \Prob (\sqrt{N}B_{n} \in R) - \gamma_{B} (R) \right | \le C \breve{\varpi}_{n}.
\]
Since $h$ is degenerate of order $k-1$, we have
\[
A_n = \sum_{\ell=k}^r {r \choose \ell} U_n^{(\ell)}(\pi_\ell h),
\]
and Step 3  in the proof of Theorem \ref{thm:gaussian_approx_bern_random_design} yields that 
\begin{equation}
\label{eqn:degenerate_An}
\E\left [ |A_{n}|_{\infty} \right ] \le C D_{n}\sum_{\ell=k}^{r} n^{-\ell/2} \log^{\ell/2+1}d \le C D_{n} n^{-k/2} \log^{k/2+1}d.
\end{equation}
Hence, for $R=\prod_{j=1}^{d} [a_{j},b_{j}], a = (a_{1},\dots,a_{d})^{T}, b=(b_{1},\dots,b_{d})^{T}$, and $t > 0$, we have 
\[
\begin{split}
&\Prob (\sqrt{N} W_{n} \in R) = \Prob (\{ -\sqrt{N} W_{n} \le -a \} \cap \{ \sqrt{N}W_{n} \le b \} ) \\
&\quad \le \Prob \left (\{ -\sqrt{N} W_{n} \le -a \} \cap \{ \sqrt{N}W_{n} \le b \} \cap \left \{  |\sqrt{N}A_{n}|_{\infty} \le t \right \} \right ) + \Prob \left ( |\sqrt{N}A_{n}|_{\infty} > t \right )  \\
&\quad \le \Prob \left (\{ -\sqrt{N(1-p_{n})} B_{n} \le -a + t \} \cap \{ \sqrt{N(1-p_{n})}B_{n} \le b+ t \} \right ) + C t^{-1} \sqrt{N}n^{-k/2}D_{n} \log^{k/2+1}d \\
&\quad \le \gamma_{B}(\{ y \in \R^{d} : -\sqrt{1-p_{n}}y \le -a+t, \sqrt{1-p_{n}}y \le b+t\} ) + C \breve{\varpi}_{n} + Ct^{-1} \sqrt{N}n^{-k/2} D_{n} \log^{k/2+1}d \\
&\quad \le \gamma_{B}([(1-p_{n})^{-1/2}R]) + Ct \sqrt{\log d} + C \breve{\varpi}_{n} + Ct^{-1} \sqrt{N}n^{-k/2} D_{n} \log^{k/2+1}d,
\end{split}
\]
where the last inequality follows from Nazarov's inequality (\cite{cck2017_AoP}, Lemma A.1). Choosing $t=(Nn^{-k} D_{n}^{2} \log^{k+1} d)^{1/4}$, we conclude that 
\[
\Prob (\sqrt{N} W_{n} \in R) \le \gamma_{B}([(1-p_{n})^{-1/2}R]) + C \left (\frac{ND_{n}^{2}\log^{k+3}d}{n^{k}} \right)^{1/4} + C \breve{\varpi}_{n}.
\]
Finally, since $p_{n} \lesssim N/n^{r}$, the Gaussian comparison inequality (Lemma \ref{lem:Gaussian_comparison}) yields that 
\[
\gamma_{B}([(1-p_{n})^{-1/2}R]) \le \gamma_{B}(R) + C\left ( \frac{ND_{n} \log^{2}d}{n^{r}} \right )^{1/3},
\]
and the second term on the right hand side is bounded from above by $C\left  (\frac{ND_{n}^{2}\log^{k+3}d}{n^{k}} \right)^{1/4}$. Hence,
\[
\Prob (\sqrt{N} W_{n} \in R)  \le \gamma_{B}(R) + C \left (\frac{ND_{n}^{2}\log^{k+3}d}{n^{k}} \right)^{1/4}+ C \breve{\varpi}_{n}.
\]
Likewise, we have 
\[
\Prob (\sqrt{N} W_{n} \in R) \ge  \gamma_{B}(R) - C \left (\frac{ND_{n}^{2}\log^{k+3}d}{n^{k}} \right)^{1/4} - C \breve{\varpi}_{n}.
\]
Finally, arguing as in Step 5 in the proof of Theorem \ref{thm:gaussian_approx_bern_random_design},
we obtain the conclusion of Theorem \ref{thm:gaussian_approx_bern_random_design_degenerate} for the Bernoulli sampling case.

\uline{Sampling with replacement case}. This case is similar to but easier than the Bernoulli sampling case under degeneracy. Recall that $U_{n,N}' =  A_{n}  + B_{n}$, where $A_{n} = U_{n}$ and $B_{n} = N^{-1}\sum_{j=1}^{N} \{ h(X_{\iota_{j}}^{*}) - U_{n} \}$. Under the assumptions that $D_n^2 (\log n) \log^5(dn) \le c_{2} n$ and $D_{n}^{2}\log^{7}(dn) \le c_{2} N$, all the sub-steps of Step 1 in the proof of Theorem \ref{thm:gaussian_approx_bern_random_design} carry over to the degenerate case, i.e., we have that
\[
\rho_{\mid X_{1}^{n}}^{\calR}(\sqrt{N}B_{n},\hat{Y}) \le C \left \{ \left ( \frac{D_{n}^{2}\log^{7}d}{N} \right )^{1/6} +\frac{N^{3/2} e^{-N^{1/6}/C}}{D_{n}} \right \}
\]
with probability at least $1-Cn^{-1}$. In addition, the error bound in Step 2 remains the same as the Bernoulli sampling case under degeneracy. Hence, we have that 
\[
\sup_{R \in \calR} \left | \Prob_{\mid X_{1}^{n}} (\sqrt{N}B_{n} \in R) - \gamma_{B} (R) \right | \le C \breve{\varpi}_{n}
\]
with probability at least $1-Cn^{-1}$. Now, using the estimate (\ref{eqn:degenerate_An}), for $R=\prod_{j=1}^{d} [a_{j},b_{j}], a = (a_{1},\dots,a_{d})^{T}, b=(b_{1},\dots,b_{d})^{T}$, and $t > 0$, we have 
\[
\begin{split}
&\Prob (\sqrt{N} U_{n,N}' \in R) = \Prob (\{ -\sqrt{N} U_{n,N}' \le -a \} \cap \{ \sqrt{N}U_{n,N}' \le b \} ) \\
&\quad \le \Prob \left (\{ -\sqrt{N} U_{n,N}' \le -a \} \cap \{ \sqrt{N}U_{n,N}' \le b \} \cap \left \{  |\sqrt{N}A_{n}|_{\infty} \le t \right \} \right ) + \Prob \left ( |\sqrt{N}A_{n}|_{\infty} > t \right )  \\
&\quad \le \Prob \left (\{ -\sqrt{N} B_{n} \le -a + t \} \cap \{ \sqrt{N}B_{n} \le b+ t \} \right ) + C t^{-1} \sqrt{N}n^{-k/2}D_{n} \log^{k/2+1}d \\
&\quad \le \gamma_{B}(\{ y \in \R^{d} : -y \le -a+t, y \le b+t\} ) + C \breve{\varpi}_{n} + Ct^{-1} \sqrt{N}n^{-k/2} D_{n} \log^{k/2+1}d \\
&\quad \le \gamma_{B}(R) + Ct \sqrt{\log d} + C \breve{\varpi}_{n} + Ct^{-1} \sqrt{N}n^{-k/2} D_{n} \log^{k/2+1}d,
\end{split}
\]
where the last inequality follows from Nazarov's inequality (\cite{cck2017_AoP}, Lemma A.1). Choosing $t=(Nn^{-k} D_{n}^{2} \log^{k+1} d)^{1/4}$, we conclude that 
\[
\Prob (\sqrt{N} U_{n,N}' \in R) \le \gamma_{B}(R) + C \left (\frac{ND_{n}^{2}\log^{k+3}d}{n^{k}} \right)^{1/4} + C \breve{\varpi}_{n}.
\]
Likewise, we have the reverse inequality and the conclusion of Theorem \ref{thm:gaussian_approx_bern_random_design_degenerate} for the sampling with replacement case follows.
\end{proof}

\begin{proof}[Proof of Theorem \ref{thm:gaussian_approx_bern_random_design_polymom}]
We first prove Part (i). We shall follow the notation used in the proof of Theorem \ref{thm:gaussian_approx_bern_random_design}. In addition to assuming (\ref{eqn:gaussian_approx_reduction}), we may also assume that 
\begin{equation}
\label{eqn:gaussian_approx_reduction_2}
D_{n}^{2} n^{2r/q} \log^{3}(dn) \le c_{1} (n \wedge N)^{1-2/q},
\end{equation}
for some sufficiently small constant $c_{1} \in (0,1)$ depending only on $\underline{\sigma},r$, and $q$. In this proof, the constant $c_{1}$ in (\ref{eqn:gaussian_approx_reduction})  may also depend on $\underline{\sigma},r$, and $q$. Let 
\[
\varphi_{n} := \left( {D_{n}^{2} n^{2r/q} \log^{3}(dn) \over (n \wedge N)^{1-2/q} } \right)^{1/3} \le 1.
\]
Without loss of generality, we may assume $q > 2(r+1)$ since otherwise $\varphi_{n} > 1$ and the bound (\ref{eqn:gaussian_approx_bern_random_design_polymom}) holds trivially. 

\uline{Bernoulli sampling case}. Recall that $\hat{Y} \mid X_{1}^{n} \sim N(0, \hat{\Gamma}_{h})$, where $\hat{\Gamma}_{h} = |I_{n,r}|^{-1} \sum_{\iota \in I_{n,r}} h(X_{\iota}) h(X_{\iota})^{T}$. We first verify that 
\[
\rho_{\mid X_{1}^{n}}^{\calR}(\sqrt{N}B_{n},\hat{Y}) \le C (\varpi_{n} + \varphi_{n})
\]
with probability at least $1-C(n \wedge N)^{-1}$. Step 1.1 in the proof of Theorem \ref{thm:gaussian_approx_bern_random_design} goes through, namely that there exists a constant $C_{2}$ depending only on $\underline{\sigma}$ and $q$ such that the bound (\ref{eqn:primative_bound_gauss_approx}) holds for 
\[
\phi_n = C_2 \left( {\overline{L}_n^2 \log^4 d \over |I_{n,r}|} \right)^{-1/6}
\]
on the event $\mathcal{E}_{n} := \{ \hat{M}_{n}(\phi_{n}) \le \overline{M}_{n} \} \cap \{ \hat{L}_{n} \le \overline{L}_{n} \} \cap \{ \min_{1 \le j \le d}  \hat{\Gamma}_{h,jj} \ge \underline{\sigma}^{2}/2 \}$, where $\hat{L}_{n}$ and $\hat{M}_{n}(\phi)$ are defined in (\ref{eqn:defn_Ln_bern_sampling}) and (\ref{eqn:defn_Mn_bern_sampling}), respectively. Step 1.2--Step 1.5 need modifications due to the polynomial moment condition (C2'). Observe that $(\E[\mathsf{M}_{1}^{q}])^{3/q} \le m^{3/q} D_{n}^{3}$, and by Lemma 9 in \cite{cck2015_anticoncentration}, we have
\[
\E[\breve{\mathsf{Z}}_{1}] \lesssim  \max_{1 \le j \le d} \sum_{i=1}^{m} \E \left [ |h_j(X_{(i-1)r+1}^{ir})|^3 \right ] +\E[\mathsf{M}_{1}^{3}] \log d \lesssim m D_{n} + m^{3/q} D_{n}^{3} \log d.
\]
Applying Lemma E.4 in \cite{chen2017a} and recalling that $\hat{L}_{n} \le C p_{n}^{-1/2} \mathsf{Z}_{1}$, we have 
\[
\Prob \left (\hat{L}_{n} \ge C p_{n}^{-1/2} D_{n} \{1 + D_{n}^{2} n^{-1+6/q} \log{d} \}  \right) \le n^{-1}.
\]
Choose 
\[
\overline{L}_{n} = C p_{n}^{-1/2} D_{n} \left\{1 + { D_{n}^{2} n^{3r/q} \log{d} \over (n \wedge N)^{1-3/q} } \right\}.
\]
Then $\Prob(\hat{L}_{n} \ge \overline{L}_{n}) \le n^{-1}$ since $r \ge 2$. Now for our choice of $\overline{L}_{n}$ and $\phi_{n}$, we observe that 
\begin{align*}
{\sqrt{N} \over 4 D_{n} \phi_{n} \log{d}} =& {C^{1/3} \over 4 C_{2}} \left( { N \over D_{n}^{2} \log{d}} \right)^{1/3} \left( 1 + {D_{n}^{2} n^{3r/q} \log{d} \over (n \wedge N)^{1-3/q} } \right)^{1/3} \\
\ge& {C^{1/3} \over 4 C_{2}} {N^{1/3} n^{r/q} \over (n \wedge N)^{1/3-1/q}} \ge {C^{1/3} \over 4 C_{2}} n^{r/q} N^{1/q}.
\end{align*}
Then Markov's inequality and Condition (C2') yield that 
\[
\Prob \left(\max_{\iota \in I_{n,r}} \max_{1 \le j \le d} |h_{j}(X_{\iota})| > {\sqrt{N} \over 4 \phi_{n} \log{d}} \right) \le {n^{r} (4 D_{n} \phi_{n} \log{d})^{q} \over N^{q/2}} \le {(4 C_{2})^{q} \over C^{q/3}  N}.
\]
In addition, similar calculations under the assumptions (\ref{eqn:gaussian_approx_reduction}) and (\ref{eqn:gaussian_approx_reduction_2}) show that $\phi_{n}^{-1} \le 1$ for a sufficiently small $c_{1}$. Indeed, since $\varpi_{n} \le c_{1}^{1/6} \le 1$ and $\varphi_{n} \le c_{1}^{1/3} \le 1$ under (\ref{eqn:gaussian_approx_reduction_2}), we have 
\begin{align*}
\phi_{n}^{-1} &= {C^{1/3} \over C_{2}} \left( D_{n}^{2} \log^{4} {d} \over N \right)^{1/6} \left( 1 + {D_{n}^{2} n^{3r/q} \log{d} \over (n \wedge N)^{1-3/q}} \right)^{1/3} \\
&\le  {C^{1/3} \over C_{2}} \left\{ \left( D_{n}^{2} \log^{4} {d} \over N \right)^{1/6} + \left( D_{n}^{2} n^{2r/q} \log^{2} {d} \over (n \wedge N)^{1-2/q} \right)^{1/2}  \right\} \\
&\le C( \varpi_{n} + \varphi_{n}^{3/2} ) \le C(\varpi_{n} + \varphi_{n}) \le C c_{1}^{1/6} \le 1.
\end{align*}
Hence we have $\Prob(\hat{M}_{n,X}(\phi_{n}) = 0) \ge 1 - C N^{-1}$. For $\hat{M}_{n,Y}(\phi)$, suppose that 
\[
\max_{\iota \in I_{n,r}} \max_{1 \le j \le d} |h_{j}(X_{\iota})| \le D_{n} n^{(r+1)/q},
\]
which holds with probability at least $1-n^{-1}$ under (C2'). Observe that 
\[
\phi_{n}^{-1} \ge {C^{1/3} \over C_{2}} {D_{n} n^{r/q} \log{d} \over |I_{n,r}|^{1/6} (n \wedge N)^{1/3-1/q}}.
\]
Then in view of (\ref{eqn:bound_hatY_iota_j_GA}) and Lemma C.1 in \cite{cck2017_AoP}, we have  
\begin{align*}
&\E_{\mid X_{1}^{n}} \left [ \max_{1 \le j \le d} |\hat{Y}_{\iota,j}|^3 \vone \left( \max_{1 \le j \le d} |\hat{Y}_{\iota,j}| > {\sqrt{|I_{n,r}|} \over 4 \phi_{n} \log{d}} \right)  \right] \\
&\lesssim \left( {\sqrt{|I_{n,r}|} \over \phi_n \log d } + \max_{1 \le j \le d} |h_j(X_{\iota})| \log^{1/2} d \right)^3 \exp \left ( -{\sqrt{|I_{n,r}|} \over C \phi_n \max_{1 \le j \le d} |h_j(X_{\iota})| \log^{3/2} d} \right ) \\
&\lesssim (n^{r/2} + D_{n} n^{(r+1)/q} \log^{1/2} {d} )^{3} \exp \left ( -{n^{r/3} \over C  (\log^{1/2} {d}) n^{1/3} } \right ) \\
&\lesssim n^{3r/2} \exp \left ( - \frac{n^{1/3}}{C \log^{1/2}{d}} \right ) \le n^{3r/2} \exp (-n^{11/42}/C) \le n^{3r/2} e^{-n^{1/4}/C},
\end{align*}
where we have used the assumption (\ref{eqn:gaussian_approx_reduction}) and $q > 2(r+1)$. Therefore, we conclude that 
$\hat{M}_{n,Y}(\phi_{n}) \le C n^{3r/2} e^{-n^{1/4}/C}$ with probability at least $1-n^{-1}$. Now we can choose $\overline{M}_{n} = Cn^{3r/2} e^{-n^{1/4}/C}$ such that $\Prob(\hat{M}_{n}(\phi_{n}) > \overline{M}_{n}) \le C (n \wedge N)^{-1}$. For Step 1.5, we note that 
\[
\max_{1 \le i \le m} \max_{1 \le j,k \le d} |h_{j}(X_{(i-1)r+1}^{ir}) h_{k}(X_{(i-1)r+1}^{ir})| \le \max_{1 \le i \le m} \max_{1 \le j \le d} h_{j}^{2}(X_{(i-1)r+1}^{ir}),
\]
which implies that 
\[
\E[\mathsf{M}_{2}^{q/2}] \lesssim \E[ \max_{1 \le i \le m} \max_{1 \le j \le d} |h_{j}(X_{(i-1)r+1}^{ir})|^{q} ] \le m D_{n}^{q};
\]
i.e., $(\E[ \mathsf{M}_{2}^{q/2}])^{2/q} \lesssim m^{2/q} D_{n}^{2}$. Lemma 9 in \cite{cck2015_anticoncentration} yields that 
\[
\E[\breve{\mathsf{Z}}_{2}] \lesssim \sqrt{m D_{n}^{2} \log{d}} + \sqrt{\E[\mathsf{M}_{2}]} \log{d} \lesssim D_{n} \sqrt{m \log{d}} + D_{n}^{2} m^{2/q} \log d.
\]
Now Lemma \ref{lem:talagrand_ineq_ustat_polymom} yields that 
\begin{equation}
\label{eqn:step1.5_polymom_bern_sampling}
\Prob \left( \mathsf{Z}_{2} \ge C \{n^{-1/2} D_{n} \log^{1/2}(dn) + n^{-1+4/q} D_{n}^{2} \log{d} \} \right) \le Cn^{-1}.
\end{equation}
Then under the assumptions (\ref{eqn:gaussian_approx_reduction}) and (\ref{eqn:gaussian_approx_reduction_2}), we conclude that $|\hat{\Gamma}_{h} - \Gamma_{h}|_{\infty} \le \underline{\sigma}^{2} / 2$ and hence $\min_{1 \le j \le d} \hat{\Gamma}_{h,jj} \ge \underline{\sigma}^{2} / 2$ with probability at least $1-Cn^{-1}$.  Step 1.6 is similar to that step in Theorem \ref{thm:gaussian_approx_bern_random_design}, namely we have with probability at least $1-C(n \wedge N)^{-1}$ that 
\begin{align*}
\rho_{\mid X_{1}^{n}}^{\calR} (\sqrt{N}B_{n},\hat{Y}) \le& C \left\{ \left ( {D_n^2 \log^7 d \over N} \right)^{1/6} + \left( {D_{n}^{6} n^{6r/q} \log^{9}{d} \over N (n \wedge N)^{2-6/q} } \right)^{1/6} + {p_n^{1/2} \over D_{n}}n^{3r/2} e^{-n^{1/4}/C} \right\} \\
\le& C \left\{ \varpi_{n} + \varphi_{n}^{3/2} + {p_n^{1/2} \over D_{n}}n^{3r/2} e^{-n^{1/4}/C} \right\} \le C (\varpi_{n} + \varphi_{n}).
\end{align*}
Step 2 is also similar to that step in Theorem \ref{thm:gaussian_approx_bern_random_design}. Indeed, by the Gaussian comparison inequality (Lemma \ref{lem:Gaussian_comparison}) and (\ref{eqn:step1.5_polymom_bern_sampling}), we have 
\[
\sup_{R \in \calR} \left| \Prob_{\mid X_{1}^{n}}(\hat{Y} \in R)- \gamma_{B}(R) \right| \le  C \left \{ \left ( {D_n^2 \log^5(d n) \over n} \right )^{1/6} + \left( \frac{D_{n}^{2} \log^{3}{d} }{n^{1-4/q}} \right )^{1/3} \right \}  \lesssim \varpi_{n} + \varphi_{n}
\]
with probability at least $1-Cn^{-1}$. Next, Proposition 2.1 in \cite{cck2017_AoP} yields that 
\[
\sup_{R \in \calR} \left | \Prob (r\sqrt{n}U_{n}^{(1)}(\pi_{1}h) \in R) - \gamma_{A}(R) \right | \leq C (\varpi_{n} + \varphi_{n})
\]
under our assumptions. Following the argument in Step 3 of Theorem \ref{thm:gaussian_approx_bern_random_design}, we have  
\[
\Prob (\sqrt{n}A_{n} \in R) - \gamma_{A}(R) \le C\left ( \frac{D_{n}^{2}\log^{5}d}{n} \right )^{1/4} + C (\varpi_{n} + \varphi_{n}) \le C (\varpi_{n} + \varphi_{n}).
\]
Likewise, we have the reverse inequality and therefore $\sup_{R \in \calR} | \Prob (\sqrt{n}A_{n} \in R) - \gamma_{A}(R)| \le C (\varpi_{n} + \varphi_{n})$. Since $(n \wedge N)^{-1} \lesssim \varpi_{n}$, Steps 4 and 5 follow similar lines with $\varpi_{n}$ being replaced by $\varpi_{n} + \varphi_{n}$ verbatim. This leads to the conclusion of the theorem in the Bernoulli sampling case. 

\uline{Sampling with replacement case}. 
Next, consider sampling with replacement. The proof is similar to the Bernoulli sampling case, so we only point out the differences. In this case, recall that $\hat{Y} \mid X_{1}^{n} \sim N(0, \hat{\Gamma}_{h} - U_{n} U_{n}^{T})$. We first verify that 
\[
\rho_{\mid X_{1}^{n}}^{\calR}(\sqrt{N}B_{n},\hat{Y}) \le C (\varpi_{n} + \varphi_{n})
\]
with probability at least $1-C(n \wedge N)^{-1}$. Let $\hat{L}_{n}$ and $\hat{M}_{n}(\phi)$ be defined in (\ref{eqn:defn_Ln_sampling_with_replacement}) and (\ref{eqn:defn_Mn_sampling_with_replacement}), respectively. By Jensen's inequality, $\hat{L}_{n} \le 8 \mathsf{Z}_{1}$ and by the calculations of Step 1 in the Bernoulli sampling case, we have $\Prob(\hat{L}_{n} \ge \overline{L}_{n}) \le n^{-1}$, where 
\[
\overline{L}_{n} = C D_{n} \left\{1 + { D_{n}^{2} n^{3r/q} \log{d} \over (n \wedge N)^{1-3/q} } \right\}.
\]
Let $\phi_{n} = C_{2} (N^{-1} \overline{L}_{n}^{2} \log^{4}{d})^{-1/6}$. Then $\phi_{n} \ge 1$ by choosing $c_{1}$ in (\ref{eqn:gaussian_approx_reduction}) and (\ref{eqn:gaussian_approx_reduction_2}) sufficiently small. By Markov's and Jensen's inequalities together with Condition (C2'), we have
\begin{align*}
&\Prob \left( \max_{\iota \in I_{n,r}} \max_{1 \le j \le d} |h_{j}(X_{\iota}) - U_{n,j}| > {\sqrt{N} \over 4 \phi_{n} \log{d}} \right) \le \E\left[ \max_{\iota \in I_{n,r}} \max_{1 \le j \le d} |h_{j}(X_{\iota}) - U_{n,j}|^{q} \right] { (4\phi_{n}\log{d})^{q} \over N^{q/2} } \\
&\le { (4\phi_{n}\log{d})^{q} \over N^{q/2} } 2^{q-1} \left (  n^{r} D_{n}^{q} + \E[|U_{n}|_{\infty}^{q}] \right ) \le {n^{r} (8\phi_{n}\log{d})^{q} \over N^{q/2} }.
\end{align*}
By the calculations in Step 1.3, we have 
\[
\Prob \left( \max_{\iota \in I_{n,r}} \max_{1 \le j \le d} |h_{j}(X_{\iota}) - U_{n,j}| > {\sqrt{N} \over 4 \phi_{n} \log{d}} \right) \le {(8C_{2})^{q} \over C^{q/3} N}.
\]
For Step 1.4, recall that $V_{n} = \max_{1 \le j \le d} |I_{n,r}|^{-1} \sum_{\iota \in I_{n,r}} h_j^2(X_{\iota})$. Then Lemma E.4 in \cite{chen2017a} and Lemma 9 in \cite{cck2015_anticoncentration} yield that 
\[
\Prob \left ( V_{n} \ge C \{D_{n} + n^{-1+4/q} D_{n}^{2} \log{d} \} \right) \le C n^{-1}.
\]
Under (\ref{eqn:gaussian_approx_reduction_2}) and using $D_{n} \ge 1$, we have that $V_{n} \le C D_{n}$ with probability at least $1-Cn^{-1}$. Now in view of (\ref{eqn:bound_hatY_j_GA_replacement}) and following the lines of Step 1.4 in the proof of Theorem \ref{thm:gaussian_approx_bern_random_design}, we conclude that 
$\hat{M}_{n,Y}(\phi_{n}) \le C N^{3/2} \exp(-N^{1/6} / C)$ with probability at least $1-Cn^{-1}$. Hence we have $\Prob(\hat{M}_{n}(\phi_{n}) > \overline{M}_{n}) \le C (n \wedge N)^{-1}$, where $\overline{M}_{n} = C N^{3/2} e^{-N^{1/6}/C}$. Since $\E[\mathsf{M}_{1}^{q}] \le m D_{n}^{q} \le n D_{n}^{q}$, Lemma 8 in \cite{cck2015_anticoncentration} yields that  
\[
\E[\breve{\mathsf{Z}}_{3}] \lesssim \sqrt{m D_{n} \log{d}} + \sqrt{\E[\mathsf{M}_{1}^{2}]} \log{d} \lesssim \sqrt{n D_{n} \log{d}} + n^{1/q} D_{n} \log{d}.
\]
Then using Lemma \ref{lem:talagrand_ineq_ustat_polymom}, we have 
\[
\Prob \left( |U_{n}|_{\infty} \ge C \left\{ \sqrt{D_{n} \log(dn) \over n} + {D_{n} \log{d} \over n^{1-2/q}} \right\} \right) \le Cn^{-1},
\]
so that $|U_{n}|_{\infty}^{2} \le \underline{\sigma}^{2} / 4$ with probability at least $1-Cn^{-1}$ by choosing $c_{1}$ sufficiently small in (\ref{eqn:gaussian_approx_reduction}) and (\ref{eqn:gaussian_approx_reduction_2}). As in Step 1.5 in the Bernoulli sampling case, we have shown that $|\hat{\Gamma}_{h} - \Gamma_{h}|_{\infty} \le \underline{\sigma}^{2} / 4$ with probability at least $1-Cn^{-1}$. Hence $|\hat{\Gamma}_{h} - U_{n} U_{n}^{T} - \Gamma_{h}|_{\infty} \le \underline{\sigma}^{2} / 2$ and $\min_{1 \le j \le d} \{ \hat{\Gamma}_{h,jj} - U_{n,j}^{2} \} \ge \underline{\sigma}^{2} / 2$ with probability at least $1-Cn^{-1}$. Then we have with probability at least $1-C(n \wedge N)^{-1}$ that 
\begin{align*}
\rho_{\mid X_{1}^{n}}^{\calR} (\sqrt{N}B_{n},\hat{Y}) \le& C \left\{ \left ( {D_n^2 \log^7 d \over N} \right)^{1/6} + \left( {D_{n}^{6} n^{6r/q} \log^{9}{d} \over N (n \wedge N)^{2-6/q} } \right)^{1/6} + {N^{3/2} e^{-N^{1/6}/C} \over D_{n}} \right\} \\
\le& C \left\{ \varpi_{n} + \varphi_{n}^{3/2} + {N^{3/2} e^{-N^{1/6}/C} \over D_{n}} \right\} \le C (\varpi_{n} + \varphi_{n}).
\end{align*}
Step 2 goes through in view of the current Step 1.5 and Step 3 is exactly the same as in the Bernoulli sampling case. Step 4 follows similarly as the Bernoulli sampling case with $p_{n}=0$ and Step 5 is not needed in the sampling with replacement case. This completes the proof of Part (i). 

We next prove Part (ii). We may assume that 
\begin{equation}
\label{eqn:gaussian_approx_reduction_2_degenerate}
\begin{split}
& ND_{n}^{2} \log^{k+3} d \le c_{2} n^{k}, \quad D_{n}^{2} \log^{5}(dn) \le c_{2} n, \\
& D_{n}^{2} \log^{7}{d} \le c_{2}N, \quad D_{n}^{2} n^{2r/q} \log^{3}(dn) \le c_{2} (N \wedge n)^{1-2/q},
\end{split}
\end{equation}
for some sufficiently small constant $c_{2}$ depending only on $\underline{\sigma},r$, and $q$. 

\uline{Bernoulli sampling case}. We first verify that 
\begin{equation}
\label{eqn:step1_degenerate_poly}
\rho_{\mid X_{1}^{n}}^{\calR} (\sqrt{N}B_{n},\hat{Y}) \le C \left \{ \left ( \frac{D_{n}^{2}\log^{7}d}{N} \right )^{1/6} +  \left( {D_{n}^{6} n^{6r/q} \log^{9}{d} \over N (n \wedge N)^{2-6/q} } \right)^{1/6} + \frac{p_{n}^{1/2}}{D_{n}} n^{3r/2} e^{-n^{-1/10}/C} \right \}
\end{equation}
with probability at least $1-Cn^{-1}$.  It is not difficult to verify from Step 1.2 in the proof of Theorem \ref{thm:gaussian_approx_bern_random_design} that under the present assumptions $\Prob (\hat{L}_{n} \ge \overline{L}_{n}) \le n^{-1}$, where
\[
\overline{L}_{n} = C p_{n}^{-1/2} D_{n} \left\{1 + { D_{n}^{2} n^{3r/q} \log{d} \over (n \wedge N)^{1-3/q} } \right\}.
\]
Step 1.3 goes through as in Part (i), namely $\Prob(\hat{M}_{n,X}(\phi_{n})=0) \ge 1 - CN^{-1}$. Step 1.4 needs a modification. From Step 1.4, we have on the event $\max_{\iota \in I_{n,r}}\max_{1 \le j \le d} | h_{j} (X_{\iota}) | \le CD_{n} n^{(r+1)/q}$ that, 
\begin{align*}
\E_{\mid X_{1}^{n}} \left [ \max_{1 \le j \le d} |\hat{Y}_{\iota,j}|^3 \vone \left( \max_{1 \le j \le d} |\hat{Y}_{\iota,j}| > {\sqrt{|I_{n,r}|} \over 4 \phi_{n} \log{d}} \right)  \right] &\le C n^{3r/2} \exp \left ( - \frac{n^{1/3}}{C \log^{1/2}{d}} \right ) \\
&\le n^{3r/2} \exp (-n^{7/30}/C),
\end{align*}
under the assumption that $D_{n}^{2} \log^{5}(dn) \le c_{2} n$. Thus $\hat{M}_{n,Y} (\phi_{n}) \le C n^{3r/2} e^{-n^{7/30}/C}$ with probability at least $1-n^{-1}$, so that 
\begin{align*}
\rho_{\mid X_{1}^{n}}^{\calR} (\sqrt{N}B_{n},\hat{Y}) \le& C \left\{ \left ( {D_n^2 \log^7 d \over N} \right)^{1/6} + \left( {D_{n}^{6} n^{6r/q} \log^{9}{d} \over N (n \wedge N)^{2-6/q} } \right)^{1/6} + {p_{n}^{1/2} n^{3r/2} e^{-n^{7/30}/C} \over D_{n}} \right\} \\
\le& C \left\{ \left ( {D_n^2 \log^7 d \over N} \right)^{1/6} +  {D_{n} n^{r/q} \log^{3/2}{d} \over (n \wedge N)^{1/2-1/q}} + {p_{n}^{1/2} n^{3r/2} e^{-n^{7/30}/C} \over D_{n}} \right\} 
\end{align*}
with probability at least $1-C(n \wedge N)^{-1}$. Step 2 remains exactly the same as in Part (i). Under (\ref{eqn:gaussian_approx_reduction_2_degenerate}), we have
\begin{align*}
\sup_{R \in \calR} \left | \Prob (\sqrt{N} B_{n} \in R) - \gamma_{B} (R) \right | \le C \left \{ \left( {D_{n}^{2} \log^{7}{d} \over N} \right)^{1/6}  +  \left ( {D_n^{2} \log^{5}(d n) \over n} \right )^{1/6} + \left( {D_{n}^{2} n^{2r/q} \log^{3}{d} \over (n \wedge N)^{1-2/q}} \right)^{1/3} \right\}.
\end{align*}
The rest of the proof follows similar lines as in the proof of Theorem \ref{thm:gaussian_approx_bern_random_design_degenerate}. This leads to the conclusion (\ref{eqn:gaussian_approx_degenerate_bern_random_design_polymom}) in the Bernoulli sampling case. 

\uline{Sampling with replacement case}. The proof is similar to the Bernoulli sampling case under degeneracy. Recall that $U_{n,N}' =  A_{n}  + B_{n}$, where $A_{n} = U_{n}$ and $B_{n} = N^{-1}\sum_{j=1}^{N} \{ h(X_{\iota_{j}}^{*}) - U_{n} \}$. Under the assumptions in (\ref{eqn:gaussian_approx_reduction_2_degenerate}), the overall bound for Step 1 in the sampling with replacement case becomes
\[
\rho_{\mid X_{1}^{n}}^{\calR}(\sqrt{N}B_{n},\hat{Y}) \le C \left \{ \left ( \frac{D_{n}^{2}\log^{7}d}{N} \right )^{1/6} +  {D_{n} n^{r/q} \log^{3/2}{d} \over (n \wedge N)^{1/2-1/q}} + \frac{N^{3/2} e^{-N^{1/6}/C}}{D_{n}} \right \}
\]
with probability at least $1-C(n \wedge N)^{-1}$. The Gaussian comparison inequality of Step 2 in the Bernouli sampling case remains exactly the same under degeneracy. The rest of the proof goes through as in the proof of Theorem \ref{thm:gaussian_approx_bern_random_design_degenerate}. 
\end{proof}

\subsection{Proofs of Theorems \ref{thm:bootstrap_validity_A} and \ref{thm:bootstrap_validity}}
\label{subsec:proofs_bootstrap_validity_subexp}

As before, we will assume that $\theta = P^{r}h=0$.
Throughout this section, the notation $\lesssim$ signifies that the left hand side is bounded by the right hand side up to a constant that depends only on $\underline{\sigma}, r$, and $C_{1}$. Let $C$ denote a generic constant that depends only on $\underline{\sigma}, r$, and $C_{1}$; its value may change from place to place.
Recall that $Y_{A} \sim N(0,r^{2}\Gamma_{r}) = \gamma_{A}$ and $Y_{B} \sim N(0, \Gamma_{h}) = \gamma_{B}$, and $Y_{A}$ and $Y_{B}$ are independent. 
Define
\[
\rho_{\mid \calD_{n}}^{\calR} (U_{n,\star}^{\sharp},Y_{\star}) := \sup_{R \in \calR} \left | \Prob_{\mid \calD_{n}} (U_{n,\star}^{\sharp} \in R) - \Prob (Y_{\star} \in R) \right |, \ \star = A,B.
\]

\begin{proof}[Proof of Theorem \ref{thm:bootstrap_validity_A}]
\uline{Bernoulli sampling case}. Conditionally on $\calD_{n}$, the vector $U_{n,B}^{\sharp}$ is Gaussian with mean zero and covariance matrix 
\[
\frac{1}{\hat{N}}  \sum_{\iota \in I_{n,r}} Z_{\iota} \{ h(X_{\iota}) - U_{n,N}' \}\{ h(X_{\iota}) - U_{n,N}' \}^{T}.
\]
On the other hand, $Y_{B} \sim N(0,\Gamma_{h})$ and $\min_{1 \le j \le d}  P^{r}h_{j}^{2} \ge \underline{\sigma}^{2}$. Hence, the Gaussian comparison inequality (Lemma \ref{lem:Gaussian_comparison}) yields that 
\begin{equation}
\label{eqn:boot_GC}
\rho_{\mid \calD_{n}}^{\calR} (U_{n,B}^{\sharp},Y_{B}) \lesssim (\hat{\Delta}_{B} \log^{2} d)^{1/3},
\end{equation}
where $\hat{\Delta}_{B}$ is defined by
\[
\hat{\Delta}_{B} =  \left |  \hat{N}^{-1} {\textstyle \sum}_{\iota \in I_{n,r}} Z_{\iota} \{ h(X_{\iota}) - U_{n,N}' \}\{ h(X_{\iota}) - U_{n,N}' \}^{T} -\Gamma_{h}  \right|_{\infty}.
\]
Observe that 
\begin{equation}
\label{eqn:decomp_hat_Delta_B}
\begin{split}
\hat{\Delta}_{B} &\le |N /\hat{N} | \cdot \left (  \left |  N^{-1} {\textstyle \sum}_{\iota \in I_{n,r}} (Z_{\iota} - p_{n}) h(X_{\iota})h(X_{\iota})^{T}  \right|_{\infty} +| \hat{\Gamma}_{h} - \Gamma_{h} |_{\infty} \right ) \\
&\quad + |N/\hat{N} - 1| \cdot | \Gamma_{h} |_{\infty} +|U_{n,N}'|_{\infty}^{2} \\
&=: |N /\hat{N} |(\hat{\Delta}_{B,1} + \hat{\Delta}_{B,2}) + \hat{\Delta}_{B,3} + \hat{\Delta}_{B,4},
\end{split}
\end{equation}
where $\hat{\Gamma}_{h} = |I_{n,r}|^{-1} \sum_{\iota \in I_{n,r}} h(X_{\iota})h(X_{\iota})^{T}$. 

From Step 5 in the proof of Theorem \ref{thm:gaussian_approx_bern_random_design}, $|\hat{N}/N -1| \le C (N^{-1/2}\log^{1/2} n + N^{-1} \log n) \le C N^{-1/2} \log^{1/2}n \le Cn^{-\zeta/2}$ with probability at least $1-2n^{-1}$.
Choose the smallest $n_{0}$ such that $Cn^{-\zeta/2} \le 1/2$ for all $n \ge n_{0}$. Clearly, $n_{0}$ depends only on $\underline{\sigma}, C_{1}$, and $\zeta$, and since for $n < n_{0}$, the conclusion of the theorem is trivial by taking the constant $C$ in (\ref{eqn:growth_condition_B}) sufficiently large (the constant $C$ in (\ref{eqn:growth_condition_B}) can be taken independent of $\zeta$), we may assume in what follows that $n \ge n_{0}$. Then, $|\hat{N}/N -1| \le C N^{-1/2} \log^{1/2}n \le 1/2$ with probability at least $1-2n^{-1}$, and hence using the inequality $|z^{-1}-1| \le 2|z-1|$ for $|z-1| \le 1/2$, we have that $|N/\hat{N} - 1| \le CN^{-1/2}\log^{1/2}n$ with probability at least $1-2n^{-1}$.  In particular, $|N/\hat{N}| \le C$ with probability at least $1-2n^{-1}$.
In addition, since $| \Gamma_{h} |_{\infty} \lesssim D_{n}$, we have that 
\[
\hat{\Delta}_{B,3} \log^{2} d \le CD_{n} N^{-1/2} (\log{n})^{1/2} \log^{2} d \le C n^{-\zeta/2}
\]
with probability at least $1-2n^{-1}$.  For $\hat{\Delta}_{B,2}$, (\ref{eqn:gaussian_approx_bern_random_design_step1.5}) in Step 1.5 in the proof of Theorem \ref{thm:gaussian_approx_bern_random_design} yields that 
\[
\Prob \left(\hat{\Delta}_{B,2} \ge C \{  n^{-1/2} D_{n} \log^{1/2} (dn) + n^{-1} D_n^2 (\log{n}) \log^3(d n) \}  \right) \le Cn^{-1}.
\]
Then under Condition (\ref{eqn:growth_condition_B}), we have that 
\[
\hat{\Delta}_{B,2} \log^{2}{d} \le C  n^{-1/2} D_{n} \log^{5/2} (dn) \le C n^{-\zeta/2}
\]
with probability at least $1-Cn^{-1}$. Next we deal with $\hat{\Delta}_{B,1}$. Let  
\[
\mathsf{V}_{n} = \max_{1 \le j, k \le d}  |I_{n,r}|^{-1}\sum_{\iota \in I_{n,r}} h_{j}^{2}(X_{\iota}) h_{k}^{2}(X_{\iota}) \quad \text{and} \quad \mathsf{M}_{1} = \max_{\iota \in I_{n,r}} \max_{1 \le j \le d} h_{j}^{2}(X_{\iota}).
\] 
Since $|Z_{\iota}-p_{n}| \le 1$ and $\Var (Z_{\iota}-p_n) \le p_{n} = N/|I_{n,r}|$, by Lemma E.2 in \cite{cck2017_AoP} (applied conditionally on $X_{1}^{n}$), we have
\begin{equation}
\label{eqn:hat_Delta_{B1}_conditional_bound_primative}
\Prob_{\mid X_{1}^{n}} \left( N\hat{\Delta}_{B,1} \ge 2 \E_{\mid X_{1}^{n}} [N\hat{\Delta}_{B,1}] + t \right) \le \exp\left( -{t^{2} \over 3 N \mathsf{V}_{n} } \right) + 3 \exp \left(- {t \over C \mathsf{M}_{1}} \right).
\end{equation}
By Lemma 8 in \cite{cck2015_anticoncentration} (applied again conditionally on $X_{1}^{n}$), we have  
\[
\E_{\mid X_{1}^{n}} [N\hat{\Delta}_{B,1}] \lesssim \sqrt{N \mathsf{V}_{n} \log{d}} + \mathsf{M}_{1} \log{d}.
\]
Then we have 
\begin{equation}
\label{eqn:hat_Delta_{B1}_conditional_bound}
\Prob_{\mid X_{1}^{n}} \left( N\hat{\Delta}_{B,1} \ge C \{ \sqrt{N \mathsf{V}_{n} \log(dn)} + \mathsf{M}_{1} \log(dn) \} \right) \le Cn^{-1}.
\end{equation}
Next we shall find an upper bound for $\mathsf{V}_{n}$ and $\mathsf{M}_{1}$ with probability at least $1-Cn^{-1}$. Since $\|\mathsf{M}_{1}\|_{\psi_{1/2}} \lesssim D_{n}^{2} \log^{2}(dn)$, we have $\mathsf{M}_{1} \le C D_{n}^{2} (\log^{2}n) \log^{2}(dn)$ with probability at least $1-2n^{-1}$. Let $m=\lfloor n/r \rfloor$. By Lemma E.3 in \cite{chen2017a}, we have for all $t > 0$ that 
\[
\Prob\left(m \mathsf{V}_{n} \ge 2 \E\left[ \max_{1 \le j,k \le d} \sum_{i=1}^{m} h_{j}^{2}(X_{(i-1)r+1}^{ir}) h_{k}^{2}(X_{(i-1)r+1}^{ir})  \right] + t \right) \le 3 \exp \left( - {t^{1/4} \over C \|\mathsf{M}_{1}'\|_{\psi_{1}}} \right),
\]
where $\mathsf{M}_{1}' = \max_{1 \le i \le m} \max_{1 \le j \le d} |h_{j}(X_{(i-1)r+1}^{ir})|$. Clearly, $\|\mathsf{M}_{1}'\|_{\psi_{1}} \lesssim D_{n} \log(dn)$. In addition,  Lemma 9 in \cite{cck2015_anticoncentration} yields that 
\[
\begin{split}
&\E\left[ \max_{1 \le j,k \le d} \sum_{i=1}^{m} h_{j}^{2}(X_{(i-1)r+1}^{ir}) h_{k}^{2}(X_{(i-1)r+1}^{ir}) \right] \\
&\lesssim \max_{1 \le j,k \le d} \sum_{i=1}^{m} \E[  h_{j}^{2}(X_{(i-1)r+1}^{ir})h_{k}^{2}(X_{(i-1)r+1}^{ir}) ] + \E\left [\max_{1 \le i \le m}\max_{1 \le j \le d} h_{j}^{4}(X_{(i-1)r+1}^{ir})\right] \log d \\
&\lesssim n D_{n}^{2} + D_{n}^{4} \log^{5} (dn).
\end{split}
\]
Hence we have $\Prob(\mathsf{V}_{n} \ge C D_{n}^{2} \{ 1 + n^{-1} D_{n}^{2} (\log^{3}{n}) \log^{5}(dn) \}  ) \le 3 n^{-1}$, so that $\mathsf{V}_{n} \le C D_{n}^{2}$ with probability at least $1-3n^{-1}$ under Condition (\ref{eqn:growth_condition_B}). Then it follows from (\ref{eqn:hat_Delta_{B1}_conditional_bound}) that 
\[
\Prob_{\mid X_{1}^{n}} \left( N\hat{\Delta}_{B,1} \ge C \{ \sqrt{N D_{n}^{2} \log(dn)} + D_{n}^{2} (\log^{2}n) \log^{3}(dn) \} \right) \le Cn^{-1}
\]
with probability at least $1-5n^{-1}$. Then Fubini yields that 
\[
\Prob \left( \hat{\Delta}_{B,1} \ge C \{ N^{-1/2} D_{n} \log^{1/2}(dn) + N^{-1} D_{n}^{2} (\log^{2}n) \log^{3}(dn) \} \right) \le C n^{-1}.
\]
Under Condition (\ref{eqn:growth_condition_B}), we have that
\[
\hat{\Delta}_{B,1} \log^{2}{d} \le C \{N^{-1} D_{n}^{2} (\log^{2}{n}) \log^{5}(dn) \}^{1/2} \le C n^{-\zeta/2}
\]
with probability at least $1-Cn^{-1}$. Finally, for $\hat{\Delta}_{B,4}$, observe that 
\[
\hat{\Delta}_{B,4} = |N/\hat{N}|^{2} |W_{n}|_{\infty}^{2} \le 2|N/\hat{N}|^{2} (|A_{n}|_{\infty}^{2}+|B_{n}|_{\infty}^{2}),
\]
where $A_{n} = U_{n} = |I_{n,r}|^{-1} \sum_{\iota \in I_{n,r}} h(X_{\iota})$ and $B_{n} = |I_{n,r}|^{-1}\sum_{\iota \in I_{n,r}} p_{n}^{-1}(Z_{\iota}-p_{n}) h(X_{\iota}) = N^{-1} \sum_{\iota \in I_{n,r}} (Z_{\iota}-p_{n}) h(X_{\iota})$. Note that bounding $|B_{n}|_{\infty}$ is similar to bounding $\hat{\Delta}_{B,1}$. Here, for completeness and later usage of similar argument, we give the proof. Let $\tilde{\mathsf{V}}_{n} = \max_{1 \le k \le d} |I_{n,r}|^{-1} \sum_{\iota \in I_{n,r}} h_{k}^{2}(X_{\iota})$ and $\tilde{\mathsf{M}}_{1} = \max_{\iota \in I_{n,r}} \max_{1 \le k \le d}| h_{k}(X_{\iota})|$. By Lemma E.2 in \cite{cck2017_AoP} (applied conditionally on $X_{1}^{n}$), we have
\begin{equation}
\label{eqn:hat_Delta_{B4}_Bn_conditional_bound_primative}
\Prob_{\mid X_{1}^{n}} \left( |N B_{n}|_{\infty} \ge 2 \E_{\mid X_{1}^{n}} [|N B_{n}|_{\infty}] + t \right) \le \exp\left( -{t^{2} \over 3 N \tilde{\mathsf{V}}_{n} } \right) + 3 \exp \left(- {t \over C \tilde{\mathsf{M}}_{1}} \right).
\end{equation}
By Lemma 8 in \cite{cck2015_anticoncentration} (applied again conditionally on $X_{1}^{n}$), we have $\E_{\mid X_{1}^{n}} [|N B_{n}|_{\infty}] \lesssim \sqrt{N \tilde{\mathsf{V}}_{n} \log{d}} +\tilde{\mathsf{M}}_{1} \log{d}$, which together with (\ref{eqn:hat_Delta_{B4}_Bn_conditional_bound_primative}) implies that 
\begin{equation}
\label{eqn:hat_Delta_{B4}_Bn_conditional_bound}
\Prob_{\mid X_{1}^{n}} \left( |N B_{n}|_{\infty} \ge C \{ \sqrt{N \tilde{\mathsf{V}}_{n} \log(dn)} + \tilde{\mathsf{M}}_{1} \log(dn) \} \right) \le Cn^{-1}.
\end{equation}
Since $\|\tilde{\mathsf{M}}_{1}\|_{\psi_{1}} \lesssim D_{n} \log(dn)$ so that $\tilde{\mathsf{M}}_{1} \lesssim D_{n} (\log{n}) \log(dn)$ with probability at least $1-2n^{-1}$. By Lemma E.3 in \cite{chen2017a}, we have for all $t > 0$ that 
\[
\Prob\left(m \tilde{\mathsf{V}}_{n} \ge 2 \E\left[ \max_{1 \le j \le d} \sum_{i=1}^{m} h_{j}^{2}(X_{(i-1)r+1}^{ir})  \right] + t \right) \le 3 \exp \left( - {t^{1/2} \over C \|\mathsf{M}_{1}'\|_{\psi_{1}}} \right),
\]
where we recall that $\mathsf{M}_{1}' = \max_{1 \le i \le m} \max_{1 \le j \le d} |h_{j}(X_{(i-1)r+1}^{ir})|$. Clearly, $\|\mathsf{M}_{1}'\|_{\psi_{1}} \lesssim D_{n} \log(dn)$. In addition,  Lemma 9 in \cite{cck2015_anticoncentration} yields that 
\[
\begin{split}
\E\left[ \max_{1 \le j \le d} \sum_{i=1}^{m} h_{j}^{2}(X_{(i-1)r+1}^{ir}) \right] &\lesssim \max_{1 \le j \le d} \sum_{i=1}^{m} \E[  h_{j}^{2}(X_{(i-1)r+1}^{ir}) ] + \E\left [\max_{1 \le i \le m}\max_{1 \le j \le d} h_{j}^{2}(X_{(i-1)r+1}^{ir})\right] \log d \\
&\lesssim n D_{n} + D_{n}^{2} \log^{3} (dn).
\end{split}
\]
Hence we have $\Prob(\tilde{\mathsf{V}}_{n} \ge C D_{n} \{ 1 + n^{-1} D_{n} \log^{3}(dn) \}  ) \le 3 n^{-1}$, so that $\tilde{\mathsf{V}}_{n} \le C D_{n}$ with probability at least $1-3n^{-1}$ under Condition (\ref{eqn:growth_condition_B}). Then it follows from (\ref{eqn:hat_Delta_{B4}_Bn_conditional_bound}) that 
\[
\Prob_{\mid X_{1}^{n}} \left( |N B_{n}|_{\infty} \ge C \{ \sqrt{N D_{n} \log(dn)} + D_{n} (\log{n}) \log^{2}(dn) \} \right) \le Cn^{-1}
\]
with probability at least $1-5n^{-1}$. Then Fubini yields that 
\[
\Prob \left( |B_{n}|_{\infty} \ge C \{ (N^{-1} D_{n} \log(dn))^{1/2} + N^{-1} D_{n} (\log{n}) \log^{2}(dn) \} \right) \le C n^{-1},
\]
which implies that $|B_{n}|_{\infty} \le C \{ N^{-1} D_{n} \log(dn) \}^{1/2}$ with probability at least $1-C n^{-1}$ under Condition (\ref{eqn:growth_condition_B}). In addition, Lemma \ref{lem:talagrand_ineq_ustat} and Lemma 8 in \cite{cck2015_anticoncentration} yield that $|A_{n}|_{\infty} \le C \{n^{-1} D_{n} \log(dn) \}^{1/2}$ with probability at least $1-Cn^{-1}$. Hence we have that
\[
(|A_{n}|_{\infty}^{2} + |B_{n}|_{\infty}^{2})\log^{2} d \le C n^{-\zeta/2}
\]
with probability at least $1-Cn^{-1}$.

In conclusion, we have shown that $\hat{\Delta}_{B} \log^{2} d \le C n^{-\zeta/2}$ with probability at least $1-Cn^{-1}$, and in view of (\ref{eqn:boot_GC}), this leads to the desired conclusion. 

\uline{Sampling with replacement case}. 
Conditionally on $\mathcal{D}_{n}$, the vector  $U_{n,B}^{\sharp}$ is Gaussian with mean zero and covariance matrix 
\[
\frac{1}{N} \sum_{j=1}^{N} \{ h(X_{\iota_{j}}^{*}) - U_{n,N}' \} \{ h(X_{\iota_{j}}^{*}) - U_{n,N}' \}^{T}.
\]
In view of the previous proof, it suffices to prove that $\hat{\Delta}_{B} \log^{2}{d} \le C n^{-\zeta/2}$ with probability at least $1-Cn^{-1}$, where $\hat{\Delta}_{B}$ is now defined by 
\[
\hat{\Delta}_{B} = \left | N^{-1} {\textstyle \sum}_{j=1}^{N} \{ h(X_{\iota_{j}}^{*})- U_{n,N}' \} \{ h(X_{\iota_{j}}^{*}) - U_{n,N}' \}^{T}- \Gamma_{h} \right |_{\infty}.
\]
Observe that 
\begin{equation}
\label{eqn:decomp_hat_Delta_B_replacement}
\begin{split}
\hat{\Delta}_{B} &\le \left | N^{-1} {\textstyle \sum}_{j=1}^{N} h(X_{\iota_{j}}^{*}) h(X_{\iota_{j}}^{*})^{T} - \hat{\Gamma}_{h} \right |_{\infty} + | \hat{\Gamma}_{h} -  \Gamma_{h} |_{\infty}  + |U_{n,N}'|_{\infty}^{2} \\
&=: \hat{\Delta}_{B,1} + \hat{\Delta}_{B,2} + \hat{\Delta}_{B,3}. 
\end{split}
\end{equation}
We have shown that $\hat{\Delta}_{B,2} \log^{2}{d} \le C  n^{-1/2} D_{n} \log^{5/2} (dn) \le C n^{-\zeta/2}$ with probability at least $1-Cn^{-1}$. Since $X_{\iota_{j}}^{*}, j=1,\dots,N$ are (conditionally on $X_{1}^{n}$) i.i.d. draws from the empirical distribution $|I_{n,r}|^{-1} \sum_{\iota \in I_{n,r}} \delta_{X_{\iota}}$, we have $\E_{\mid X_{1}^{n}}[h(X_{\iota_{j}}^{*}) h(X_{\iota_{j}}^{*})^{T}] = \hat{\Gamma}_{h}$. In addition, by Jensen's inequality, we have 
\[
\begin{split}
& \max_{1 \le k,\ell \le d}\sum_{j=1}^{N} \E_{\mid X_{1}^{n}} \left[ \left ( h_{k}(X_{\iota_{j}}^{*}) h_{\ell}(X_{\iota_{j}}^{*}) - \hat{\Gamma}_{h,k\ell} \right)^{2} \right] \le \max_{1 \le k,\ell \le d} \sum_{j=1}^{N} \E_{\mid X_{1}^{n}} \left[ h_{k}^{2}(X_{\iota_{j}}^{*}) h_{\ell}^{2}(X_{\iota_{j}}^{*}) \right] \\
& \qquad = N \max_{1 \le k,\ell \le d} |I_{n,r}|^{-1} \sum_{\iota \in I_{n,r}} h_{k}^{2}(X_{\iota}) h_{\ell}^{2}(X_{\iota}) = N \mathsf{V}_{n},
\end{split}
\]
\[
\begin{split}
\max_{1 \le j \le N} \max_{1 \le k,\ell \le d}| h_{k}(X_{\iota_{j}}^{*}) h_{\ell}(X_{\iota_{j}}^{*})| \le \max_{\iota \in I_{n,r}} \max_{1 \le k,\ell \le d}| h_{k}(X_{\iota}) h_{\ell}(X_{\iota})| \le \max_{\iota \in I_{n,r}} \max_{1 \le k \le d} h_{k}^{2}(X_{\iota}) = \mathsf{M}_{1}, 
\end{split}
\]
and $\max_{1 \le k,\ell \le d} |\hat{\Gamma}_{h,k\ell}| \le \max_{\iota \in I_{n,r}} \max_{1 \le k,\ell \le d}| h_{k}(X_{\iota}) h_{\ell}(X_{\iota})| \le \mathsf{M}_{1}$. By Lemma E.2 in \cite{cck2017_AoP} (applied conditionally on $X_{1}^{n}$), we deduce that (\ref{eqn:hat_Delta_{B1}_conditional_bound_primative}) continues to hold for every $t > 0$ for the sampling with replacement case. Then following the argument in the previous proof, we have 
\[
\hat{\Delta}_{B,1} \log^{2}{d} \le C \{N^{-1} D_{n}^{2} (\log^{2}{n}) \log^{5}(dn) \}^{1/2} \le C n^{-\zeta/2}
\]
with probability at least $1-Cn^{-1}$. Finally, $\hat{\Delta}_{B,3} \le 2(|A_{n}|_{\infty}^{2} + |B_{n}|_{\infty}^{2})$ where $A_{n}  =U_{n}$ and $B_{n} = N^{-1} \sum_{j=1}^{N} \{ h(X_{\iota_{j}}^{*}) - U_{n} \}$. We have shown that $|A_{n}|_{\infty}^{2} \le C n^{-1} D_{n} \log(dn)$ with probability at least $1-Cn^{-1}$. Next, since $h(X_{\iota_{j}}^{*}), j=1,\dots,N$ are i.i.d. with mean $U_{n}$ conditionally on $X_{1}^{n}$, we have 
\[
\begin{split}
& \max_{1 \le k \le d}\sum_{j=1}^{N} \E_{\mid X_{1}^{n}} \left[ \left (h_{k}(X_{\iota_{j}}^{*}) - U_{n,k} \right )^{2} \right] \le \max_{1 \le k \le d} \sum_{j=1}^{N} \E_{\mid X_{1}^{n}} \left[ h_{k}^{2}(X_{\iota_{j}}^{*}) \right] \\
& \qquad = N \max_{1 \le k \le d} |I_{n,r}|^{-1} \sum_{\iota \in I_{n,r}} h_{k}^{2}(X_{\iota}) = N \tilde{\mathsf{V}}_{n}.
\end{split}
\]
In addition, $\max_{1 \le j \le N} \max_{1 \le k \le d}| h_{k}(X_{\iota_{j}}^{*}) | \le \max_{\iota \in I_{n,r}} \max_{1 \le k \le d}| h_{k}(X_{\iota})| = \tilde{\mathsf{M}}_{1}$ and $\max_{1 \le k \le d}|U_{n,k}| \le \tilde{\mathsf{M}}_{1}$. By Lemma E.2 in \cite{cck2017_AoP} (applied conditionally on $X_{1}^{n}$), we deduce that (\ref{eqn:hat_Delta_{B4}_Bn_conditional_bound_primative}) continues to hold for every $t > 0$ for the sampling with replacement case. Then following the argument in the previous proof, we have $|B_{n}|_{\infty} \le C \{ N^{-1} D_{n} \log(dn) \}^{1/2}$ with probability at least $1-C n^{-1}$ under Condition (\ref{eqn:growth_condition_B}). Now we conclude that $(|A_{n}|_{\infty}^{2} + |B_{n}|_{\infty}^{2})\log^{2} d \le C n^{-\zeta/2}$ and hence $\hat{\Delta}_{B} \log^{2}{d} \le C n^{-\zeta/2}$ with probability at least $1-Cn^{-1}$.
\end{proof}

\begin{proof}[Proof of Theorem \ref{thm:bootstrap_validity}]
The proof is divided into three steps.

\uline{Step 1: Bounding $\rho_{\mid \calD_{n}}^{\calR}(U_{n,B}^{\sharp},Y_{B})$}. Since Condition (C3-ND) implies Condition (C3-D), and $n_{1} \le n$ by definition, by Theorem \ref{thm:bootstrap_validity_A}, we have that 
\[
\rho_{\mid \calD_{n}}^{\calR} (U_{n,B}^{\sharp},Y_{B}) \le Cn^{-\zeta_1/6}
\]
with probability at least $1-Cn^{-1}$. 

\uline{Step 2: Bounding $\rho_{\mid \calD_{n}}^{\calR}(U_{n,A}^{\sharp},Y_{A})$}. In this step, we shall show that 
\[
\rho_{\mid \calD_{n}}^{\calR} (U_{n,A}^{\sharp},Y_{A}) \le Cn^{-(\zeta_1 \wedge \zeta_2)/6}
\]
with probability at least $1-Cn^{-1}$. 

 Without loss of generality, we may  assume $S_{1} = \{1,\dots,n_{1}\}$. 
Conditionally on $\calD_{n}$, the vector $U_{n,A}^{\sharp}$ is Gaussian with mean zero and covariance matrix
\[
\frac{r^{2}}{n_{1}} \sum_{i_{1}=1}^{n_{1}} \{ \hat{g}^{(i_{1})}(X_{i_{1}}) -\breve{g} \} \{ \hat{g}^{(i_{1})}(X_{i_{1}}) -\breve{g}  \}^{T}.
\]
On the other hand, $Y_{A} \sim N(0,r^{2}\Gamma_{g})$ and $\min_{1 \le j \le d}  Pg_{j}^{2} \ge \underline{\sigma}^{2}$.
Hence, the Gaussian comparison inequality (Lemma \ref{lem:Gaussian_comparison}) yields that 
\[
\rho_{\mid \calD_{n}}^{\calR} (U_{n,A}^{\sharp},Y_{A}) \lesssim (\hat{\Delta}_{A}\log^{2} d)^{1/3},
\]
where 
\[
\hat{\Delta}_{A} = \max_{1 \le j,\ell \le d} \left | n_{1}^{-1} \sum_{i_{1}=1}^{n_{1}} \{ \hat{g}_{j}^{(i_{1})}(X_{i_{1}}) -\breve{g}_{j} \} \{ \hat{g}_{\ell}^{(i_{1})}(X_{i_{1}}) -\breve{g}_{\ell} \} - Pg_{j}g_{\ell} \right |.
\]
Observe that for every $1 \le j,\ell \le d$,
\[
\begin{split}
&n_{1}^{-1} \sum_{i_{1}=1}^{n_{1}} \{ \hat{g}_{j}^{(i_{1})}(X_{i_{1}}) -\breve{g}_{j} \} \{ \hat{g}_{\ell}^{(i_{1})}(X_{i_{1}}) -\breve{g}_{\ell} \} \\
&=n_{1}^{-1} \sum_{i_{1}=1}^{n_{1}}  \hat{g}_{j}^{(i_{1})}(X_{i_{1}}) \hat{g}_{\ell}^{(i_{1})}(X_{i_{1}}) -\breve{g}_{j} \breve{g}_{\ell} \\
&=n_{1}^{-1} \sum_{i_{1}=1}^{n_{1}}  \{ \hat{g}_{j}^{(i_{1})}(X_{i_{1}}) - g_{j}(X_{i_{1}}) \} \{ \hat{g}_{\ell}^{(i_{1})}(X_{i_{1}}) -  g_{\ell}(X_{i_{1}})\}  \\
&\quad + n_{1}^{-1} \sum_{i_{1}=1}^{n_{1}}  \{ \hat{g}_{j}^{(i_{1})}(X_{i_{1}}) - g_{j}(X_{i_{1}}) \}  g_{\ell}(X_{i_{1}}) + n_{1}^{-1} \sum_{i_{1}=1}^{n_{1}}  \{ \hat{g}_{\ell}^{(i_{1})}(X_{i_{1}}) - g_{\ell}(X_{i_{1}}) \}  g_{j}(X_{i_{1}}) \\
&\quad + n_{1}^{-1} \sum_{i_{1}=1}^{n_{1}} g_{j}(X_{i_{1}}) g_{\ell} (X_{i_{1}}) - \breve{g}_{j} \breve{g}_{\ell},
\end{split}
\]
so that by the Cauchy-Schwarz inequality, 
\begin{equation*}
\begin{split}
\hat{\Delta}_{A}
&\le \underbrace{\max_{1 \le j \le d} n_{1}^{-1} \sum_{i_{1}=1}^{n_{1}}  \{ \hat{g}_{j}^{(i_{1})}(X_{i_{1}}) - g_{j}(X_{i_{1}}) \}^{2}}_{=\hat{\Delta}_{A,1}}   + 2 \hat{\Delta}_{A,1}^{1/2} \sqrt{ \max_{1 \le j \le d} n_{1}^{-1}\sum_{i_{1}=1}^{n_{1}} g_{j}^{2}(X_{i_{1}})} \\
&\quad + \max_{1 \le j,\ell \le d} \left |  n_{1}^{-1} \sum_{i_{1}=1}^{n_{1}} \{ g_{j}(X_{i_{1}}) g_{\ell} (X_{i_{1}}) - Pg_{j}g_{\ell} \} \right | + \max_{1 \le j \le d} | \breve{g}_{j} |^{2}. 
\end{split}
\end{equation*}
For the notational convenience, define 
\[
\hat{\Delta}_{A,2} := \max_{1 \le j,\ell \le d} \left |  n_{1}^{-1} \sum_{i_{1}=1}^{n_{1}} \{ g_{j}(X_{i_{1}}) g_{\ell} (X_{i_{1}}) - Pg_{j}g_{\ell} \}\right |. 
\]
Then, since $n_{1}^{-1}\sum_{i_{1}=1}^{n_{1}} g_{j}^{2}(X_{i_{1}}) \le Pg_{j}^{2} + |n_{1}^{-1}\sum_{i_{1}=1}^{n_{1}} \{ g^{2}_{j}(X_{i_{1}}) - Pg_{j}^{2} \} | \le \overline{\sigma}_{g}^{2} + \hat{\Delta}_{A,2}$, and $\breve{g}_{j} = n_{1}^{-1} \sum_{i_{1}=1}^{n} \{ \hat{g}_{j}^{(i_{1})}(X_{i_{1}}) - g_{j}(X_{i_{1}}) \} + n_{1}^{-1} \sum_{i_{1}=1}^{n_{1}} g_{j}(X_{i_{1}})$, so that 
\[
\max_{1 \le j \le d} | \breve{g}_{j} |^{2} \lesssim \hat{\Delta}_{A,1} + \hat{\Delta}_{A,3}^{2}, \quad \text{with} \  \hat{\Delta}_{A,3} := \max_{1 \le j \le d} \left |n_{1}^{-1} \sum_{i_{1}=1}^{n_{1}} g_{j}(X_{i_{1}})\right |,
\]
we have 
\begin{equation}
\label{eqn:decomp_hat_Delta_A}
\hat{\Delta}_A \lesssim \overline{\sigma}_{g} \hat{\Delta}_{A,1}^{1/2} + \hat{\Delta}_{A,1}  + \hat{\Delta}_{A,2} + \hat{\Delta}_{A,3}^{2}, 
\end{equation}
where we have used the inequality $2ab \le a^{2}+b^{2}$ for $a,b \in \R$. 

Now, by assumption, $\overline{\sigma}_{g} \hat{\Delta}_{A,1}^{1/2} \log^{2} d \le C n^{-\zeta_2/2}$ with probability at least $1-Cn^{-1}$. For $\hat{\Delta}_{A,2}$, Lemma 8 in \cite{cck2015_anticoncentration} yields that 
\begin{equation*}
\begin{split}
\E[\hat{\Delta}_{A,2}] &\lesssim n_{1}^{-1} \sqrt{(\log d)\max_{1 \le j,\ell \le d} \sum_{i_{1}=1}^{n_{1}} \E[g_{j}^{2}(X_{i_{1}}) g_{\ell}^{2}(X_{i_{1}})]} + n_{1}^{-1} \sqrt{\E\left[\max_{1 \le i_{1} \le n_{1}} \max_{1 \le j \le d} g_{j}^{4}(X_{i_{1}}) \right]} \log d \\
&\lesssim n_{1}^{-1/2} D_{n}\log^{1/2} d + n_{1}^{-1} D_{n}^{2} \log^{3}(dn).
\end{split}
\end{equation*}
By Lemma E.2 in \cite{cck2017_AoP}, we have for every $t > 0$
\[
\Prob( \hat{\Delta}_{A,2} \ge 2 \E[\hat{\Delta}_{A,2}] + t ) \le \exp\left(-{n_{1} t^{2} \over 3 D_{n}^{2} }\right) + 3 \exp\left\{ -\left( {n_{1} t \over C D_{n}^{2} \log^{2}(dn)} \right)^{1/2} \right\}.
\]
Choosing $t = \{C n_{1}^{-1} D_{n}^{2} \log{n} \}^{1/2} \bigvee \{ C n_{1}^{-1} D_{n}^{2} (\log^{2}n)\log^{3}(dn) \}$, we have
\[
\Prob \left( \hat{\Delta}_{A,2} \ge C \{ (n_{1}^{-1} D_{n}^{2} \log(dn))^{1/2} + n_{1}^{-1} D_{n}^{2} (\log^{2}n) \log^{3}(dn)  \} \right) \le C n^{-1}.
\]
Now, using the first part of Condition (\ref{eqn:growth_condition}), we deduce that 
\[
\hat{\Delta}_{A,2} \log^{2}{d} \le C \{n_{1}^{-1} D_{n}^{2} (\log{n})^{2} \log^{5}(dn)\}^{1/2} \le C n^{-\zeta_1/2}
\]
with probability at least $1-Cn^{-1}$. The term $\hat{\Delta}_{A,3}^{2}$ can be similarly dealt with. In particular, using $\overline{\sigma}_{g}^{2} \le 1+\max_{j} P|g_{j}|^{3} \lesssim D_{n}$, we have
\begin{equation}
\label{eqn:delta_A3}
\E[\hat{\Delta}_{A,3}] \lesssim (n_{1}^{-1}D_{n} \log d)^{1/2} + n_{1}^{-1} D_{n} \log^{2}(dn).
\end{equation}
Then using Lemma 8 in \cite{cck2015_anticoncentration} and Lemma E.2 in \cite{cck2017_AoP}, we can show that 
\[
\hat{\Delta}_{A,3}^{2} \log^{2}{d} \le C \{ n_{1}^{-1} D_{n} \log^{3}(dn) + n_{1}^{-2} D_{n}^{2} \log^{6}(dn) \} \le C n^{-\zeta_1} \le C n^{-\zeta_1/2}
\]
with probability at least $1-Cn^{-1}$. Hence $\hat{\Delta}_{A} \log^{2}{d} \le C n^{-(\zeta_1 \wedge \zeta_2)/2}$ with probability at least $1-Cn^{-1}$, which leads to the conclusion of this step.

\uline{Step 3: Conclusion}. 
Let $\Xi = \{ \xi_{i_{1}} : i_{1} \in S_{1} \}$ and $\Xi' = \{ \xi_{\iota}' : \iota \in I_{n,r} \}$. Recall that $\Xi,\Xi'$, and $\calD_{n}$ are mutually independent. 
Suppose that 
\[
\rho_{\mid \calD_{n}}^{\calR} (U_{n,A}^{\sharp},Y_{A}) \bigvee \rho_{\mid \calD_{n}}^{\calR} (U_{n,B}^{\sharp},Y_{B}) \le Cn^{-(\zeta_1 \wedge \zeta_2)/6},
\]
which holds with probability at least $1-Cn^{-1}$. Pick any hyperrectangle $R \in \calR$. 
Observe that 
\[
\Prob_{\mid \calD_{n}} (U_{n}^{\sharp} \in R) =  \E_{\mid \calD_{n}} \left[ \Prob_{\mid (\calD_{n},\Xi)} \left (U_{n,B}^{\sharp} \in [\alpha_{n}^{-1/2} R - \alpha_{n}^{-1/2} U_{n,A}^{\sharp}]\right )  \right].
\]
The conditional probability on the right hand side is bounded from above by
$\gamma_{B} ([\alpha_{n}^{-1/2} R - \alpha_{n}^{-1/2} U_{n,A}^{\sharp}]) + Cn^{-(\zeta_1 \wedge \zeta_2)/6}$,
and hence 
\[
\begin{split}
&\Prob_{\mid \calD_{n}} (U_{n}^{\sharp} \in R) \le \E_{\mid \calD_{n}} \left [ \gamma_{B}\left ([\alpha_{n}^{-1/2} R - \alpha_{n}^{-1/2} U_{n,A}^{\sharp}]\right ) \right ] + Cn^{-(\zeta_1 \wedge \zeta_2)/6} \\
&=\Prob_{\mid \calD_{n}} \left ( \breve{Y}_{B} \in [\alpha_{n}^{-1/2} R - \alpha_{n}^{-1/2} U_{n,A}^{\sharp}] \right )  + Cn^{-(\zeta_1 \wedge \zeta_2)/6} =\Prob_{\mid \calD_{n}} \left ( U_{n,A}^{\sharp}  \in [R-\alpha_{n}^{1/2}\breve{Y}_{B}] \right )  + Cn^{-(\zeta_1 \wedge \zeta_2)/6},
\end{split}
\]
where $\breve{Y}_{B} \sim N(0,\Gamma_{h})$ independent of $\calD_{n}$ and $\Xi$. 
The first term on the far right hand side can be written as $\E_{\mid \calD_{n}}  [ \Prob_{\mid (\calD_{n},\breve{Y}_{B})} ( U_{n,A}^{\sharp}  \in [R-\alpha_{n}^{1/2}\breve{Y}_{B}]   )  ]$,
and the inner conditional probability is bounded from above by $\gamma_{A} ( [R-\alpha_{n}^{1/2}\breve{Y}_{B}]  ) + Cn^{-(\zeta_1 \wedge \zeta_2)/6}$.
Hence, 
\[
\Prob_{\mid \calD_{n}} (U_{n}^{\sharp} \in R) \le \E_{\mid \calD_{n}} \left [ \gamma_{A}\left ( [R-\alpha_{n}^{1/2}\breve{Y}_{B}]  \right ) \right ] + Cn^{-(\zeta_1 \wedge \zeta_2)/6} = \Prob \left ( Y \in R \right ) + Cn^{-(\zeta_1 \wedge \zeta_2)/6}.
\]
Likewise, we have  $\Prob_{\mid \calD_{n}} (U_{n}^{\sharp} \in R) \ge \Prob (Y \in R) - Cn^{-(\zeta_1 \wedge \zeta_2)/6}$. 

Finally, the last statement of the theorem is trivial since the bootstrap distribution is taken only with respect to $\{ \xi_{i_{1}} : i_{1} \in S_{1} \}$ and $\{ \xi_{\iota}': \iota \in I_{n,r} \}$. This completes the proof. 
\end{proof}

\begin{proof}[Proof of Corollary \ref{cor:bootstrap_validity}]
This follows from Step 2 in the proof of Theorem \ref{thm:bootstrap_validity}.
\end{proof}

\subsection{Proofs of Theorem \ref{thm:bootstrap_validity_B_poly} and \ref{thm:bootstrap_validity_poly}}
\label{subsec:proofs_bootstrap_validity_poly}

As before, we will assume that $\theta = P^{r}h=0$.
Throughout this section, the notation $\lesssim$ signifies that the left hand side is bounded by the right hand side up to a constant that depends only on $\underline{\sigma}, r, q$, and $C_{1}$. Let $C$ denote a generic constant that depends only on $\underline{\sigma}, r, q$, and $C_{1}$; its value may change from place to place.
Recall that $\rho_{\mid \calD_{n}}^{\calR} (U_{n,\star}^{\sharp},Y_{\star}) := \sup_{R \in \calR} | \Prob_{\mid \calD_{n}} (U_{n,\star}^{\sharp} \in R) - \Prob (Y_{\star} \in R) |$ for $\star = A,B$.

\begin{proof}[Proof of Theorem \ref{thm:bootstrap_validity_B_poly}]
We shall follow the notation used in the proof of Theorem \ref{thm:bootstrap_validity_A}. 

\uline{Bernoulli sampling case}. We consider the decomposition of $\hat{\Delta}_{B}$ in (\ref{eqn:decomp_hat_Delta_B}). First note that $\hat{\Delta}_{B,3} \log^{2}{d} \le C n^{-\zeta/2}$ with probability at least $1-2n^{-1}$ since the randomness of $\hat{\Delta}_{B,3}$ only comes from $\hat{N}$ which is independent of $X_{1}^{n}$. For $\hat{\Delta}_{B,2}$, by (\ref{eqn:step1.5_polymom_bern_sampling}), we have 
\[
\Prob \left( \hat{\Delta}_{B,2} \ge C \{n^{-1/2} D_{n} \log^{1/2}(dn) + n^{-1+4/q} D_{n}^{2} \log{d} \} \right) \le C n^{-1}.
\]
Since $r \ge 2$, under Condition (\ref{eqn:growth_condition_B_poly}), we have $\hat{\Delta}_{B,2} \log^{2}{d} \le C n^{-\zeta/2}$ with probability at least $1-Cn^{-1}$. For $\hat{\Delta}_{B,1}$, note that, under (C2'), (\ref{eqn:hat_Delta_{B1}_conditional_bound}) continues to hold and we only need to find an upper bound for $\mathsf{V}_{n}$ and $\mathsf{M}_{1}$ with probability at least $1-Cn^{-1}$. Note that $(\E[\mathsf{M}_{1}^{q/2}])^{2/q} \le n^{2r/q} D_{n}^{2}$, so that by Markov's inequality we have $\Prob(\mathsf{M}_{1} \ge D_{n}^{2} n^{2(r+1)/q}) \le n^{-1}$.  Let $m=\lfloor n/r \rfloor$. By Lemma E.4 in \cite{chen2017a}, we have for all $t > 0$ that 
\[
\Prob\left(m \mathsf{V}_{n} \ge 2 \E\left[ \max_{1 \le j,k \le d} \sum_{i=1}^{m} h_{j}^{2}(X_{(i-1)r+1}^{ir}) h_{k}^{2}(X_{(i-1)r+1}^{ir})  \right] + t \right) \le  {C\E[{\mathsf{M}'_{1}}^{q}] \over t^{q/4}}.
\]
Now $(\E[{\mathsf{M}_{1}'}^{q}])^{1/q} \lesssim n^{1/q} D_{n}$ and Lemma 9 in \cite{cck2015_anticoncentration} yields that 
\[
\begin{split}
&\E\left[ \max_{1 \le j,k \le d} \sum_{i=1}^{m} h_{j}^{2}(X_{(i-1)r+1}^{ir}) h_{k}^{2}(X_{(i-1)r+1}^{ir}) \right] \\
&\lesssim \max_{1 \le j,k \le d} \sum_{i=1}^{m} \E[  h_{j}^{2}(X_{(i-1)r+1}^{ir})h_{k}^{2}(X_{(i-1)r+1}^{ir}) ] + \E\left [\max_{1 \le i \le m}\max_{1 \le j \le d} h_{j}^{4}(X_{(i-1)r+1}^{ir})\right] \log d \\
&\lesssim n D_{n}^{2} + n^{4/q} D_{n}^{4} \log{d}.
\end{split}
\]
Hence we have $\Prob(\mathsf{V}_{n} \ge C D_{n}^{2} \{ 1 + n^{-1+4/q} D_{n}^{2} \log{d} + n^{-1+8/q} D_{n}^{2} \}  ) \le C n^{-1}$, so that $\mathsf{V}_{n} \le C D_{n}^{2}$ with probability at least $1-Cn^{-1}$ under Condition (\ref{eqn:growth_condition_B_poly}) because $r \ge 2$. Then it follows from (\ref{eqn:hat_Delta_{B1}_conditional_bound}) and Fubini that 
\begin{equation}
\label{eqn:bound_hat_Delta_{B1}_poly}
\Prob \left( \hat{\Delta}_{B,1} \ge C \{ N^{-1/2} D_{n} \log^{1/2}(dn) + N^{-1} D_{n}^{2} n^{2(r+1)/q} \log(dn) \} \right) \le C n^{-1},
\end{equation}
which in turn implies that
\[
\hat{\Delta}_{B,1} \log^{2}{d} \le C \{ N^{-1/2} D_{n} \log^{5/2}(dn) + N^{-1} D_{n}^{2} n^{2(r+1)/q} \log^{3}(dn) \} \le C n^{-\zeta/2}
\]
with probability at least $1-Cn^{-1}$. For $\hat{\Delta}_{B,4}$, by similar calculations to those in bounding $\hat{\Delta}_{B,1}$ (cf. the details for bouding the term $\hat{\Delta}_{B,4}$ in the proof of Theorem \ref{thm:bootstrap_validity_A}), we can show that $|B_{n}|_{\infty} \le C\{ (N^{-1} D_{n} \log(dn))^{1/2} + N^{-1} D_{n} n^{(r+1)/q} \log(dn) \}$ with probability at least $1-Cn^{-1}$. In addition, Lemma \ref{lem:talagrand_ineq_ustat_polymom} and Lemma 8 in \cite{cck2015_anticoncentration} yield that $|A_{n}|_{\infty} \le C \{ (n^{-1} D_{n} \log(dn))^{1/2} + n^{-1+2/q} D_{n} \log{d} \}$ with probability at least $1-Cn^{-1}$. Hence we have that
\[
(|A_{n}|_{\infty}^{2} + |B_{n}|_{\infty}^{2})\log^{2} d \le n^{-\zeta/2}
\]
with probability at least $1-Cn^{-1}$.

Combining the above bounds, we have $\hat{\Delta}_{B} \log^{2} d \le C n^{-\zeta/2}$ with probability at least $1-Cn^{-1}$, and in view of (\ref{eqn:boot_GC}), this leads to the desired conclusion. 

\uline{Sampling with replacement case}. We consider the decomposition of $\hat{\Delta}_{B}$ in (\ref{eqn:decomp_hat_Delta_B_replacement}). We have shown that $\hat{\Delta}_{B,2} \log^{2}{d} \le C \{n^{-1/2} D_{n} \log^{5/2}(dn) + n^{-1+4/q} D_{n}^{2} \log^{3}{d} \} \le C n^{-\zeta/2}$ with probability at least $1-Cn^{-1}$. Note that the bound (\ref{eqn:bound_hat_Delta_{B1}_poly}) on $\hat{\Delta}_{B,1}$ holds exactly as in the Bernoulli sampling case in view of (\ref{eqn:hat_Delta_{B1}_conditional_bound}) and the proof of sampling with replacement case in Theorem \ref{thm:bootstrap_validity_A}; that is, we have $\hat{\Delta}_{B,1} \log^{2}{d} \le C \{ N^{-1/2} D_{n} \log^{5/2}(dn) + N^{-1} D_{n}^{2} n^{2(r+1)/q} \log^{3}(dn) \} \le C n^{-\zeta/2}$ with probability at least $1-Cn^{-1}$. Finally, $\hat{\Delta}_{B,3} \le 2(|A_{n}|_{\infty}^{2} + |B_{n}|_{\infty}^{2})$ where $A_{n}  =U_{n}$ and $B_{n} = N^{-1} \sum_{j=1}^{N} \{ h(X_{\iota_{j}}^{*}) - U_{n} \}$. We have shown that $|A_{n}|_{\infty} \le C \{ (n^{-1} D_{n} \log(dn))^{1/2} + n^{-1+2/q} D_{n} \log{d} \}$ with probability at least $1-Cn^{-1}$. Similarly, $|B_{n}|_{\infty} \le C\{ (N^{-1} D_{n} \log(dn))^{1/2} + N^{-1} D_{n} n^{(r+1)/q} \log(dn) \}$ with probability at least $1-Cn^{-1}$ in view of (\ref{eqn:hat_Delta_{B4}_Bn_conditional_bound}) and the proof of sampling with replacement case in Theorem \ref{thm:bootstrap_validity_A}.
\end{proof}

\begin{proof}[Proof of Theorem \ref{thm:bootstrap_validity_poly}]
We shall follow the notation and the proof structure in the proof of Theorem \ref{thm:bootstrap_validity}. In Step 1, it follows from Theorem \ref{thm:bootstrap_validity_B_poly} that $\rho_{\mid \calD_{n}}^{\calR}(U_{n,B}^{\sharp},Y_{B}) \le Cn^{-\zeta_{1}/6}$ with probability at least $1-Cn^{-1}$ in view that (C3-ND) implies Condition (C3-D) and $n_{1} \le n$. Step 2 in Theorem \ref{thm:bootstrap_validity} needs modifications. In particular, we shall show in this step that $\rho_{\mid \calD_{n}}^{\calR} (U_{n,A}^{\sharp},Y_{A}) \le Cn^{-(\zeta_{1} \wedge \zeta_{2})/6}$ holds with probability at least $1-Cn^{-1}$ under (C2') and (\ref{eqn:growth_condition_poly}). Consider the decomposition of $\hat{\Delta}_{A}$ in (\ref{eqn:decomp_hat_Delta_A}). For $\hat{\Delta}_{A,2}$, Lemma 8 in \cite{cck2015_anticoncentration} yields that 
\begin{equation*}
\begin{split}
\E[\hat{\Delta}_{A,2}] &\lesssim n_{1}^{-1} \sqrt{(\log d)\max_{1 \le j,\ell \le d} \sum_{i_{1}=1}^{n_{1}} \E[g_{j}^{2}(X_{i_{1}}) g_{\ell}^{2}(X_{i_{1}})]} + n_{1}^{-1} \sqrt{\E\left[\max_{1 \le i_{1} \le n_{1}} \max_{1 \le j \le d} g_{j}^{4}(X_{i_{1}}) \right]} \log d \\
&\lesssim n_{1}^{-1/2} D_{n}\log^{1/2} d + n_{1}^{-1+2/q} D_{n}^{2} \log{d}.
\end{split}
\end{equation*}
By Lemma E.2 in \cite{cck2017_AoP}, we have for every $t > 0$
\[
\Prob( \hat{\Delta}_{A,2} \ge 2 \E[\hat{\Delta}_{A,2}] + t ) \le \exp\left(-{n_{1} t^{2} \over 3 D_{n}^{2} }\right) +  {CD_{n}^{q} \over n_{1}^{q/2-1} t^{q/2}}.
\]
Choosing $t = \{C n_{1}^{-1} D_{n}^{2} \log{n} \}^{1/2} \bigvee \{ C n_{1}^{-1+2/q} n^{2/q} D_{n}^{2} \}$, we have
\[
\Prob \left( \hat{\Delta}_{A,2} \ge C \{ (n_{1}^{-1} D_{n}^{2} \log(dn))^{1/2} + n_{1}^{-1+2/q} n^{2/q} D_{n}^{2} \log{d}  \} \right) \le C n^{-1}.
\]
Now, using Condition (\ref{eqn:growth_condition_poly}), we deduce that 
\[
\hat{\Delta}_{A,2} \log^{2}{d} \le C \{ (n_{1}^{-1} D_{n}^{2} \log^{5}(dn))^{1/2} +n_{1}^{-1+2/q} n^{2/q} D_{n}^{2} \log^{3}{d}  \}\le C n^{-\zeta_1/2}
\]
with probability at least $1-Cn^{-1}$. The term $\hat{\Delta}_{A,3}^{2}$ can be similarly dealt with. In particular, using $\overline{\sigma}_{g}^{2} \le 1+\max_{j} P|g_{j}|^{3} \lesssim D_{n}$, we have
\begin{equation}
\label{eqn:delta_A3}
\E[\hat{\Delta}_{A,3}] \lesssim (n_{1}^{-1}D_{n} \log d)^{1/2} + n_{1}^{-1+1/q} D_{n} \log{d}.
\end{equation}
Then using Lemma 8 in \cite{cck2015_anticoncentration} and Lemma E.2 in \cite{cck2017_AoP}, we can show that 
\[
\hat{\Delta}_{A,3}^{2} \log^{2}{d} \le C \{ n_{1}^{-1} D_{n} \log^{3}(dn) + n_{1}^{-2+2/q} n^{2/q} D_{n}^{2} \log^{4}{d} \} \le C n^{-\zeta_1} \le C n^{-\zeta_1/2}
\]
with probability at least $1-Cn^{-1}$. Hence $\hat{\Delta}_{A} \log^{2}{d} \le C n^{-(\zeta_1 \wedge \zeta_2)/2}$ with probability at least $1-Cn^{-1}$, which leads to the conclusion of this step.
\end{proof}

\subsection{Proofs of Proposition \ref{prop:bootstrap_validity_DC} and \ref{prop:bootstrap_validity_incomplete_ustat}}
For the notational convenience, let $H=\max_{1 \le j \le d} |h_{j}|$. For each fixed $x \in S$, denote by $\delta_{x} h$ the function of $(r-1)$ variables, $(\delta_{x}h)(x_{2},\dots,x_{r}) = h(x,x_{2},\dots,x_{r})$.

\begin{proof}[Proof of Proposition \ref{prop:bootstrap_validity_DC}]
In this proof, the notation $\lesssim$ signifies that the left hand side is bounded by the right hand side up to a constant that depends only on $r, \nu$, and $C_1$. We will bound $\E[ \hat{\Delta}_{A,1}^{\nu}]$. To this end, we begin with observing that 
\[
\hat{\Delta}_{A,1} \le \frac{1}{n_{1}} \sum_{i_{1} \in S_{1}} \max_{1 \le j \le d} \{ \hat{g}_{j}^{(i_{1})}(X_{i_{1}}) - g_{j}(X_{i_{1}}) \}^{2},
\]
and by Jensen's inequality, we have 
\[
\E[ \hat{\Delta}_{A,1}^{\nu} ] \le \frac{1}{n_{1}} \sum_{i_{1} \in S_{1}} \E \left [ \max_{1 \le j \le d} \left | \hat{g}_{j}^{(i_{1})}(X_{i_{1}}) - g_{j}(X_{i_{1}}) \right |^{2\nu} \right ].
\]

For each $i_{1} \in S_{1}$ and $k=1,\dots,K$, let 
\[
g^{(i_{1},k)}(x) = \frac{1}{|I_{L,r-1}|} \sum_{\substack{i_{2},\dots,i_{r} \in S_{2,k}^{(i_{1})} \\ i_{2} < \cdots < i_{r}}} (\delta_{x}h)(X_{i_{2}},\dots,X_{i_{r}}),
\]
which is the $U$-statistic with kernel $\delta_{x}h$ for the sample $\{ X_{i} : i \in S_{2,k}^{(i_{1})} \}$. Recall that the size of each block $S_{2,k}^{(i_{1})}$ is $L$, $|S_{2,k}^{(i_{1})} | = L$. Then, $\hat{g}^{(i_{1})}(x)$ is the average of $ g^{(i_{1},k)}(x), k=1,\dots,K$, 
\[
\hat{g}^{(i_{1})}(x) = \frac{1}{K} \sum_{k=1}^{K} g^{(i_{1},k)}(x).
\]
For each $i_{1} \in S_{1}$, since the blocks $S_{2,k}^{(i_{1})}, k=1,\dots,K$ are disjoint and do not contain $i_{1}$, the vectors $g^{(i_{1},k)}(X_{i_{1}}), k=1,\dots,K$ are independent with mean $g(X_{i_{1}})$ conditionally on $X_{i_{1}}$. Hence, applying first the Hoffmann-J{\o}rgensen inequality \cite[Proposition A.1.6]{vandervaartwellner1996} conditionally on $X_{i_{1}}$, we have 
\[
\begin{split}
&\E_{\mid X_{i_{1}}} \left [  \max_{1 \le j \le d} \left | \hat{g}_{j}^{(i_{1})}(X_{i_{1}})  - g_{j}(X_{i_{1}}) \right |^{2\nu} \right ]  \\
&\quad \lesssim \left (\E_{\mid X_{i_{1}}} \left [  \max_{1 \le j \le d}\left | \hat{g}_{j}^{(i_{1})}(X_{i_{1}})  - g_{j}(X_{i_{1}}) \right|  \right ] \right )^{2\nu}  + K^{-2\nu} \E_{\mid X_{i_{1}}}\left [ \max_{1 \le k \le K} \max_{1 \le j \le d} |g_{j}^{(i_{1},k)}(X_{i_{1}})-g_{j}(X_{i_{1}})|^{2\nu} \right].
\end{split}
\]
Applying Hoeffiding's averaging and Theorem 2.14.1 in \cite{vandervaartwellner1996}, we have
\[
\E_{\mid X_{i_{1}}}\left [ \max_{1 \le j \le d} |g_{j}^{(i_{1},k)}(X_{i_{1}}) - g_{j}(X_{i_{1}})|^{2\nu} \right ] \lesssim L^{-\nu} (\log d)^{\nu} P^{r-1} |\delta_{X_{i_{1}}}H|^{2\nu}. 
\]
Further, applying Lemma 8 in \cite{cck2015_anticoncentration} conditionally on $X_{i_{1}}$, we have  
\[
\begin{split}
&\E_{\mid X_{i_{1}}} \left [  \max_{1 \le j \le d} \left| \hat{g}_{j}^{(i_{1})}(X_{i_{1}})  - g_{j}(X_{i_{1}}) \right|  \right ] \\
&\lesssim  K^{-1} \sqrt{(\log d) \max_{1 \le j \le d}\sum_{k=1}^{K}\E_{\mid X_{i_{1}}}\left [ \left \{\hat{g}_{j}^{(i_{1},k)}(X_{i_{1}})  - g_{j}(X_{i_{1}}) \right \}^{2} \right]}\\
&\quad  + K^{-1} \sqrt{\E_{\mid X_{i_{1}}} \left [\max_{1 \le k \le K} \max_{1 \le j \le d} |g_{j}^{(i_{1},k)}(X_{i_{1}}) - g_{j}(X_{i_{1}})|^{2} \right]} \log d.
\end{split}
\] 
From the variance formula for $U$-statistics (cf. \cite{lee1990}, Theorem 3), we have
\[
\begin{split}
\E_{\mid X_{i_{1}}}\left [ \left \{\hat{g}_{j}^{(i_{1},k)}(X_{i_{1}})  - g_{j}(X_{i_{1}}) \right \}^{2} \right] &\le \binom{L}{r-1}^{-1} \sum_{\ell=1}^{r-1} \binom{r-1}{\ell} \binom{L-r+1}{r-1-\ell} P^{r-1} (\delta_{X_{i_{1}}}h_{j})^{2}\\
&\lesssim L^{-1} P^{r-1} (\delta_{X_{i_{1}}}h_{j})^{2}. 
\end{split}
\]
It remains to bound 
\begin{equation}
\label{eqn:remainder}
\E_{\mid X_{i_{1}}} \left [\max_{1 \le k \le K} \max_{1 \le j \le d} |g_{j}^{(i_{1},k)}(X_{i_{1}}) - g_{j}(X_{i_{1}})|^{2} \right].
\end{equation}
Observe that the term (\ref{eqn:remainder}) is bounded from above by
\[
\left ( \sum_{k=1}^{K} \E_{\mid X_{i_{1}}}\left [ \max_{1 \le j \le d} |g_{j}^{(i_{1},k)}(X_{i_{1}}) - g_{j}(X_{i_{1}})|^{2\nu} \right ] \right )^{1/\nu},
\]
which  is bounded from above by $K^{1/\nu} L^{-1} (\log d) (P^{r-1} |\delta_{X_{i_{1}}}H|^{2\nu})^{1/\nu}$
up to a constant that depends only on $r$ and $\nu$. 

By Fubini and Jensen's inequality, we have 
\[
\begin{split}
&\E \left [ \max_{1 \le j \le d} \left | \hat{g}_{j}^{(i_{1})}(X_{i_{1}})  - g_{j}(X_{i_{1}}) \right |^{2\nu} \right ]  \\
&\lesssim (KL)^{-\nu}  (\log d)^{\nu}  \E\left [ (P^{r-1}|\delta_{X_{i_{1}}}H|^{2})^{\nu}  \right ] + K^{-2\nu+1}L^{-\nu} (\log d)^{3\nu} \E\left[ P^{r-1} |\delta_{X_{i_{1}}}H|^{2\nu} \right ] \\
&\lesssim (KL)^{-\nu} (\log d)^{\nu} P^{r}H^{2\nu} + K^{-2\nu+1}L^{-\nu}(\log d)^{3\nu} P^{r}H^{2\nu} \\
&\lesssim (KL)^{-\nu} D_{n}^{2\nu} (\log d)^{3\nu} + K^{-2\nu+1}L^{-\nu}D_{n}^{2\nu} (\log d)^{5\nu} \\
&\le \left \{ (KL)^{-1} D_{n}^{2} (\log^{3} d) (1+K^{-1+1/\nu} \log^{2} d) \right \}^{\nu}, 
\end{split}
\]
from which we conclude that $\E[(\overline{\sigma}_{g}^{2}\hat{\Delta}_{A,1} \log^{4} d)^{\nu}] \lesssim n^{-\zeta \nu}$, 
so that by Markov's inequality,
\[ 
\Prob \left (\overline{\sigma}_{g}^{2} \hat{\Delta}_{A,1} \log^{4} d > n^{-\zeta +1/\nu} \right) \lesssim n^{-1}.
\]
In view of Theorem \ref{thm:bootstrap_validity} and Corollary \ref{cor:bootstrap_validity}, this leads to the conclusion of the proposition.
\end{proof}

\begin{proof}[Proof of Proposition \ref{prop:bootstrap_validity_incomplete_ustat}]
In this proof, the notation $\lesssim$ signifies that the left hand side is bounded by the right hand side up to a constant that depends only on $r,\nu$, and $C_{1}$. Observe that 
\[
\begin{split}
\hat{g}^{(i_{1})}(X_{i_{1}})  - g(X_{i_{1}}) &= M^{-1} \sum_{\iota' \in I_{n-1,r-1}} (Z_{\iota'}'-\vartheta_{n}) (\delta_{X_{i_{1}}}h)(X_{\sigma_{i_{1}}(\iota')}) \\
&\quad + |I_{n-1,r-1}|^{-1} \sum_{\iota'  \in I_{n-1,r-1}}\{ (\delta_{X_{i_{1}}}h)(X_{\sigma_{i_{1}}(\iota')})- g(X_{i_{1}}) \}.
\end{split}
\]
Conditionally on $X_{1}^{n}$, the first term is the sum of centered independent random vectors, and hence the Hoffmann-J{\o}rgensen inequality yields that 
\[
\begin{split}
&\E_{\mid X_{1}^{n}} \left [ \max_{1 \le j \le d} \left | \sum_{\iota' \in I_{n-1,r-1}} (Z_{\iota'}'-\vartheta_{n}) (\delta_{X_{i_{1}}}h)(X_{\sigma_{i_{1}}(\iota')}) \right |^{2\nu} \right ] \\
&\lesssim \left ( \E_{\mid X_{1}^{n}} \left [ \max_{1 \le j \le d} \left | \sum_{\iota'\in I_{n-1,r-1}} (Z_{\iota'}'-\vartheta_{n}) (\delta_{X_{i_{1}}}h)(X_{\sigma_{i_{1}}(\iota')}) \right | \right ] \right )^{2\nu} + \max_{\iota' \in I_{n-1,r-1}} \max_{1 \le j \le d} | (\delta_{X_{i_{1}}}h_{j})(X_{\sigma_{i_{1}}(\iota')})|^{2\nu}. 
\end{split}
\]
By Lemma 8 in \cite{cck2015_anticoncentration}, 
\[
\begin{split}
&\E_{\mid X_{1}^{n}} \left [ \max_{1 \le j \le d} \left | \sum_{\iota'\in I_{n-1,r-1}} (Z_{\iota'}'-\vartheta_{n}) (\delta_{X_{i_{1}}}h)(X_{\sigma_{i_{1}}(\iota')}) \right | \right ]   \\
&\lesssim \sqrt{ M (\log d) \max_{1 \le j \le d} |I_{n-1,r-1}|^{-1} \sum_{\iota' \in I_{n-1,r-1}} (\delta_{X_{i_{1}}}h_{j})^{2}(X_{\sigma_{i_{1}}(\iota')})} \\
&\quad + (\log d)\max_{\iota' \in I_{n-1,r-1}} \max_{1 \le j \le d} | (\delta_{X_{i_{1}}}h_{j})(X_{\sigma_{i_{1}}(\iota')})|. 
\end{split}
\]
Observe that 
\[
\max_{\iota' \in I_{n-1,r-1}} \max_{1 \le j \le d} | (\delta_{X_{i_{1}}}h_{j})(X_{\sigma_{i_{1}}(\iota')})| \le \max_{\iota \in I_{n,r}} \max_{1 \le j \le d} |h_{j}(X_{\iota})| = \max_{\iota \in I_{n,r}} H(X_{\iota}).
\]
Hence, using Jensen's inequality, we have
\[
\begin{split}
&\E_{\mid X_{1}^{n}} \left [ \max_{1 \le j \le d} \left | \sum_{\iota' \in I_{n-1,r-1}} (Z_{\iota'}'-\vartheta_{n}) (\delta_{X_{i_{1}}}h)(X_{\sigma_{i_{1}}(\iota')}) \right |^{2\nu} \right ] \\
&\lesssim M^{\nu} (\log d)^{\nu}\max_{1 \le j \le d} |I_{n-1,r-1}|^{-1} \sum_{\iota' \in I_{n-1,r-1}} (\delta_{X_{i_{1}}}h_{j})^{2\nu}(X_{\sigma_{i_{1}}(\iota')}) + (\log d)^{2\nu} \max_{\iota \in I_{n,r}} H^{2\nu}(X_{\iota}).
\end{split}
\]
In addition, applying the Hoeffding averaging and Lemma 9 in \cite{cck2015_anticoncentration} conditionally on $X_{i_{1}}$, we have 
\[
\begin{split}
&\E_{\mid X_{i_{1}}} \left [ \max_{1 \le j \le d} |I_{n-1,r-1}|^{-1} \sum_{\iota' \in I_{n-1,r-1}} (\delta_{X_{i_{1}}}h_{j})^{2\nu}(X_{\sigma_{i_{1}}(\iota')}) \right ] \\
&\lesssim \max_{1 \le j \le d} P^{r-1} (\delta_{X_{i_{1}}} h_{j})^{2\nu}+ n^{-1} (\log d) \E_{\mid X_{i_{1}}} \left [ \max_{\iota' \in I_{n-1,r-1}} \max_{1 \le j \le d} | (\delta_{X_{i_{1}}}h_{j})(X_{\sigma_{i_{1}}(\iota')})|^{2\nu} \right ].
\end{split}
\]
Hence,
\[
\begin{split}
&\E \left [ \max_{1 \le j \le d} \left | M^{-1}\sum_{\iota'\in I_{n-1,r-1}} (Z_{\iota'}'-\vartheta_{n}) (\delta_{X_{i_{1}}}h)(X_{\sigma_{i_{1}}(\iota')}) \right |^{2\nu} \right ] \\
&\lesssim M^{-\nu} (\log d)^{\nu} \left \{ P^{r} H^{2\nu} + n^{-1}(\log d) \E\left [ \max_{\iota \in I_{n,r}} H^{2\nu}(X_{\iota}) \right ] \right\} + M^{-2\nu}(\log d)^{2\nu}\E\left [ \max_{\iota \in I_{n,r}} H^{2\nu}(X_{\iota}) \right ] \\
&\lesssim M^{-\nu}(\log d)^{\nu} \{ D_{n}^{2\nu} (\log d)^{2\nu} + n^{-1}D_{n}^{2\nu} (\log (dn))^{2\nu+1}  \} + M^{-2\nu}D_{n}^{2\nu}( \log (dn) )^{4\nu} \\
&\lesssim M^{-\nu} D_{n}^{2\nu} (\log (dn))^{3\nu} 
\end{split}
\]
where we have used the fact that $(n \wedge M)^{-1} \log (dn) \lesssim 1$. 

On the other hand, applying Hoeffding averaging and Theorem 2.14.1 in \cite{vandervaartwellner1996} conditionally on $X_{i_{1}}$, we have 
\[
\begin{split}
&\E_{\mid X_{i_{1}}} \left [ \left | \max_{1 \le j \le d} |I_{n-1,r-1}|^{-1} \sum_{\iota'  \in I_{n-1,r-1}}\{ (\delta_{X_{i_{1}}}h)(X_{\sigma_{i_{1}}(\iota')})- g(X_{i_{1}}) \} \right |^{2\nu} \right ] \\
&\lesssim n^{-\nu} (\log d)^{\nu} P^{r-1}|\delta_{X_{i_{1}}} H|^{2\nu}. 
\end{split}
\]
The expectation of the left hand side is bounded by $\lesssim n^{-\nu} (\log d)^{\nu} P^{r}H^{2\nu} \lesssim n^{-\nu} D_{n}^{2\nu}(\log d)^{3\nu}$. 

Therefore, using Condition (\ref{eqn:growth_condition_incomplete_ustat}), we conclude that 
\[
\left( \E[ (\overline{\sigma}_{g}^{2} \hat{\Delta}_{A,1} \log^{4}d)^{\nu}] \right)^{1/\nu} \lesssim \overline{\sigma}_{g}^{2} \left \{ M^{-1}D_{n}^{2} \log^{7} (dn) +  n^{-1} D_{n}^{2} \log^{7}d \right \} \lesssim n^{-\zeta}.
\]
By Markov's inequality, we conclude that 
\[
\Prob \left ( \overline{\sigma}_{g}^{2} \hat{\Delta}_{A,1} \log^{4}d > n^{-(\zeta-1/\nu)} \right ) \lesssim n^{-1}.
\]
In view of Theorem \ref{thm:bootstrap_validity} and Corollary \ref{cor:bootstrap_validity}, this leads to the conclusion of the proposition.
\end{proof}

\subsection{Proofs of Propositions \ref{prop:bootstrap_validity_DC_poly} and \ref{prop:bootstrap_validity_incomplete_ustat_poly}}

The proofs are almost the same as those of Propositions \ref{prop:bootstrap_validity_DC} and \ref{prop:bootstrap_validity_incomplete_ustat}. The required modifications are to take $\nu = q/2$, and to bound $P^{r} H^{q}$ and $\E[ \max_{\iota \in I_{n,r}} H^{q}(X_{\iota})]$ by $D_{n}^{q}$ and $n^{r}D_{n}^{q}$, respectively. We omit the detail for brevity. \qed 

\subsection{Proofs for Section \ref{sec:normalized U-statistics}}

\begin{lem}[Variance estimation]
\label{lem:variance_estimation}
(i) Suppose that Conditions (C1), (C2), and (C3-ND) hold, and in addition suppose that Condition (\ref{eqn:growth_condition}) holds for some constants $0 < C_{1} < \infty$ and $\zeta_1,\zeta_2 \in (0,1)$. 
Then there exists a constant $C$ depending only on $\underline{\sigma},r$, and $C_{1}$ such that $\max_{1 \le j \le d} | \hat{\sigma}_{A,j}^{2}/\sigma_{A,j}^{2} - 1 | \le Cn^{-(\zeta_1 \wedge \zeta_2)/2} \log^{2} d$ with probability at least $1-Cn^{-1}$. 

(ii) Suppose that Conditions (C1), (C2), and (C3-D) hold, and in addition suppose that Condition (\ref{eqn:growth_condition_B}) holds for some constants $0 < C_{1} < \infty$ and $\zeta \in (0,1)$. Then there exists a constant $C$ depending only on $\underline{\sigma},r$, and $C_{1}$ such that $\max_{1 \le j \le d} | \hat{\sigma}_{B,j}^{2}/\sigma_{B,j}^{2} - 1 | \le Cn^{-\zeta/2}/\log^{2} d$ with probability at least $1-Cn^{-1}$.
\end{lem}

\begin{proof}[Proof of Lemma \ref{lem:variance_estimation}]
The proofs of Theorems \ref{thm:bootstrap_validity_A} and \ref{thm:bootstrap_validity} immediately imply the lemma. Recall the notation used in the proofs of Theorems \ref{thm:bootstrap_validity_A} and \ref{thm:bootstrap_validity}. Then we have $\max_{1 \le j \le d} | \hat{\sigma}_{A,j}^{2} - \sigma_{A,j}^{2} | \le r^{2} \hat{\Delta}_{A}$ and $\max_{1 \le j \le d} | \hat{\sigma}_{B,j}^{2} - \sigma_{B,j}^{2} | \le \hat{\Delta}_{B}$.
Since $\min_{1 \le j \le d} \sigma_{A,j}^{2} \ge r^{2}\underline{\sigma}^{2}$ in Case (i) and $\min_{1 \le j \le d}\sigma_{B,j}^{2} \ge \underline{\sigma}^{2}$ in Case (ii), the conclusion of the lemma follows from the bounds on $\hat{\Delta}_{B}$ and $\hat{\Delta}_{A}$ established in the proofs of Theorems \ref{thm:bootstrap_validity_A} and \ref{thm:bootstrap_validity}, respectively. 
\end{proof}

\begin{proof}[Proof of Corollary \ref{cor:normalized_U_statistics}]
We only prove Case (i) since the proof for Case (ii) is analogous. As before, we will assume that $\theta = P^{r} h = 0$. In this proof, let $C$ denote a generic constant that depends only on $\underline{\sigma},r$, and $C_{1}$; its value may change from place to place. In addition, without loss of generality, we may assume that $n^{-(\zeta_1 \wedge \zeta_2)/6} \le c_{1}$ for some sufficiently small constant $c_{1}$ depending only on $\underline{\sigma},r$ and $C_{1}$, since otherwise the conclusion of Case (i) is trivial by taking $C$ in the bounds sufficiently large (say, $C \ge 1/c_{1}$).
  We begin with noting that 
\[
\left| \frac{\hat{\sigma}_{j}^{2}}{\sigma_{j}^{2}} -1 \right | = \left | \frac{\hat{\sigma}_{A,j}^{2} - \sigma_{A,j}^{2}+\alpha_{n}(\hat{\sigma}_{B,j}^{2}-\sigma_{B,j}^{2})}{\sigma_{A,j}^{2}+\alpha_{n}\sigma_{B,j}^{2}}  \right | \le \left | \frac{\hat{\sigma}_{A,j}^{2}}{\sigma_{A,j}^{2}} -1 \right | + \left | \frac{\hat{\sigma}_{B,j}^{2}}{\sigma_{B,j}^{2}} - 1 \right |,
\]
so that by Lemma \ref{lem:variance_estimation}, we have that $\max_{1 \le j \le d}|\hat{\sigma}_{j}^{2}/\sigma_{j}^{2} -1| \le Cn^{-(\zeta_1 \wedge \zeta_2)/2}/\log^{2}d$
with probability at least $1-Cn^{-1}$. Choosing $c_{1}$ sufficiently small so that $Cn^{-(\zeta_1 \wedge \zeta_2)/2}/\log^{2}d \le 1/2$, and using the inequalities that 
$|z-1| \le |z^{2}-1|$ for $z \ge 0$ and $|z^{-1}-1| \le 2|z-1|$ for $|z-1| \le 1/2$, we have that 
\[
\max_{1 \le j \le d}|\sigma_{j}/\hat{\sigma}_{j}-1| \le Cn^{-(\zeta_1 \wedge \zeta_2)/2}/\log^{2}d
\]
with probability at least $1-Cn^{-1}$. Now, by Theorem \ref{thm:gaussian_approx_bern_random_design}, we have
\[
\sup_{R \in \calR} \left | \Prob (\sqrt{n} \Lambda^{-1/2} U_{n,N}' \in R) - \Prob (\Lambda^{-1/2}Y \in R) \right | \le  Cn^{-(\zeta_1 \wedge \zeta_2)/6}.
\]
Since $\Lambda^{-1/2} Y$ is Gaussian with mean zero and covariance matrix whose diagonal elements are $1$, by the Borell-Sudakov-Tsirel'son inequality together with the bound $\E[| \Lambda^{-1/2} Y|_{\infty}] \le C\sqrt{\log d}$, we have $\Prob (| \Lambda^{-1/2} Y|_{\infty} > C\sqrt{\log (dn)} ) \le 2n^{-1}$.
Hence, $\Prob (|\sqrt{n} \Lambda^{-1/2} U_{n,N}'|_{\infty}> C\sqrt{\log (dn)}) \le Cn^{-(\zeta_1 \wedge \zeta_2)/6}$.
Since $\frac{n^{-(\zeta_1 \wedge \zeta_2)/2}}{\log^{2}d} \times \sqrt{\log (dn)} \le \frac{C n^{-(\zeta_1 \wedge \zeta_2)/6}}{\log^{3/2}d}$,
we have 
\[
\Prob (|\sqrt{n} (\hat{\Lambda}^{-1/2}-\Lambda^{-1/2}) U_{n,N}'|_{\infty} > t_{n}) \le Cn^{-(\zeta_1 \wedge \zeta_2)/6},
\]
where $t_{n} = Cn^{-(\zeta_1 \wedge \zeta_2)/6}/\log^{3/2}d$. 

Now, for $R=\prod_{j=1}^{d}[a_{j},b_{j}], a=(a_{1},\dots,a_{d})^{T}$, and $b=(b_{1},\dots,b_{d})^{T}$, we have 
\[
\begin{split}
\Prob (\sqrt{n}\hat{\Lambda}^{-1/2} U_{n,N}' \in R)  &\le \Prob (\{ -\sqrt{n}\Lambda^{-1/2} U_{n,N}' \le -a+t_{n} \} \cap \{ \sqrt{n}\Lambda^{-1/2} U_{n,N}' \le b + t_{n} \}) \\
&\quad + \Prob (|\sqrt{n} (\hat{\Lambda}^{-1/2}-\Lambda^{-1/2}) U_{n,N}'|_{\infty} > t_{n}) \\
&\le \Prob (\{ -\Lambda^{-1/2}Y \le -a+t_{n} \} \cap \{ \Lambda^{-1/2}Y \le b+t_{n} \}) + Cn^{-(\zeta_1 \wedge \zeta_2)/8} \\
&\le \Prob (\Lambda^{-1/2} Y \in R) + Ct_{n}\sqrt{\log d} + Cn^{-(\zeta_1 \wedge \zeta_2)/6},
\end{split}
\]
where the last inequality follows from Nazarov's inequality. Since $t_{n}\sqrt{\log d} \le Cn^{-(\zeta_1 \wedge \zeta_2)/6}/\log d \le Cn^{-(\zeta_1 \wedge \zeta_2)/6}$, we conclude that $\Prob (\sqrt{n}\hat{\Lambda}^{-1/2} U_{n,N}' \in R) \le \Prob(\Lambda^{-1/2} Y \in R) +Cn^{-(\zeta_1 \wedge \zeta_2)/6}$. Likewise, we have $\Prob (\sqrt{n}\hat{\Lambda}^{-1/2} U_{n,N}' \in R) \ge \Prob(\Lambda^{-1/2} Y \in R) -Cn^{-(\zeta_1 \wedge \zeta_2)/6}$. Hence we have shown that 
\[
\sup_{R \in \calR} \left| \Prob (\sqrt{n}\hat{\Lambda}^{-1/2} U_{n,N}' \in R) - \Prob (\Lambda^{-1/2}Y \in R) \right| \le Cn^{-(\zeta_1 \wedge \zeta_2)/6}.
\]

Similarly, using Theorem \ref{thm:bootstrap_validity}, we have that 
\[
\Prob_{\mid \calD_{n}} (|(\hat{\Lambda}^{-1/2}-\Lambda^{-1/2})U_{n}^{\sharp}|_{\infty} > t_{n}) \le Cn^{-(\zeta_1 \wedge \zeta_2)/6}
\]
with probability at least $1-Cn^{-1}$. Following arguments similar to those above, we conclude that 
\[
\sup_{R \in \calR} \left | \Prob_{\mid \calD_{n}} (\hat{\Lambda}^{-1/2}U_{n}^{\sharp} \in R) - \Prob (\Lambda^{-1/2}Y \in R) \right | \le Cn^{-(\zeta_1 \wedge \zeta_2)/6}
\]
with probability at least $1-Cn^{-1}$. This completes the proof. 
\end{proof}

\section{Additional simulation results}
\label{sec:additional_simulation}

In Section \ref{subsec:simulation_bootstrap_plots}, we first examine the statistical accuracy of the bootstrap test statistic $U_{n}^{\sharp}$ in terms of the size for MB-NDG-DC (Spearman's $\rho$), MB-NDG-RS (Spearman's $\rho$), and the bootstrap test statistic $U_{n,B}^{\sharp}$ for MB-DG (Bergsma-Dassios' $t^*$). In Section \ref{subsec:simulation_partial_bootstrap}, we report the simulation results of the partial bootstrap $U_{n,A}^{\sharp}$ for Spearman's $\rho$. 

\subsection{Empirical performance of the bootstrap tests}
\label{subsec:simulation_bootstrap_plots}

Recall that, for each nominal size $\alpha \in (0,1)$, $\hat{R}(\alpha)$ is the empirical rejection probability of the null hypothesis, where the critical values are calibrated by our bootstrap methods. Figures \ref{fig:statistical_performance_bootstrap_spearman_rho_dc}, \ref{fig:statistical_performance_bootstrap_spearman_rho_random_sampling}, and \ref{fig:statistical_performance_bootstrap_bergsma_dassios} display the plots of the empirical size graphs $\{(\alpha, \hat{R}(\alpha)) : \alpha \in (0,1)\}$ for MB-NDG-DC (Spearman's $\rho$), MB-NDG-RS (Spearman's $\rho$), and MB-DG (Bergsma-Dassios' $t^*$), respectively. Clearly, the bootstrap approximations becomes more accurate as $n$ increases. Qualitatively, it is worth noting that the bootstrap approximations work quite well on the lower (left) tail, which is relevant in the testing application for small values of $\alpha$ (say $\alpha \le 0.10$). Quantitatively, the uniform errors-in-size on $\alpha \in [0.01, 0.10]$ of our bootstrap tests are summarized in Table \ref{tab:bootstrap_size_error_right_tail} in the main paper. 

\begin{figure}[t!] 
   \centering
       \includegraphics[scale=0.74]{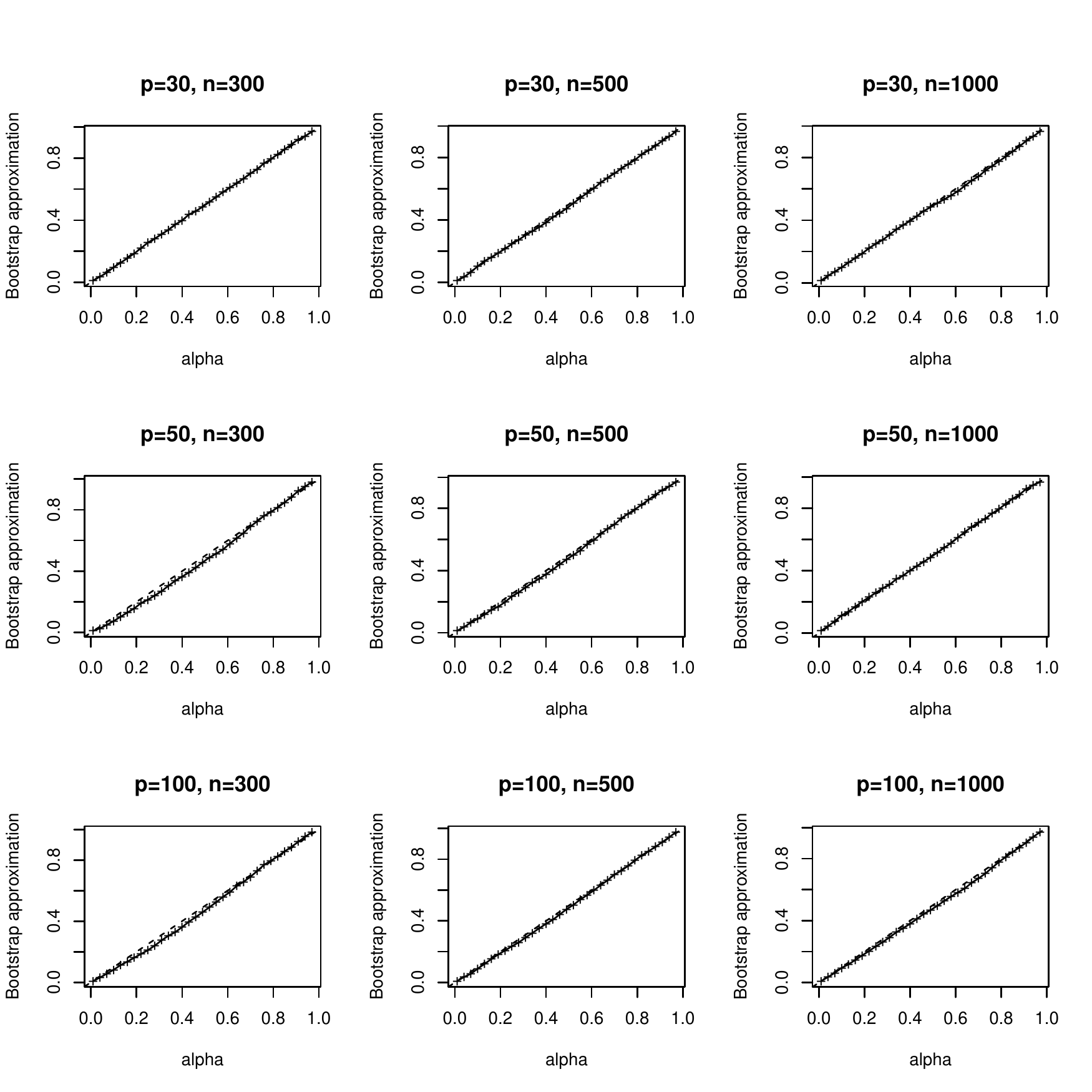}
    \caption{Bootstrap approximation $U_{n}^{\sharp}$ for Spearman's $\rho$ test statistic with the divide and conquer estimation (MB-NDG-DC). Plot of the nominal size $\alpha$ versus the empirical rejection probability $\hat{R}(\alpha)$.}
   \label{fig:statistical_performance_bootstrap_spearman_rho_dc}
\end{figure}

\begin{figure}[t!] 
   \centering
       \includegraphics[scale=0.74]{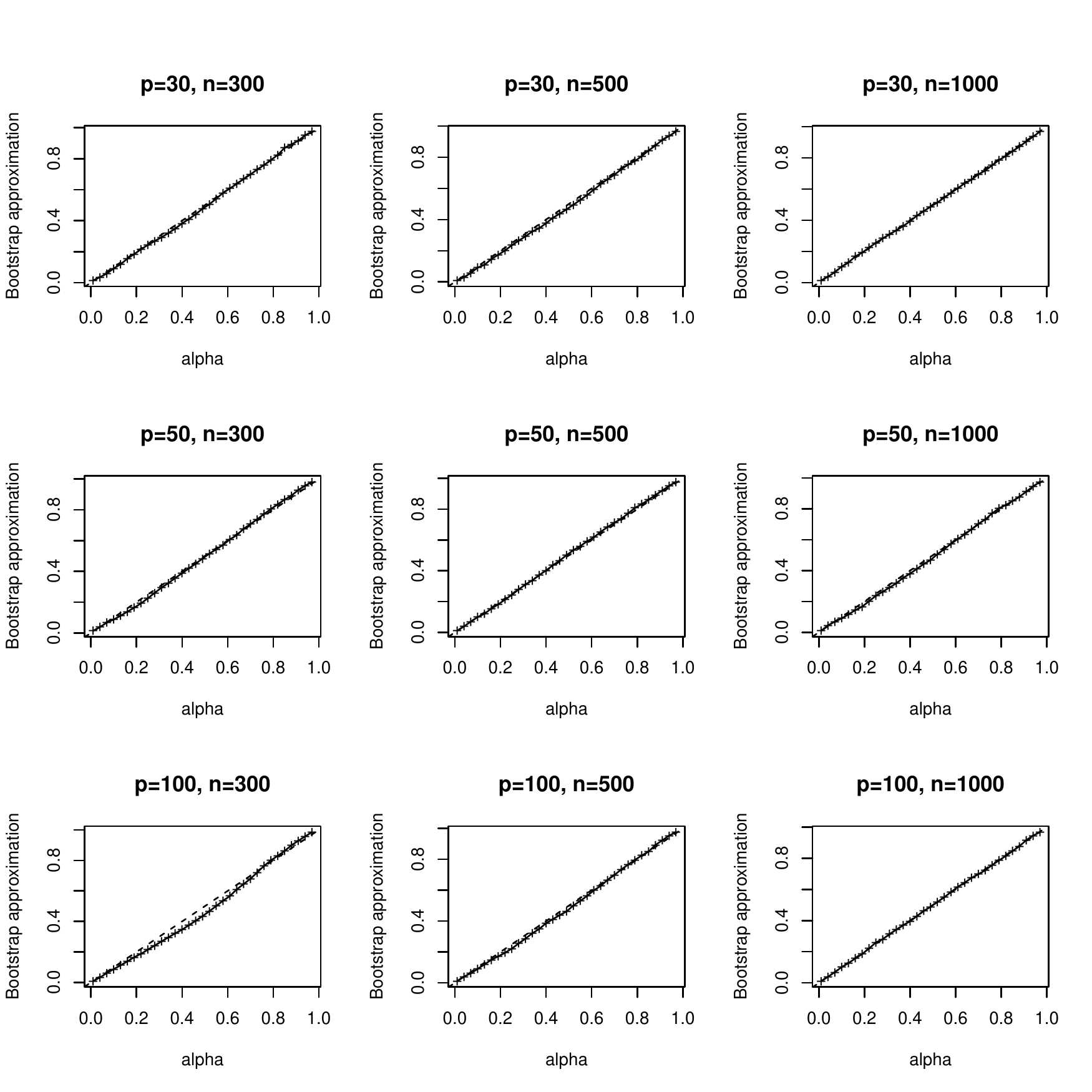}
    \caption{Bootstrap approximation $U_{n}^{\sharp}$ for Spearman's $\rho$ test statistic with the random sampling estimation (MB-NDG-RS). Plot of the nominal size $\alpha$ versus the empirical rejection probability $\hat{R}(\alpha)$.}
   \label{fig:statistical_performance_bootstrap_spearman_rho_random_sampling}
\end{figure}

\begin{figure}[t!] 
   \centering
       \includegraphics[scale=0.74]{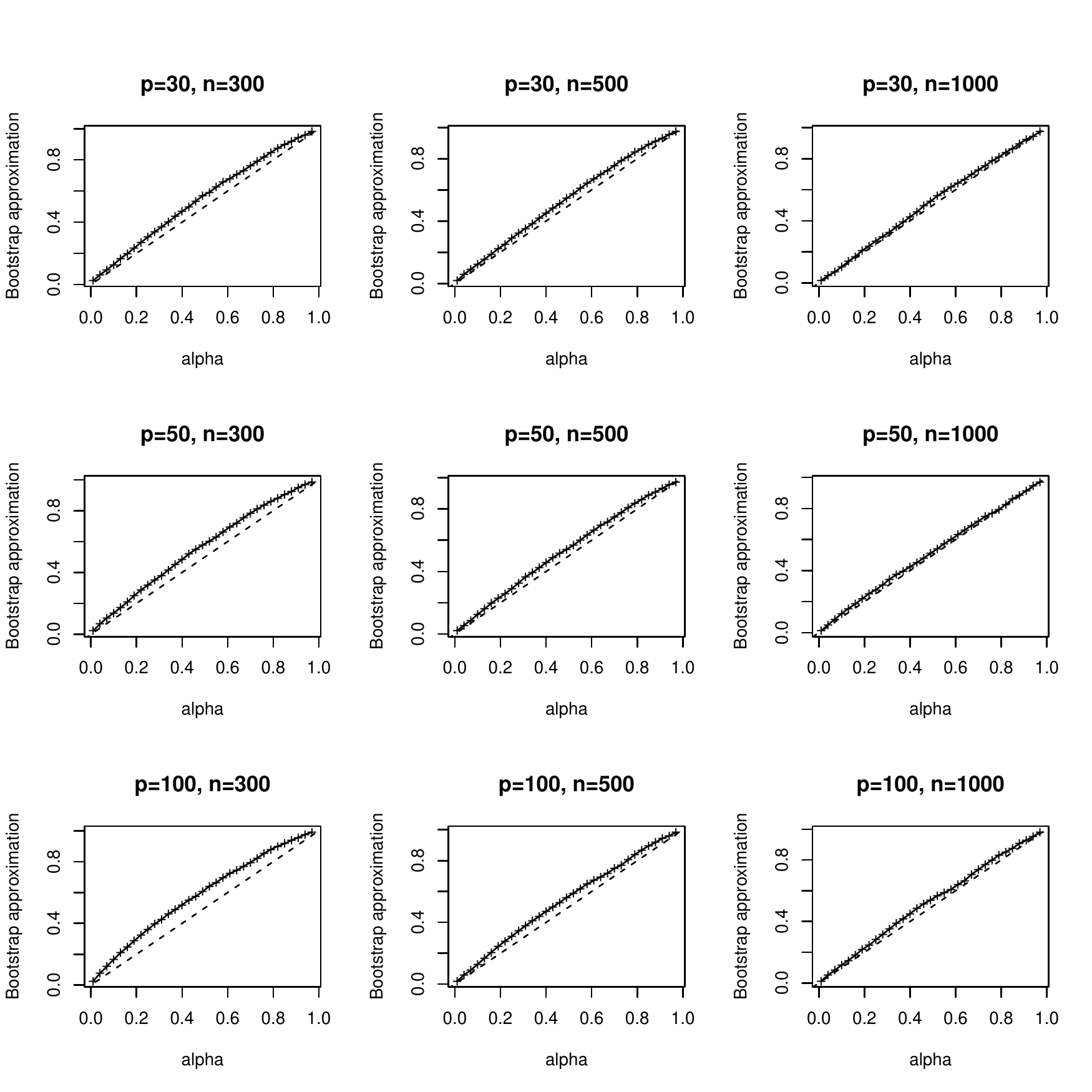}
    \caption{Bootstrap approximation $U_{n,B}^{\sharp}$ for Bergsma-Dassios' $t^*$ test statistic (MB-DG). Plot of the nominal size $\alpha$ versus the empirical rejection probability $\hat{R}(\alpha)$.}
   \label{fig:statistical_performance_bootstrap_bergsma_dassios}
\end{figure}

\subsection{Partial bootstrap $U_{n,A}^{\sharp}$}
\label{subsec:simulation_partial_bootstrap}

We also provide additional results of the partial bootstrap $U_{n,A}^\sharp$ for the non-degenerate Spearman's $\rho$ statistic. As in Section \ref{sec:numerics}, we test the performance of MB-NDG-DC and MB-NDG-RS. The computational budget parameter value is set as $N = 4 n^{3/2}$ and other parameter values remain the same as the simulation examples in Section \ref{sec:numerics}. The exponent of $n^{3/2}$ in $N$ is chosen by minimizing the rate in the error bound of the Gaussian approximation (cf. Corollary \ref{cor:Gaussian_approximation}). We empirically observe that the bootstrap approximation is sensitive to small constant values in $N = C n^{3/2}$ and we find that $C \ge 4$ can produce reasonably accurate bootstrap approximation quality (cf. Figure \ref{fig:statistical_performance_bootstrap_spearman_rho_drop_B_dc} and \ref{fig:statistical_performance_bootstrap_spearman_rho_drop_B_rs}). Table \ref{tab:partial_bootstrap_size_error_right_tail} shows the uniform errors-in-size on $\alpha \in [0.01, 0.10]$ of the partial bootstrap tests. We also show the P-P plots of the (simplified) Gaussian approximation for Spearman's $\rho$ (i.e., $\sqrt{n} U_{n,N}'$ versus $N(0, r^{2} \Gamma_{g})$) in Figure \ref{fig:gauss_approx_spearman_rho_drop_B}. 

\begin{figure}[t!] 
   \centering
       \includegraphics[scale=0.74]{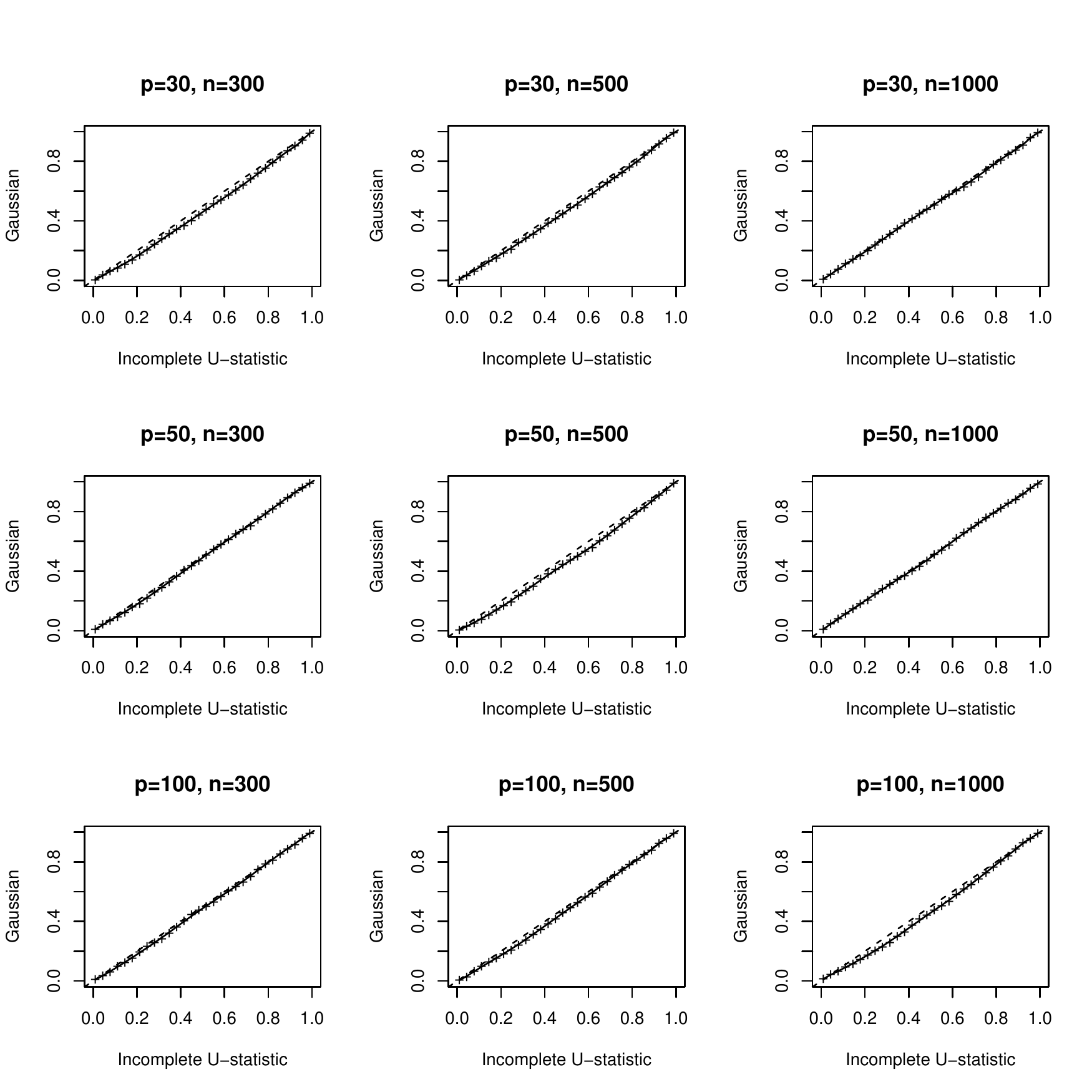}
    \caption{P-P plots for the Gaussian approximation $N(0, r^{2} \Gamma_{g})$ of $\sqrt{n} U_{n,N}'$ for Spearman's $\rho$ test statistic with the Bernoulli sampling.}
   \label{fig:gauss_approx_spearman_rho_drop_B}
\end{figure}

\begin{figure}[t!] 
   \centering
       \includegraphics[scale=0.74]{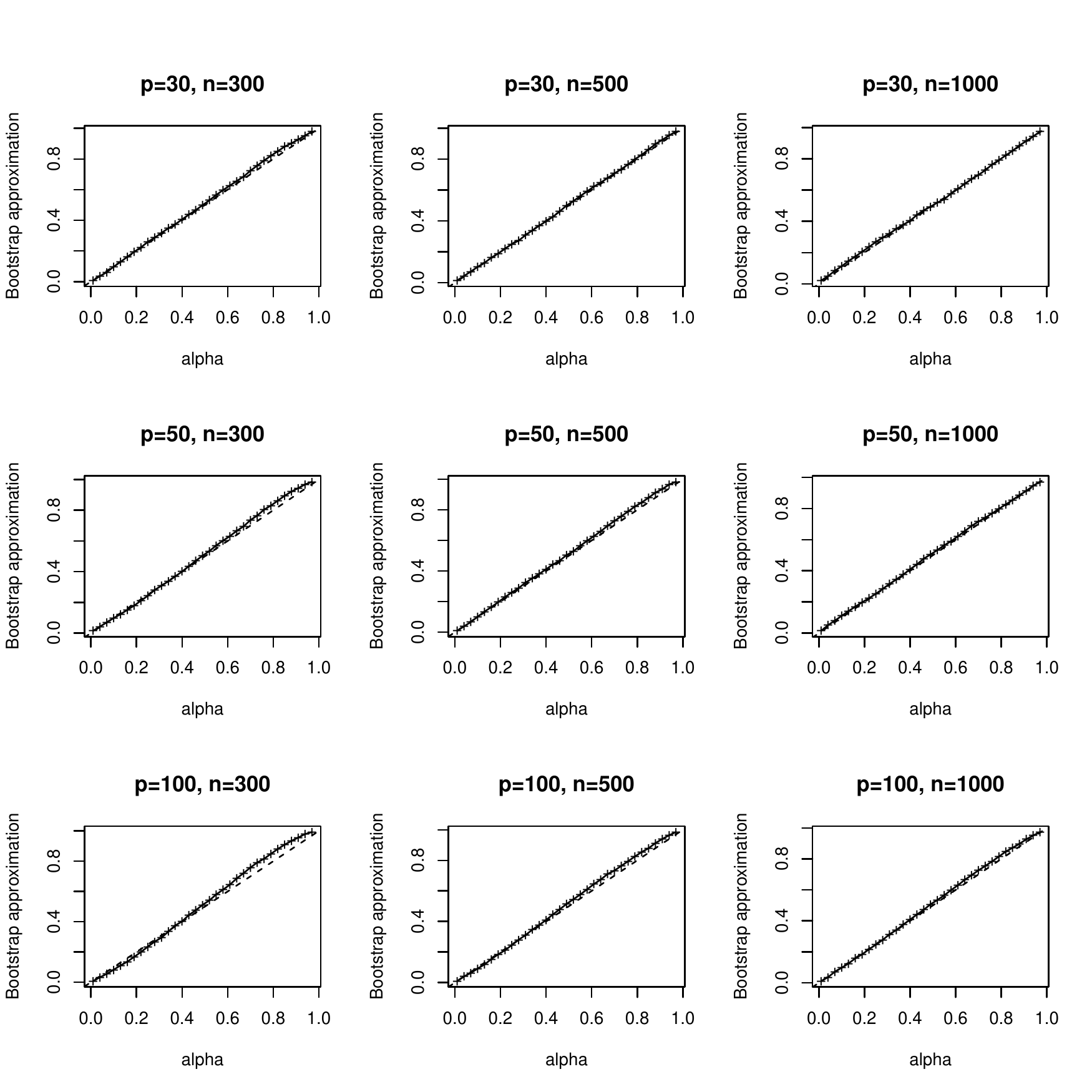}
    \caption{Bootstrap approximation $U_{n,A}^{\sharp}$ for Spearman's $\rho$ test statistic with the divide and conquer estimation (MB-NDG-DC). Plot of the nominal size $\alpha$ versus the empirical rejection probability $\hat{R}(\alpha)$.}
   \label{fig:statistical_performance_bootstrap_spearman_rho_drop_B_dc}
\end{figure}

\begin{figure}[t!] 
   \centering
       \includegraphics[scale=0.74]{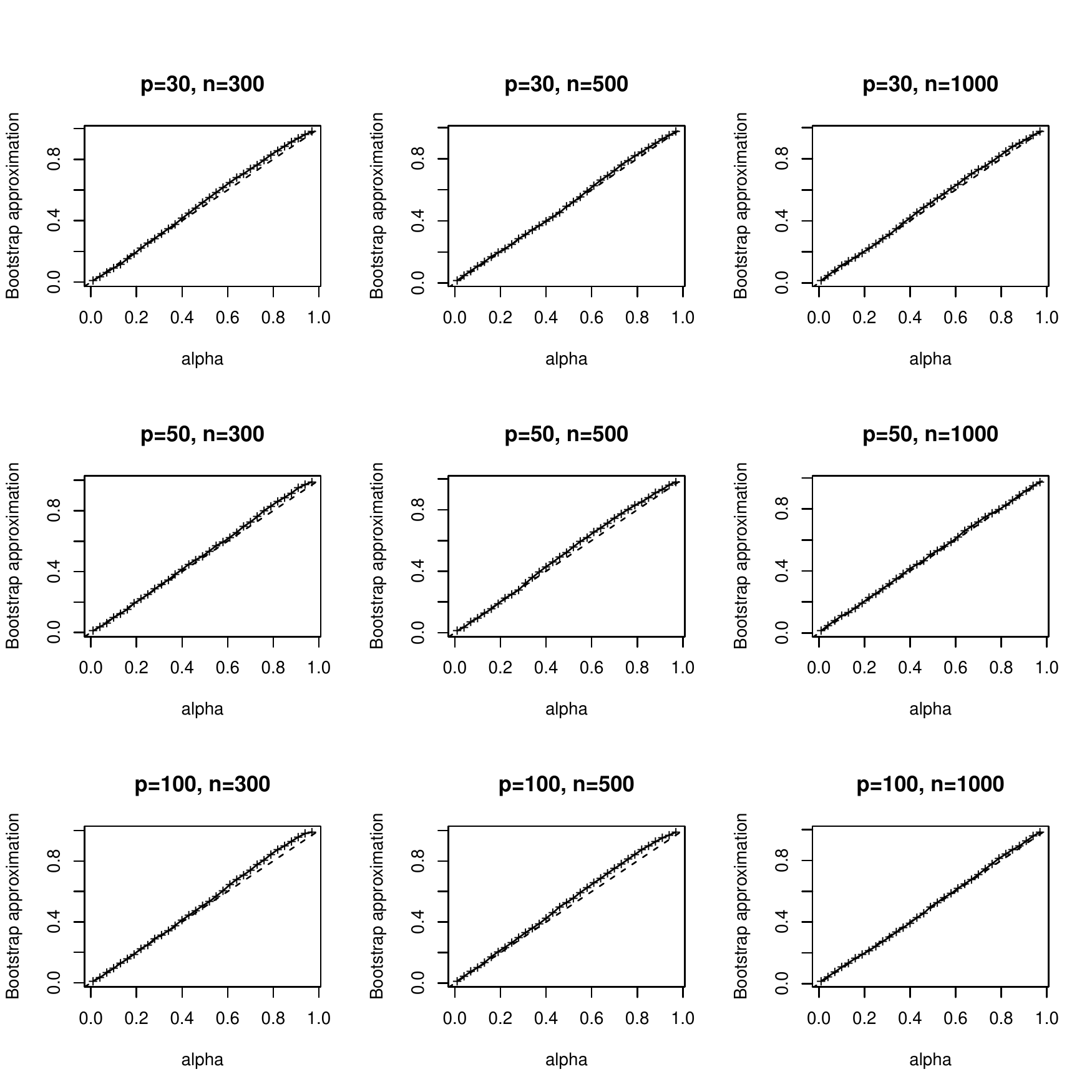}
    \caption{Bootstrap approximation $U_{n,A}^{\sharp}$ for Spearman's $\rho$ test statistic with the random sampling estimation (MB-NDG-RS). Plot of the nominal size $\alpha$ versus the empirical rejection probability $\hat{R}(\alpha)$.}
   \label{fig:statistical_performance_bootstrap_spearman_rho_drop_B_rs}
\end{figure}

\begin{table}[h]
\begin{center}
\begin{tabular}{ c | c c }
\hline
\multirow{2}{*}{Setup} & Spearman's $\rho$  & Spearman's $\rho$  \\
& (MB-NDG-DC) &  (MB-NDG-RS) \\
\hline
$p = 30, n = 300$ & 0.0085 & 0.0080 \\
$p = 30, n = 500$ & 0.0045 & 0.0095 \\
$p = 30, n = 1000$ & 0.0160 & 0.0135 \\
$p = 50, n = 300$ & 0.0070 & 0.0070  \\
$p = 50, n = 500$ & 0.0050 &  0.0045 \\
$p = 50, n = 1000$ & 0.0120 & 0.0110 \\
$p = 100, n = 300$ & 0.0175 & 0.0085 \\
$p = 100, n = 500$ & 0.0080 & 0.0065 \\
$p = 100, n = 1000$ & 0.0040 & 0.0070 \\
\hline
\end{tabular}
\end{center}
\caption{Uniform error-in-size $\sup_{\alpha \in [0.01.0.10]} |\hat{R}(\alpha)-\alpha|$ of the partial bootstrap tests, where $\alpha$ is the nominal size.}
\label{tab:partial_bootstrap_size_error_right_tail}
\end{table}

Plots of the computer running time of the partial bootstraps are shown in Figure \ref{fig:running_time_bootstrap_drop_B}. Fitting a linear model with the (log-)running time for the bootstrap methods as the response variable and the (log-)sample size as the covariate (with the intercept term), we find that the slope coefficient for $p=(30,50,100)$ is $(1.830, 1.829, 1.810)$ in the case MB-NDG-DC, and $(1.955, 1.961, 1.950)$ in the case MB-NDG-RS. In both cases, the slope coefficients again match very well to the theoretic value 2.

\begin{figure}[t!] 
   \centering
       \includegraphics[scale=0.23]{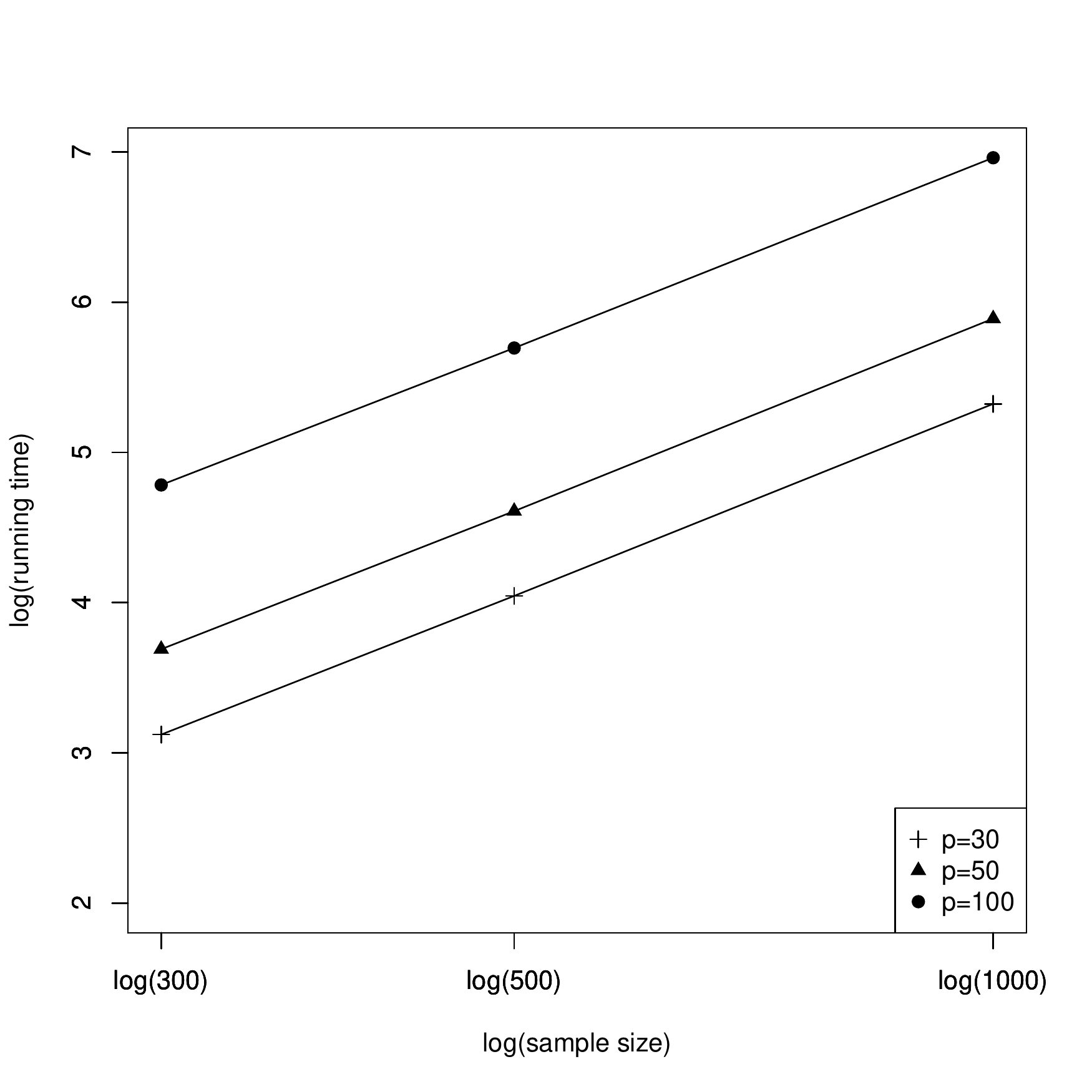}
       \includegraphics[scale=0.23]{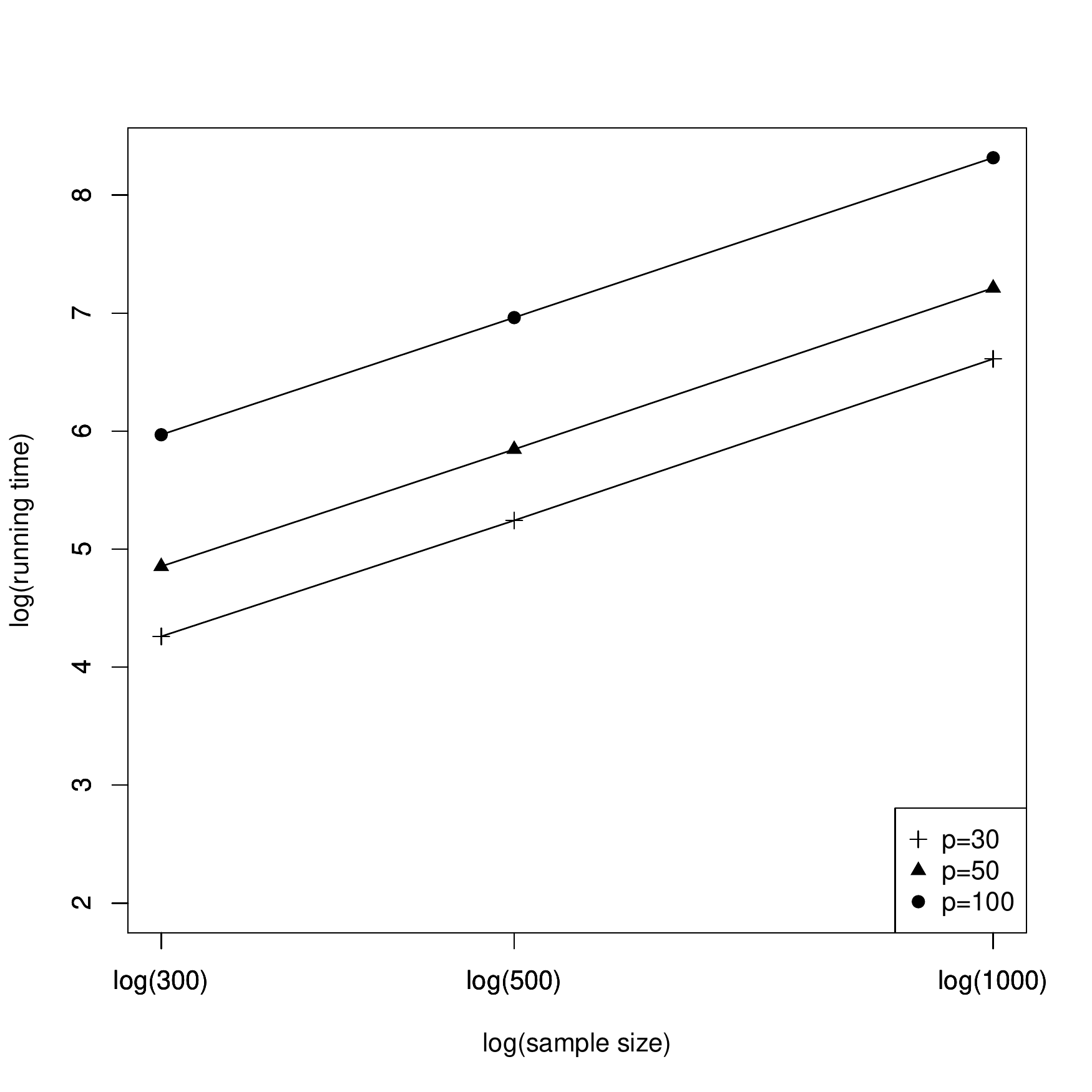}
    \caption{Computer running time of the bootstrap versus the sample size on the log-scale. Left: bootstrap $U_{n,A}^\sharp$ for Spearman's $\rho$ with the divide and conquer estimation (MB-NDG-DC). Right: bootstrap $U_{n,A}^\sharp$ for Spearman's $\rho$ with the random sampling estimation (MB-NDG-RS).}
   \label{fig:running_time_bootstrap_drop_B}
\end{figure}

\section{Random versus deterministic normalization in the Bernoulli sampling case}
\label{sec:normalization_effect}

In this section, we report the empirical effect of the random and deterministic normalization in the Bernoulli sample case on the inference of the copula dependence structures \cite[Chapter 8]{EmbrechtsLindskogMcneil2003}. Let $F$ be the joint distribution of $X = (X^{(1)}, \dots, X^{(p)})$ and $F_{j}$ be the marginal distribution of $X^{(j)}, j = 1,\dots,p$. Then it is easy to see that 
\[
\theta := \E[\hat{\rho}_{jk}] = 12 \Cov(F_{j}(X^{(j)}), F_{k}(X^{(k)})) = \text{Corr}(F_{j}(X^{(j)}), F_{k}(X^{(k)})),
\]
where $\hat{\rho}_{jk} = U_{n}^{(3)}(h^{S}_{jk}), j,k=1,\dots,p$, is defined in Example \ref{exmp:spearman_rho}. Thus Spearman's rank correlation coefficient $\rho$ is equivalent to the copula correlation, and a natural statistic to estimate the copula correlation is the $U$-statistic associated with the kernel $h^{S}$. To compare the empirical effect of normalization of the incomplete $U$-statistic in the Bernoulli sampling, we simulate $n=300$ i.i.d. copies $X_{1},\dots,X_{n}$ of $X$ from the bivariate (i.e., $p = 2$) Gaussian distribution with the following parameters: 
\[
\left(
\begin{array}{c}
X^{(1)}_{i} \\
X^{(2)}_{i} \\
\end{array} \right) \stackrel{\text{i.i.d.}}{\sim} N\left( 
\left(
\begin{array}{c}
0 \\
0 \\
\end{array} \right), \left(
\begin{array}{cc}
1 & a \\
a & 1\\
\end{array} \right) \right), \quad i =1,\dots,n,
\] 
and $a \in (-1, 1)$. In our simulation example, we choose $a = 0.9$ so that the effect of the centering term in the Gaussian approximations can be clearly seen. Histograms of $\sqrt{n}(U_{n,N}'-\theta)$ for the above Gaussian copula model under the Bernoulli sampling are shown in Figure \ref{fig:gauss_approx_copula}. Two empirical observations can be drawn from Figure \ref{fig:gauss_approx_copula}. First, the Gaussian approximation is quite accurate for $\sqrt{n}(U_{n,N}'-\theta)$ under both random and deterministic normalizations. Second, the variance of the approximating Gaussian distribution in the random normalization case is smaller than that in the deterministic case. Both observations are consistent with our theory; cf. Theorem \ref{thm:gaussian_approx_bern_random_design}, Remark \ref{rmk:random_normalization_bernoulli_sampling}, and Remark \ref{rem:deterministic_normalization} in the main paper. 

\begin{figure}[t!] 
   \centering
       \includegraphics[scale=0.45]{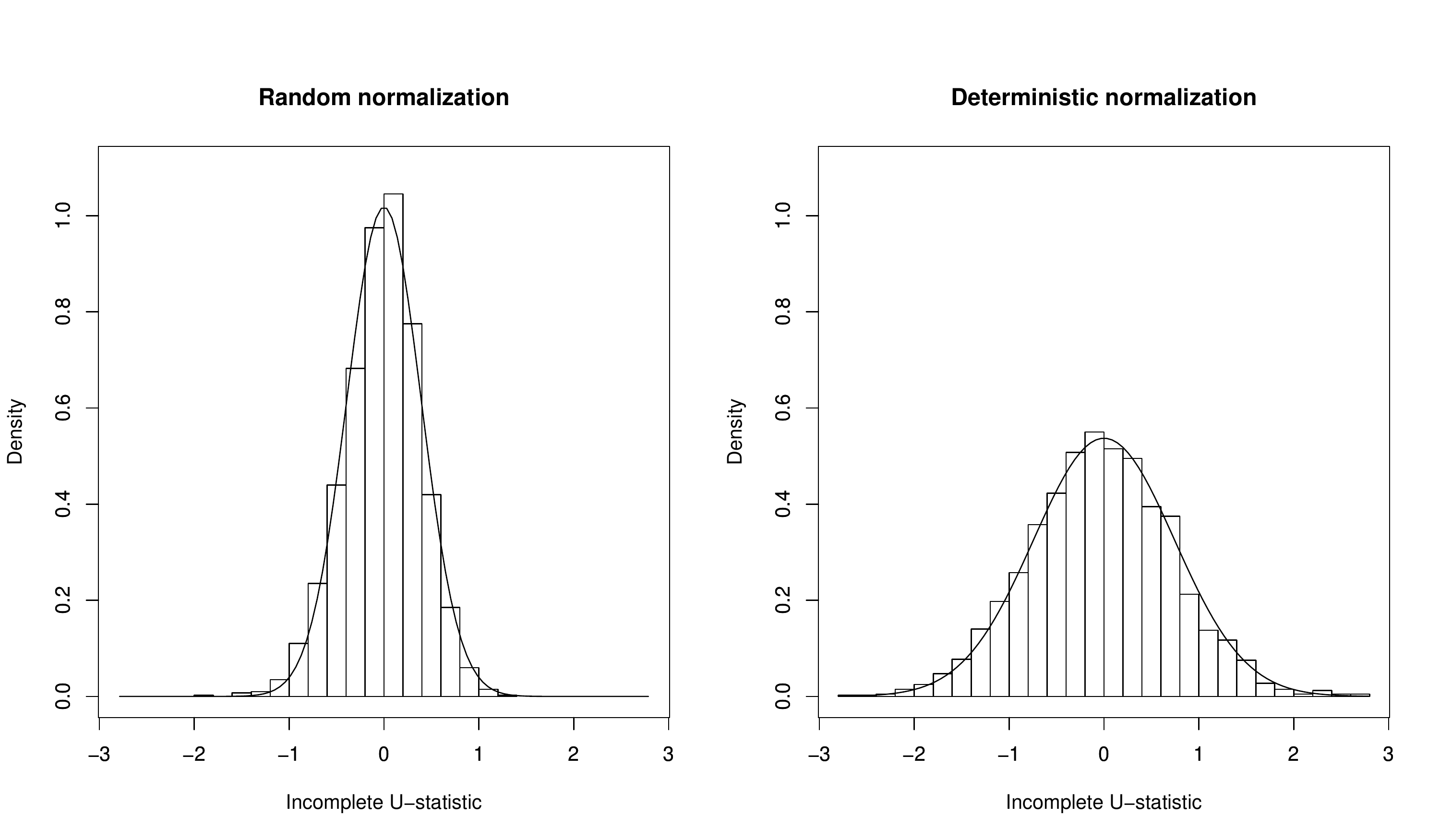}
    \caption{Histogram of $\sqrt{n}(U_{n,N}'-\theta)$ for the copula model in the Bernoulli sampling, where $\theta$ is the copula correlation. Data are generate from a bivariate Gaussian distribution with mean zero, unit variance, and covariance $0.9$. The curve is the density of the approximating Gaussian distribution for $\sqrt{n}(U_{n,N}'-\theta)$.}
   \label{fig:gauss_approx_copula}
\end{figure}


\end{document}